\documentclass{article}
\usepackage{amscd,amsfonts,euscript,amsmath,amsxtra,amssymb,amsthm}
\usepackage[english]{babel}
\newtheorem{theorem}{Theorem}
\newtheorem{lemma}{Lemma}
\newtheorem{proposition}{Proposition}
\newtheorem{definition}{Definition}
\newtheorem{claim}{Claim}
\newtheorem{remark}{Remark}

\newtheorem{example}{Example}
\newtheorem{corollary}{Corollary}

\theoremstyle{definition}

\theoremstyle{remark}

\numberwithin{equation}{section}

\def\CC{\mathbb C}

\def\HH{\mathscr H}

\usepackage[all]{xy}
\def\A{{{\mathbb A}}}
\def\I{{{\mathbb I}}}

\def\E{{{\mathbb E }}}

\def\C{{{\mathbb C }}}
\def\Z{{{\mathbb Z }}}
\def\P{{{\mathbb P }}}
\def\F{{{\mathbb F }}}

\def\L{{{\mathbb L }}}

\def\CC{{{\mathcal C}}}

\def\EE{{{\mathcal E}}}
\def\FF{{{\mathcal F}}}

\def\HH{{{\mathcal H}}}
\def\II{{{\mathcal I}}}

\def\LL{{{\mathcal L}}}

\def\OO{{{\mathcal O}}}

\def\SS{{{\mathcal S}}}
\def\TT{{{\mathcal T}}}
\def\VV{{{\mathcal V}}}

\def\EExt{{{\mathcal E}xt}}
\def\TTor{{{\mathcal T}or}}
\def\FFitt{{{\mathcal F}itt}}
\def\Hom{{{\rm Hom }}}
\def\Ext{{{\rm Ext }}}

\def\Spec{{{\rm Spec \,}}}

\def\Proj{{{\rm Proj \,}}}
\def\Hilb{{{\rm Hilb \,}}}
\def\Univ{{{\rm Univ \,}}}
\def\Grass{{{\rm Grass }}}
\def\Quot{{{\rm Quot \,}}}
\def\Supp{{{\rm Supp \,}}}

\def\dim{{{\rm dim \,}}}

\def\ker{{{\rm ker \,}}}
\def\rank{{{\rm rank \,}}}

\def\id{{{\rm id }}}

\begin{document}
{\sl MSC 14J60, 14D20, 14M27}

 {\sl UDC 512.722+512.723}
\medskip

\begin{center}
{\Large\sc On a new compactification of moduli of vector bundles
on a surface, III: Functorial approach}\end{center}
\medskip
\begin{center}
Nadezda V. TIMOFEEVA

\smallskip

Yaroslavl' State University

Sovetskaya str. 14, 150000 Yaroslavl', Russia

e-mail: {\it ntimofeeva@list.ru}
\end{center}
\bigskip

\begin{quote}
A new compactification for the scheme of moduli for
Gieseker-stable vector bundles with prescribed Hilbert polynomial,
on the smooth projective polarized surface $(S,L)$, is
constructed. We work over the field $k=\bar k$ of characteristic
zero. Families of locally free sheaves on the surface $S$ are
completed with locally free sheaves on schemes which are
modifications of $S$. Gieseker -- Maruyama moduli space has a
birational morphism onto the new moduli space. We propose the
functor for families of pairs "polarized scheme -- vector bundle"
with moduli space of such type. 
Bibliography: 16 items.
\end{quote}

\markright{On a new compactification of moduli of vector bundles,
III}

\section*{Introduction} Let $S$ be smooth irreducible
projective algebraic surface over an algebra\-ically closed field
$k$ of characteristic zero. Fix an ample invertible sheaf $L \in
Pic\, S$. We will call it for brevity as polarization of the
surface $S$. In the whole of the text of the present article
$r=\rank E$ is the rank, $p_E(m)=\chi(E\otimes L^{m})/r$ is
reduced Hilbert polynomial of the coherent sheaf $E$ on the scheme
$S$ with respect to the polarization $L$. As usually, the symbol
$\chi(\cdot)$ denotes the Euler characteristic. We work with the
notion of (semi)stability of a coherent sheaf $E$ on a surface $S$
in the sense of D.Gieseker \cite{Gies}.
\begin{definition} Coherent  $\OO_S$-sheaf
$E$ is {\it stable} (respectively, {\it semistable}), if for any
proper subsheaf  $F\subset E$ of rank $r'=\rank F$ for $m\gg 0$
$$
p_E(m)>p_F(m),\;\; ({\mbox{\rm respectively,}} \;\; p_E(m)\ge
p_F(m)\;).
$$
\end{definition}

It is well-known \cite{EG}, if the structure of the space of
moduli for semistable sheaves depends strongly on the choice of
polarization. Analogously, for a sheaf $\widetilde E$ of rank $r$
on a projective scheme  $\widetilde S$ with polarization
$\widetilde L$ we have  the notation $p_{\widetilde
E}(m)=\chi(\widetilde E\otimes \widetilde L^{m})/r$. The Gieseker
-- Maruyama moduli scheme for semistable torsion-free sheaves on
the surface $S$, with Hilbert polynomial $rp_E(m)$ with respect to
$L$, is denoted by the symbol $\overline M.$ It is well-known that
this is a projective scheme of finite type over $k$. Points
corresponding to the stable locally free sheaves (vector bundles),
form Zariski-open subscheme  $M_0$ in $\overline M$. Let the
scheme  $\overline M$ be a fine moduli space. Then there is a
trivial product $\Sigma :=\overline M\times
S\stackrel{\pi}{\longrightarrow}\overline M$ with a universal
family of stable sheaves $\E$. In \cite{Tim0,Tim1} the projective
scheme $\widetilde M$ and non-trivial flat family of schemes
$\widetilde \Sigma \stackrel{\pi}{\longrightarrow}\widetilde M$
are constructed. The family $\widetilde \Sigma $ is supplied with
the family of locally free sheaves $\widetilde \E$. Also the
birational morphism of schemes $\Phi:\widetilde \Sigma \to \Sigma
$ such that $(\Phi_{\ast}\widetilde \E)^{\vee \vee}=\E$, is
defined. In \cite{Tim2} the analogous constructions (flat families
of schemes $\widetilde \Sigma_i \stackrel{\widetilde \pi_i}
{\longrightarrow}\widetilde B_i$ with locally free sheaves
$\widetilde \E_i$) are performed over \'{e}tale neighborhoods
$\widetilde B_i \stackrel{\acute{e}tale}{\longrightarrow
}\widetilde M$. This is done for the case when $\overline M$
carries no universal family of sheaves. In any case the scheme
$\widetilde M$ contains Zariski-open subscheme which is isomorphic
to $M_0$.  We will call this construction as a {\it standard
resolution}. To perform the procedure of standard resolution one
needs a trivial family $\Sigma=T\times S$ with reduced base $T$
and a $T$-flat family $\E$ of torsion-free coherent sheaves.
Generally, for standard resolution they need not to be semistable.
The $\OO_{T\times S}$-sheaf $\E$ must be of homological dimension
1. This condition is guaranteed by fibrewise torsion-freeness
\cite[proof of proposition 4.3]{O'Gr}.

Besides, the birational morphism $\Phi: \widetilde \Sigma \to
\Sigma$ done in \cite{Tim0, Tim1}, establishes a correspondence
$(\widetilde S, \widetilde E) \mapsto (S,E)$ among pairs
$(\widetilde S, \widetilde E)\in \widetilde M$ and $(S,E)\in
\overline M.$ The birational morphism $\phi: \widetilde M \to
\overline M$ is also constructed there. Then when working
pointwise (fibrewise) we say that a fibre $\pi^{-1}(y)=S$ of the
family $\Sigma$ is an image of the fibre $\widetilde
\pi^{-1}(\widetilde y)=\widetilde S$. The coherent sheaf
 $E$ on the fibre  $S$ is an image of vector bundle
$\widetilde E$ on the fibre $\widetilde S$. Scheme-theoretic
description of surfaces $\widetilde S$ arising as fibres of flat
families  $\widetilde \Sigma$ is given in \cite{Tim3}.

Also note for the further consideration that all the manipulations
and reasoning done in \cite{Tim1} with the universal family of
stable sheaves $\E$, hold for any its twist $\E \otimes \L^{m}$ by
fibrewise ample invertible sheaf $\L$. The same is true
\cite{Tim2} for any flat family of (semi)stable sheaves
parametrized by a smooth quasiprojective algebraic scheme.

In the present paper we work under assumption that all irreducible
components of the Gieseker -- Maruyama moduli scheme contain
locally free sheaves. This holds asymptotically \cite[Theorem
0.3]{Gies-Li}, \cite[Theorem D]{O'Gr} if the discriminant of
sheaves is very big: $\Delta:= c_2-((r-1)/2r)c_1^2\gg 0$. In the
most general case this is not true.

The purpose of the present article is to develop a functorial
approach to the compactification of the moduli scheme for stable
vector bundles on the surface, which was built up by the author in
\cite{Tim0, Tim1, Tim2} using blowups of Fitting ideals. We will
refer to the compactification constructed in these articles as to
{\it constructive} compactification and will denote it as
$\widetilde M^c$ and its morphism onto Gieseker -- \linebreak
Maruyama scheme as $\phi^c: \widetilde M^c \to \overline M$. The
constructive compactification was built up using additional
blowups (smooth resolutions, partial resolutions due to Kirwan as
well as flattening birational transformation at the formation of
image in the Hilbert scheme). It is clear that to supply the
constructive compactification with a transparent geometrical
meaning as moduli space seems not to be possible. Although we will
construct a birational morphism of the constructive
compactificat\-ion onto the projective moduli scheme for pairs
polarized projective scheme -- vector bundle.

Following \cite[ch. 2, sect. 2.2]{HL} we recall some definitions.
Let $\CC$ be a category, $\CC^o$ its opposite,
$\CC'={\FF}unct(\CC^o, Sets)$ -- a category of functors to the
category of sets. By Yoneda lemma, the functor $\CC \to \CC':
F\mapsto (\underline F: X\mapsto \Hom_{\CC}(X, F))$ includes $\CC$
as a full subcategory in  $\CC'$.

\begin{definition}\cite[ch. 2, definition 2.2.1]{HL}
The functor ${\mathfrak f} \in {\OO}b\, \CC'$ is {\it
corepres\-ent\-ed by the object} $F \in {\OO}b \,\CC$, if there
exist
 $\CC'$-morphism $\psi : {\mathfrak f} \to
\underline F$ such that any morphism $\psi': {\mathfrak f} \to
\underline F'$ factors through the unique morphism  $\omega:
\underline F \to \underline F'$.
\end{definition}

Let $T$ be a scheme over the field $k$.  Consider families of
semistable pairs
\begin{equation}{\mathfrak F}_T= \left\{
\begin{array}{l}\pi: \F \to T,  \;\widetilde \L\in Pic \F ,
\;\forall t\in T \;\widetilde L_t=\widetilde \L|_{\pi^{-1}(t)}\mbox{\rm \; is ample;}\\
(\pi^{-1}(t),\widetilde L_t) \mbox{\rm \;admissible scheme
with distinguished}\\ \mbox{\rm polarisation}; \\
 \widetilde \E - \mbox{\rm locally free } \OO_{\F}-\mbox{\rm
 sheaf};\\
 \chi(\widetilde \E\otimes\widetilde \L^{m})|_{\pi^{-1}(t)})=
 rp_E(m);\\
 ((\pi^{-1}(t), \widetilde L_t), \widetilde \E|_{\pi^{-1}(t)}) - \mbox{\rm (semi)stable pair}
 \end{array} \right\}. \nonumber\end{equation}
 and a functor ${\mathfrak f}: (Schemes_k) \to (Sets)$ from the
 category of $k$-schemes to the category of sets. This functor
 assigns to any scheme $T$ the set of equivalence classes $({\mathfrak F}_T/\sim).$

 The equivalence relation $\sim$ is defined as follows. The families $((\pi: \F \to T, \widetilde \L),
 \widetilde \E)$ and $((\pi': \F' \to T, \widetilde \L'), \widetilde
 \E')$ of the class $\mathfrak F$ are said to be equivalent (notation: \linebreak $((\pi: \F \to T, \widetilde \L),
 \widetilde \E) \sim ((\pi': \F' \to T, \widetilde \L'), \widetilde
 \E')$) if\\
 1) there is an isomorphism $\F \stackrel{\sim}{\longrightarrow}
 \F'$ such that the diagram \begin{equation}\xymatrix{\F \ar[rd]_{\pi}\ar[rr]^{\sim}&&\F' \ar[ld]^{\pi'}\\
&T }
 \end{equation} commutes.\\
 2) There is a linear bundle $L$ on $T$ such that $\widetilde \E' = \widetilde \E \otimes \pi^{\ast} L.$

By technical reason (the construction of standard resolution
described in \S 2 operates with reduced base scheme) we restrict
by full subcategory $(RSchemes_k)$ of reduced schemes in
$(Schemes_k)$. All the reasonings and results of this article are
done for the functor $\mathfrak f$ restricted to this subcategory.
Also we mean by $\overline M$ the reduced scheme corresponding to
Gieseker -- Maruyama moduli scheme.

 \begin{definition} The scheme $\widetilde M$ is a {\it coarse moduli space of
 the functor } $\mathfrak f$ if $\mathfrak f$ is corepresented by
 the scheme $\widetilde M$.
 \end{definition}

\begin{theorem} \label{th} The functor $\mathfrak f$
has a coarse moduli space $\widetilde M$ with the following properties:\\
(i) $\widetilde M$ is projective Noetherian algebraic scheme;\\
(ii) there is a birational morphism of the union of main
components of the Gieseker -- Maruyama scheme:
$\kappa:\overline M \to \widetilde M$;\\
(iii) there is a birational morphism of the constructive
compactification: $\phi:\widetilde M^c \to \widetilde M$;\\
(iv) there is a commutative triangle of compactifications
\begin{equation}\label{tri}\xymatrix{&\ar[ld]_{\phi^c} \widetilde M^c \ar[rd]^{\phi}\\
\overline M \ar[rr]^{\kappa}&& \widetilde M}
\end{equation}\\
(v) there is a Zariski-open subscheme $\widetilde M_0\subset
\widetilde M$ corresponding to the stable $S$-pairs, over which
morphisms in the diagram (\ref{tri}) are iso\-morphisms. Namely,
$M_0 \cong \widetilde M^c_0\cong
\widetilde M_0$;\\
(vi) there is a relation of M-equivalence defined on the class of
semistable pairs, such that pairs are represented by the same
point in $\widetilde M$ if and only if they are M-equivalent.
\end{theorem}
All the reasoning of the present paper is applicable to any
Hilbert polynomial with no relation to the value of discriminant
as well as to the number and geometry of irreducible components in
the corresponding Gieseker -- Maruyama scheme. In general
(reducible) case the theorem provides the existence of the coarse
moduli space for any maximal (under inclusion) irreducible
substack in $\coprod({\mathfrak F}_T/\sim)$ if it contains pairs
$((\pi^{-1}(t), \widetilde L_t), \widetilde \E|_{\pi^{-1}(t)})$
such that $(\pi^{-1}(t), \widetilde L_t)\cong (S,L)$. Such pairs
will be referred to as
 {\it $S$-pairs}. We mean under $\widetilde M$
the moduli space of a substack containing semistable $S$-pairs.

Section 1 comprises some results which will be of use in the
sequel. Besides, the structure of vector bundle $\widetilde
E=\widetilde \E|_{\pi^{-1}(t)}$ on the scheme $\widetilde
S=\pi^{-1}(t)$ is compute under the assumption that the bundle
$\widetilde E$ is obtained from semistable coherent sheaf by the
procedure of articles \cite{Tim0,Tim1,Tim2}. Pairs of such view
are called as
 {\it $dS$-pairs}.

In \S  2 we turn to the construction of the Gieseker -- Maruyama
scheme $\overline M$ as GIT-quotient $\overline M= Q/SL(V)$ of an
appropriate subscheme $Q$ in the Grothendieck's scheme of
quotients. Let  $\widetilde Q$ be the quasiprojective scheme
obtained from the scheme $Q$ by the procedure of papers
\cite{Tim0,Tim1,Tim2}. We construct a morphism $\mu$ of scheme
$\widetilde Q$ into the appropriate Hilbert scheme of subschemes
in the Grassmann variety $G(V,r)$. It is proven that the image
$\mu(\widetilde Q)$ is a quasiprojective $SL(V)$-invariant
subscheme in the Hilbert scheme.

In \S 3 the explicit view of distinguished polarizations
$\widetilde L$ is compute on  schemes  $\widetilde S.$

Section 4 is devoted to the study of the isomorphism
$\upsilon:H^0(\widetilde S, \widetilde E \otimes \widetilde L^m)
\stackrel{\sim}{\to} H^0(S, E\otimes L^m)$ of global sections,
induced by the procedure of resolution of singularit\-ies of
semistable sheaves. This isomorphism is used in \S\S  5,6.

In \S  5 the notion of (semi)stability on the set of pairs
polarized scheme -- vector bundle $((\widetilde S, \widetilde L),
\widetilde E)$ is introduced. Also we examine the relation of this
new notion of (semi)stability to the classical Gieseker
(semi)stability of coherent sheaves on the polarized surface
$(S,L)$.

In \S 6 the notion of M-equivalence is introduced and motivated.
Particularly it is shown that S-equivalent semistable coherent
sheaves are resolved into M-equivalent semistable pairs.

Section 7 plays an auxiliary role. It contains results concerning
with local freeness of subsheaves and quotient sheaves in the
inverse image of Jordan -- H\"{o}lder filtration for sheaves of
the form $\widetilde E=\sigma^{\ast} E/tors.$

In \S 8 we prove the boundedness of families of $dS$-pairs and
show that the sub\-scheme formed by $dS$-pairs in the Hilbert
scheme coincides with $\mu(\widetilde Q)$.

Section 9 is devoted to the investigation of the action of the
group $PGL(V)$ upon the set of points of Hilbert scheme
corresponding to  $dS$-pairs. Also we verify the condition of
Hilbert -- Mumford criterion for the GIT-(semi)stability. The
results of this section guarantee the existence, quasiprojectivity
and being Noetherian for the scheme $\widetilde M$ as
GIT-quotient. Besides, the existence of the open subscheme $M_0$
in $\widetilde M$ and its isomorphism to the open subscheme in the
Gieseker -- Maruyama scheme follows immediately from this section.

In \S 10 it is shown that there exists a birational morphism of
Gieseker -- Maruyama scheme onto the scheme
 $\widetilde M$. This proves the projectivity of the scheme $\widetilde M.$

In \S 11 we study the relation of M-equivalence of semistable
pairs to GIT-equival\-ence of the corresponding points in Hilbert
scheme.

Finally, in \S 12 we prove that the scheme $\widetilde M$
constructed is indeed the moduli space for semistable pairs.

\section{Coherent sheaves and their resolutions}
This section plays auxiliary role. It contains results from
author's previous \linebreak papers which are necessary in the
sequel. Also we deduce some corollaries of these results.

Remember some formulations and results from \cite{Tim3}. Since not
all algebraic schemes of the present paper are varieties, the
ample divisor class $\textsf{H}$ used in  \cite{Tim3} is replaced
with the correspondent ample invertible sheaf $L.$ For convenience
of computations we suppose that $L$ is very ample. If this is not
so, we replace the sheaf $L$ with its very ample tensor power.

\begin{definition}\cite{Tim3} {\rm Artinian sheaf $\varkappa$ is said to be
$(S, L, r, p_E(m))$-{\it admissible} if there is an exact
$\OO_S$-triple
\begin{equation}\label{mex} 0\to E \to E^{\vee \vee} \to \varkappa
\to 0,
\end{equation} where the coherent sheaf $E$ of rank
$r$ with Hilbert polynomial $rp_E(m)$ is semistable with respect
to the polarization $L$.}\end{definition}

The relation of  $(S, L, r, p_E(m))$-admissibility to Fitting
ideal sheaves \linebreak $\FFitt^0 \EExt^1(E,\OO_S)$ for
semistable coherent sheaves $E$ is given by the following
propos\-ition.

\begin{proposition}\cite{Tim3}\label{sing} {The class of all $\,\FFitt^0
\EExt^1(E,\OO_S)\,$ for semistable sheaves $E$ is contained in the
class of all sheaves of the view \linebreak $\FFitt^0
\EExt^2(\varkappa, \OO_S)$ for all Artinian quotient sheaves
$\varkappa$ of length $l=l(\varkappa)=h^0(S,\varkappa)< c_2$ of
the sheaf $ \bigoplus^r \OO_S.$ }
\end{proposition}

Let $I$ be the sheaf of ideals of some zero-dimensional subscheme
$Z$ on the surface $S$ and $t$ be a symbol which is transcendent
over $\OO_S(U)$ for all open $U\subset S$. Let $\OO_S[t]$ be the
sheaf of polynomial algebras over $\OO_S$ and $I[t]$ its subsheaf
of ideals defined by the correspondence $U \mapsto I(U)[t].$ The
symbol from the right stands for the polynomial ring (without
unity) over $I(U)$. Also consider the principal ideal subsheaf
$(t)\subset \OO_S[t]$ and the sum of ideal subsheaves $I[t]+(t).$

Form $s$-th power $(I[t]+(t))^s$, $s>0$, as a subsheaf of ideals
in $\OO_S[t].$ It contains $(t)^{s+1}$ as a submodule generated by
the element $t^{s+1}$. Then the quotient module
$(I[t]+(t))^s/(t)^{s+1}$ is defined. For $s=0$ set
$(I[t]+(t))^s/(t)^{s+1}:=\OO_S$. Now form a graded $\OO_S$-algebra
$\bigoplus_{s\ge 0}(I[t]+(t))^s/(t)^{s+1}.$ It is generated by the
component of degree 1, namely by the subgroup $(I[t]+(t))/(t)^2.$

The $\OO_S$-module morphism $\OO_S \to \bigoplus_{s\ge
0}(I[t]+(t))^s/(t)^{s+1}$ leads to the morphism of projective
schemes $\sigma: \Proj \bigoplus_{s\ge 0}(I[t]+(t))^s/(t)^{s+1}
\to S$. This morphism will be referred to as {\it canonical}. As
shown in \cite{Tim3}, the scheme $\Proj \bigoplus_{s\ge
0}(I[t]+(t))^s/(t)^{s+1}$ is obtained as a fibre of the composite
map $\widehat{T\times S}
\stackrel{\sigma\!\!\!\sigma}{\longrightarrow} T\times S
\stackrel{\pi}{\longrightarrow}T$ for $T= \Spec k[t]\cong \A^1_k$
and $\sigma\!\!\!\sigma$ be a morphism of blowing up of the
trivial family of surfaces $T \times S$ in the sheaf of ideals
$\II=I[t]+(t).$ This sheaf if ideals defines the subscheme $Z$ in
the fibre $\pi^{-1}(0)\cong S.$

\begin{definition} \label{admis} Polarized algebraic scheme $(\widetilde S,
\widetilde L)$ is called {\it $(S, L, r, p_E(m))$-\linebreak
admissible} if the scheme  $(\widetilde S,\widetilde L)$ satisfies
one of the following conditions

i) $(\widetilde S, \widetilde L) \cong (S,L)$,

ii) $\widetilde S \cong \Proj \bigoplus_{s\ge
0}(I[t]+(t))^s/(t^{s+1})$, where $I=\FFitt^0 \EExt^2(\varkappa,
\OO_S)$ for Artinian quotient sheaf $q: \bigoplus^r
\OO_S\twoheadrightarrow \varkappa$ of length $l(\varkappa)\le c_2$
and $\widetilde L^m = L^m \otimes (\sigma ^{-1} I \cdot
\OO_{\widetilde S})$ for some $m\gg 0.$ In this case for any $m$
when the sheaf $L^{m} \otimes (\sigma ^{-1} I \cdot
\OO_{\widetilde S})$ is very ample, the polarization $\widetilde
L^m$ is called  {\it distinguished}.
\end{definition}

In the present paper the parameters $(S, L, r, p_E(m))$ are fixed.
We will call for brevity $(S, L, r, p_E(m))$-admissible schemes
and $(S, L, r, p_E(m))$-admissible Artinian sheaves as simply
admissible schemes and admissible Artinian sheaves respectively.

By the constructive built up of the Fitting compactification in
\cite{Tim0, Tim1, Tim2} admissible schemes include into one or
finite collection of flat families immersed as locally closed
subschemes into the projective scheme (for example, into the
universal subscheme of the Hilbert scheme). Then there exists
positive integer $m_0$ such that for all $m\ge m_0$ and for all
isomorphism classes of schemes $\widetilde S$ sheaves $L^m \otimes
(\sigma ^{-1} I \cdot \OO_{\widetilde S})$ are very ample.

Now redenote $L^m$ for $L$ and $\widetilde L^m$ for $\widetilde
L$. It is shown in \cite{EG} that the class of (semi)stable
coherent sheaves is invariant under the change of the very ample
invertible sheaf by its tensor power.

Since in the present paper the parameters $(S,L,r,p_E(m))$ are
fixed then we will refer for brevity to
$(S,L,r,p_E(m))$-admissible schemes  and
$(S,L,r,p_E(m))$-admiss\-ible Artinian sheaves as to simply
admissible schemes and admissible sheaves respect\-ively.

It is clear that  admissible scheme of the view $\widetilde
S=$\linebreak $\Proj \bigoplus_{s\ge 0}(I[t]+(t))^s/(t^{s+1})$ can
be naturally represented as a union of irreducible components
$\widetilde S= \bigcup_{i\ge 0}\widetilde S_i$ where the  {\it
main} component $\widetilde S_0=\Proj \bigoplus_{s\ge 0}(I)^s$ is
the blowup of the surface $S$ in the sheaf of ideals $I$ and for
$i>0$ $\widetilde S_i$ are irreducible  {\it additional}
components $\bigcup_{i>0} \widetilde S_i$. As it is shown in
\cite{Tim3}, in this case the additional component can have a
structure of nonreduced scheme. Obviously, admissible scheme
consists of a single component $S\cong \widetilde S=\widetilde
S_0$ if and only if it is isomorphic to the initial surface $S$.

An admissible scheme $\widetilde S$ has a morphism $\sigma:
\widetilde S \to S$. In this case the restriction
$\sigma_0=\sigma|_{\widetilde S_0}: \widetilde S_0 \to S$ of the
canonical morphism  $\sigma$ onto the main component $\widetilde
S_0$ is a blowup morphism.

A coherent torsion-free sheaf $E$ is said to be {\it deformation
equivalent to a locally free sheaf} if $E$ can be include into a
flat family of $\OO_S$-sheaves $\E$ over a connected base $T$ and
restrictions of $\E$ on fibres of view $t \times S$ for $t\in T$
is general enough, are locally free.

\begin{proposition}\label{resdes} {Let the coherent sheaf
$E$ is deformation equivalent to a locally free sheaf and is an
image of vector bundle $\widetilde E$ on an admissible scheme
$\widetilde S$. Then there is a canonically defined subsheaf\,\,
$tor\!s \subset \sigma^{\ast} E$ such that $\widetilde E \cong
\sigma^{\ast} E /tors$.}
\end{proposition}
\begin{proof} Consider a flat family $\E$ of semistable
sheaves on the surface  $S$. Let $T$ be the base of the family. By
the construction developed in \cite{Tim0,Tim1} it has homological
dimension equal to 1. Fix an exact sequence
\begin{equation} \label{triple}0\to E_1 \to E_0 \to \E \to 0,
\end{equation} where the sheaves $E_0$ and
$E_1$ are locally free. Consider the morphism of blowing up
$\sigma \!\!\! \sigma: \widehat{T\times S} \to T\times S$ of the
coherent sheaf of ideals $\I=\FFitt^0 \EExt^1(\E, \OO_{T\times
S}).$ Applying the dualisation and inverse image under the
morphism $\sigma \!\!\! \sigma$ to  (\ref{triple}) we define
sheaves  $A=\ker (\sigma \!\!\! \sigma ^{\ast} E_1^{\vee} \to
\sigma \!\!\! \sigma ^{\ast} \EExt^1 (\E, \OO_{T\times S}))$ and
$\widehat \E=(\ker(\sigma \!\!\! \sigma ^{\ast}E_0^{\vee} \to
A))^{\vee}$. According to  \cite{Tim1} they are locally free.
There is a
following exact diagram \begin{equation}\label{exdia}\xymatrix{&0\\
&\tau\ar[u]&&0\\
0\ar[r]&A^{\vee}\ar[u] \ar[r]& \sigma \!\!\! \sigma^{\ast}E_0
\ar[r]&\widehat \E \ar[u] \ar[r]&0 \\
0\ar[r]&\sigma \!\!\! \sigma^{\ast}E_1 \ar[r] \ar[u]& \sigma
\!\!\! \sigma^{\ast} E_0 \ar[u]_= \ar[r]& \sigma \!\!\!
\sigma^{\ast} \E \ar[u]\ar[r]&0 \\
& 0\ar[u]&&\tau\ar[u]\\
&&&0\ar[u] }
\end{equation}
where the symbol  $\tau$ denotes the torsion
$\OO_{\widehat{T\times S}}$-sheaf\linebreak $\EExt^1(\sigma \!\!\!
\sigma^{\ast}\EExt^1(\E, \OO_{T\times S}), \OO_{\widehat{T\times
S}})$. The right vertical triple in (\ref{exdia}) leads to the
express\-ion $\widehat \E=\sigma \!\!\! \sigma^{\ast}\E/\tau$.
Consider the restriction of this equality to the fibre
$pr^{-1}(t)=\widetilde S$ of the composite map $pr:
\widehat{T\times S} \stackrel{\sigma\!\!\!\sigma}{\longrightarrow}
T\times S \stackrel{\pi}{\longrightarrow}T$ over a closed point
$t\in T.$ Let $i_t: pr^{-1}(t) \hookrightarrow \widehat{T\times
S}$ be the morphism of inclusion of the fibre. Then there are a
morphism $\sigma: \widetilde S \to S$ and isomorphisms
$i_t^{\ast}\sigma\!\!\!\sigma^{\ast}\E=\sigma^{\ast} E$ and
$i_t^{\ast}\widehat \E =\widetilde E.$ Since the sheaf $\widehat
\E$ is locally free as a sheaf of $\OO_{\widehat{T\times
S}}$-modules then the sheaf $\widetilde E$ is also locally free as
a sheaf of $\OO_{\widetilde S}$-modules.

In the case when $\widetilde S$ is a reduced scheme the
restriction of right vertical triple in (\ref{exdia}) on the fibre
$\widetilde S$ gives an isomorphism $\widetilde E = \sigma^{\ast}
E/tor\!s.$

In the general case of (possibly nonreduced) scheme $\widetilde S$
consider its decomposition into the union of irreducible
components $\widetilde S=\bigcup_{i\ge 0}\widetilde S_i$ for
$\widetilde S_0$ being its principal component. It has a structure
of a reduced scheme. Additional components $\widetilde S_i,$
$i>0,$ can be nonreduced. Let $U$ be a Zariski open subset of one
of components $\widetilde S_i, i\ge 0,$ and let
$\sigma^{\ast}E|_{\widetilde S_i}(U)$ be the corresponding group
of sections. This group is $\OO_{\widetilde S_i}(U)$-module.
Consider sections $s\in \sigma^{\ast}E|_{\widetilde S_i}(U)$ which
are annihilated by prime ideals of positive codimension in
$\OO_{\widetilde S_i}(U)$. They constitute a submodule in
$\sigma^{\ast}E|_{\widetilde S_i}(U)$. This submodule will be
referred to as $tor\!s_i(U).$ The correspondence $U \mapsto
tor\!s_i(U)$ defines the subsheaf $tor\!s_i \subset
\sigma^{\ast}E|_{\widetilde S_i}.$ Note that associated primes of
positive codimension annihilating sections $s\in
\sigma^{\ast}E|_{\widetilde S_i}(U)$, correspond to subschemes
supported in $\sigma^{-1}(\Supp \varkappa)=\bigcup_{i>0}\widetilde
S_i$. Since by the construction the scheme $\widetilde S$ is
connected then the sheaves $tor\!s_i, i\ge 0$ permit to construct
a subsheaf $tor\!s\subset \sigma^{\ast}E$. It is defined as
follows. The section $s \in \sigma^{\ast}E|_{\widetilde S_i}(U)$
satisfies $s\in tor\!s|_{\widetilde S_i}(U)$ if and only if
\begin{itemize}
\item{there is a section $y\in \OO_{\widetilde S_i}(U)$ such that $ys=0,$}
\item{at least one of the following two requirements is fulfilled: either
$u\in {\mathfrak p}$ for $\mathfrak p$ prime of positive
codimension, or there are Zariski-open subset $V\subset \widetilde
S$ and a section $s'\in \sigma^{\ast}E(V)$ such that $V\supset U,$
$s'|_U=s$ and $s'|_{V\cap \widetilde S_0}\in
tor\!s(\sigma^{\ast}E|_{\widetilde S_0})(V\cap \widetilde S_0)$.
The torsion subsheaf $tor\!s(\sigma^{\ast}E|_{\widetilde S_0})$ in
the last expression is understood in usual sense.}
\end{itemize}

The role of the subsheaf $tor\!s \subset \sigma^{\ast}E$  is
analogous to the role of torsion subsheaf in the case of reduced
and irreducible basis scheme. Since there is no confusion the
symbol $tor\!s$ is understood everywhere as described and the
subsheaf $tor\!s$ is called a torsion subsheaf.

Consider the epimorphism $\varpi: \sigma^{\ast}E
\twoheadrightarrow \widetilde E$ defined by the restriction of the
epimorphism $\sigma\!\!\!\sigma^{\ast}\E \twoheadrightarrow
\widehat \E$ to the fibre $\widetilde S$. It is clear that $tor\!s
\subset \ker \varpi$. Note that $\ker \varpi$ is annihilated by
local sections which belong to ideals of positive codimension on a
component $\widetilde S_i,$ $i>0$, or when restricted to the
component $\widetilde S_0$. Hence we conclude that $\ker \varpi
\subset tor\!s$ and $\widetilde E= \sigma^{\ast}E/tor\!s$, as
required.
\end{proof}

In \cite{Tim3} it is shown that the fibre of the composite
morphism $\widehat{T\times S}\stackrel{\sigma \!\!\!
\sigma}{\longrightarrow} T\times S
\stackrel{\pi}{\longrightarrow}T$ at the point $t \in T$ has a
structure of the scheme $\widetilde S =\Proj \bigoplus_{s\ge
0}(I[t]+(t))^s/(t^{s+1})$ for $I=\FFitt^0 \EExt^1(\E|_t,
\OO_{t\times S})$.

The behavior of vector bundles $\widetilde E$ on additional
components  $\widetilde S_i \subset \widetilde S$, $i> 0$, is
given by the following simple computation. The standard exact
triple (1.1) is taken by the functor of the inverse image
 $\sigma_i^{\ast}$ to the exact sequence
\begin{equation}\label{upsing}\dots \to
\TTor_1^{\sigma_i^{-1}\OO_S}(\sigma^{-1}\varkappa, \OO_{\widetilde
S_i})\to \sigma_i^{\ast}E \to \sigma_i^{\ast}E^{\vee \vee}\to
\sigma_i^{\ast}\varkappa \to 0.\end{equation} In an appropriate
neighborhood  $U\subset S$ of the support  $\Supp \varkappa$ the
locally free sheaf $E^{\vee \vee}|_U$ can be replaced by its local
trivialization $\OO_U^{\oplus r}$. Then the exact sequence
(\ref{upsing}) takes the view \begin{equation*}\label{addex}\dots
\to \TTor_1^{\sigma_i^{-1}\OO_S}(\sigma^{-1}\varkappa,
\OO_{\widetilde S_i})\to \sigma_i^{\ast}E \to
\sigma_i^{\ast}\OO_U^{\oplus r}\to \sigma_i^{\ast}\varkappa \to 0.
\end{equation*}
Consequently for  $\widetilde
E_i=\sigma^{\ast}E/tor\!s|_{\widetilde
S_i}=\sigma_i^{\ast}E/tor\!s_i$ we have
\begin{equation}\label{eei}\dots \to \sigma_i^{\ast}E/tor\!s_i \to
\sigma_i^{\ast}\OO_U^{\oplus r} \to \sigma_i^{\ast} \varkappa \to
0,\end{equation} where the subsheaf $tor\!s_i$ on (possibly,
nonreduced) scheme  $\widetilde S_i$ is defined as before and
 $tor\!s_i=tor\!s|_{\widetilde S_i}$.
Dots from the left hand side indicate the terms violating
exactness. These terms are not obliged to have positive
codimension in $\widetilde S_i$.
\begin{example} Let  $\varkappa= k_x$, then $\widetilde S$
consists of two reduced components: $\widetilde S_0$ is the
surface obtained by blowing up of $S$ in the reduced point $x$,
and $\widetilde S_1\cong \P^2$. The morphism  $\sigma_1$ is a
constant morphism  $\sigma_1: \P^2 \to x$. Then $\sigma_1^{\ast}
\varkappa =\sigma_1^{\ast} k_x=\OO_{\P^2}$, and easy counting of
ranks gives  $\rank \ker (\sigma_1^{\ast}E/tor\!s_1 \to
\OO_{\widetilde S_1}^{\oplus r})=1$.
\end{example}

Since the sheaf  $\varkappa$ is supported in the finite collection
of points then the morphism $\OO_U^{\oplus r} \twoheadrightarrow
\varkappa $ can be replaced by the morphism $\OO_S^{\oplus
r}\twoheadrightarrow \varkappa$.

Let  $q_0: \OO_S^{\oplus r}\twoheadrightarrow \varkappa$ be the
morphism induced by the exact triple (1.1). One has
\begin{equation}\label{ei} \widetilde E_i=\sigma_i^{\ast}\ker q_0/tor\!s.
\end{equation}

By the Proposition \ref{resdes}, for all semistable coherent
sheaves $E$ with fixed Hilbert polynomial $rp_E(m)$, all sheaves
$\widetilde E_i$ on additional components $\widetilde S_i$ can be
described by the relations  (\ref{ei}) for appropriate  $q_0\in
\coprod_{l\le c_2} \Quot^l \OO_S^{\oplus r}$. \vspace{5mm}

\section{Standard resolution and Grassmannians}

 The procedure of transformation of a
flat family $\E$ of coherent torsion-free sheaves on the surface
$S$ over (quasi)projective base $T$ into the flat family of
schemes over the base $\widetilde T$ with a locally free sheaf
$\widetilde \E$, is given and motivated in \cite{Tim0} --
\cite{Tim2}. The scheme $\widetilde T$ is birational to $T$.

To proceed further it is necessary to turn to the construction of
Gieseker -- Maruyama scheme using geometric invariant theory. Let
$E$ be a semistable coherent torsion-free sheaf with Hilbert
polynomial equal to $rp_E(t)$. Also let $H^0(S, E\otimes L^{m})=V$
be $k$-vector space of global sections and the sheaf $E\otimes
L^{m}$ is assumed to be globally generated. Consider
Grothendieck's Quot-scheme $\Quot^{rp_E(t)}(V \otimes L^{(-m)})$
parameterizing quotient sheaves of the form
\begin{equation}\label{quosh}
V \otimes L^{(-m)}\twoheadrightarrow E,
\end{equation}
with Hilbert polynomial equal to $\chi(E\otimes L^{t})= rp_E(t)$.
Families of Gieseker-semistable coherent sheaves $E$ on the
surface  $S$ with fixed Hilbert polynomial $rp_E(t)$ are bounded.
Then there exists integer $m_0$ such that for $m> m_0$ all the
sheaves $E \otimes L^{m}$ are globally generated and all vector
spaces $H^0(S, E\otimes L^m)$ are of dimension $rp_E(m).$ This
$m_0$ is common for all $E.$ Then all semistable coherent sheaves
$E$ under consideration can be interpreted as quotient sheaves of
the form (\ref{quosh}). The projective scheme $\Quot^{rp_E(t)}(V
\otimes L^{ (-m)})$ contains a quasiprojective subscheme $Q'$ of
points corresponding to Gieseker-semistable quotient sheaves $E$
in (\ref{quosh}) with an isomorphism $H^0(S, E\otimes L^{m})\cong
V$. The scheme $\overline M$ is obtained as GIT-quotient of a
subset of GIT-semistable points in the quasiprojective subscheme
\linebreak $Q' \subset \Quot^{rp_E(t)}(V \otimes L^{(-m)})$ by the
action of the group  $PGL(V).$  This action is induced by choices
of bases in the space $V$. Let $Q$ be the component or the union
of those components in $Q'$ which correspond to the components in
$\overline M$ containing locally free sheaves. The scheme
$\Quot^{rp_E(t)}(V \otimes L^{(-m)})$ is supplied with the
universal family of quotient sheaves $\E_{\Quot}$. Let
$\E_Q:=\E_{\Quot}|_Q$ be its restriction on the subscheme $Q$.

The procedures of the paper \cite{Tim2} are applicable to the pair
$(Q, \E_Q)$. In particular, set the base of the family of sheaves
is taken to be the quasiprojective scheme $Q$. Repeating the proof
of the proposition 1.2 in \cite{Tim2} and the constructions of \S
2 in \cite{Tim2} we obtain the following objects:
\begin{itemize}
\item{$\widetilde Q$ -- quasiprojective scheme,}
\item{$\phi: \widetilde Q \to Q$ -- birational projective morphism,}
\item{$\widetilde \Sigma_Q \stackrel{\pi}{\to}
\widetilde Q$ -- flat family of schemes,}
\item{$\widetilde \E_Q$ -- locally free sheaf on the scheme $\widetilde \Sigma_Q$,}
\item{$\Phi: \widetilde \Sigma_Q \to \Sigma_Q$ -- birational
projective morphism onto the product $\Sigma_Q=Q \times S$ and
$(\Phi_{\ast}\widetilde \E_Q)^{\vee \vee}=\E_Q$ (the analog of the
proposition 2.9 in \cite{Tim2}).}
\item{for the immersion $\widetilde \Delta=\widetilde Q \times S \hookrightarrow
\widetilde Q \times \Sigma_Q$ defined by the closure of the image
of the diagonal immersion $Q_0 \times S \hookrightarrow Q_0 \times
\pi^{-1}(Q_0)$, there is the following explicit description of the
morphism $\phi$ (the analog of the corollary 2.10 in
\cite{Tim2}):\begin{eqnarray*}\phi: \widetilde Q &\to& Q:
\nonumber\\
\widetilde y&\mapsto&((id_{\widetilde Q},
\Phi)_{\ast}\OO_{\widetilde Q}\boxtimes \widetilde \E_Q )^{\vee
\vee}|_{\widetilde \Delta}|_{\widetilde y \times S} \label{expl}
\end{eqnarray*}}
\end{itemize}
Everywhere in the further text the invertible sheaf $\widetilde
\L_Q\in Pic\, \widetilde \Sigma_Q$ is assumed to be very ample
relatively to the base $\widetilde Q$. Also the fibres of the
family $\widetilde \Sigma_Q$ are mentioned to have constant
Hilbert polynomial if it is compute relatively to $\widetilde
\L_Q$. Moreover, the Euler characteristic $\chi (\widetilde \E_Q
\otimes \widetilde \L_Q^{m}|_{\pi^{-1}( \widetilde q)})$ does not
depend on the choice of the point $\widetilde q \in \widetilde Q$.

We assume that the sheaf  $\widetilde \E_Q\otimes \widetilde
\L_Q^{ m}$ is globally generated on fibres of the morphism $\pi$.
Then $\pi_{\ast} ({\widetilde \E_Q\otimes \widetilde \L_Q^{ m}})$
is a locally free sheaf of rank $rp_E(m)$ on the scheme
$\widetilde Q.$ Form a Grassmannian bundle $\Grass(
\pi_{\ast}(\widetilde \E_Q\otimes \widetilde \L_Q^{m}), r)$ of
$r$-dimensional quotient spaces in the fibres of vector bundle  $
\pi_{\ast}(\widetilde \E_Q\otimes \widetilde \L_Q^{m})$. The fibre
of Grassmannian bundle $\Grass( \pi_{\ast}(\widetilde \E_Q\otimes
\widetilde \L_Q^{ m}), r)$ is isomorphic to the ordinary
Grassmannian $G(V,r)$.

Vector bundle  $\pi_{\ast} (\widetilde \E_Q\otimes \widetilde
\L_Q^{ m})$ is locally trivial. Fix any finite trivializing open
cover  $\widetilde Q =\bigcup_i U_i$ together with the
trivializing isomorphisms
$$\tau_i: \pi_{\ast} (\widetilde \E_Q\otimes \widetilde \L_Q^{
m})|_{U_i} \stackrel{\sim}{\longrightarrow} V \otimes \OO_{U_i}.$$
Then there is the induced trivialization of the Grassmannian
bundle \linebreak $\Grass(\pi_{\ast}(\widetilde \E_Q\otimes
\widetilde \L_Q^{m}), r)$:
\begin{eqnarray*}&&\Grass (\tau_i):
\Grass ( \pi_{\ast}(\widetilde \E_Q\otimes
\widetilde \L_Q^{m}), r)|_{U_i}\cong\nonumber \\
&&\Grass( \pi_{\ast}(\widetilde \E_Q\otimes \widetilde
\L_Q^{m})|_{U_i}, r) \stackrel{\sim}{\longrightarrow} \Grass (V
\otimes \OO_{U_i}, r) \cong G(V,r) \times U_i,\end{eqnarray*}
where the arrow is induced by the morphism $\tau_i^{-1}$. The
gluing data defined on the overlaps $U_{ij}=U_i \cap U_j$ by the
composite maps  $\varphi_{ij}=\Grass(\tau_i)\circ \Grass
(\tau_j)^{-1}: G(V,r)\times U_{ij}\to \Grass (\pi_{\ast}
(\widetilde \E_Q\otimes \widetilde \L_Q^{m})|_{U_{ij}},r) \to G(V,
r)\times U_{ij}$ are given by the elements of group $GL(V)$. These
elements act upon the space $V$ by linear transformations. The
corresponding action of the same elements on the Grassmannian
$G(V,r)$ factors through the action of projective group $PGL(V)$.

For any element $U_i$ of the trivializing cover of the scheme
$\widetilde Q$ and for the restriction  $\widetilde
\Sigma_i=\widetilde \Sigma_Q |_{U_i}$ of the family $\widetilde
\Sigma_Q $ there is a morphism $\widetilde \Sigma_i\to \Grass
(\pi_{\ast}(\widetilde \E_Q\otimes \widetilde
\L_Q^{m}),r)|_{U_i}$. After the restriction on the fibres of the
structure morphism of the Grassmannian for $m\gg 0$ it becomes an
immersion. The application of the trivializing isomorphism
 $\Grass(\tau_i)$ leads to the commutative diagram
\begin{equation*}\label{tograss}\xymatrix{
\widetilde \Sigma_i \ar[d]_{\pi_i} \ar[r] & G(V,r)\times U_i
\ar[ld]^{pr_2} \ar[r]^{\;\;\;\;pr_1}& G(V,r)\\
U_i}\end{equation*} The horizontal composite map restricted on any
fibre $\widetilde S_{\widetilde q}=\pi_i^{-1} (\widetilde q)$ of
the projection $\pi_i$, gives an immersion $j_{\widetilde q}:
\widetilde S_{\widetilde q} \hookrightarrow G(V,r)$ of the scheme
$\widetilde S_{\widetilde q}$ into the Grassmannian $G(V,r)$. The
image  $j_{\widetilde q}(\widetilde S_{\widetilde q})$ is a closed
subscheme in $G(V,r)$.  Hilbert polynomial of this subscheme
$\chi(j_{\widetilde q}^{\ast} \OO_{G(V,r)}(t))$ is constant for
all fibres, for all elements of cover and is uniform for all
trivializing covers. We denote it  $P(t).$

Let $\Hilb^{P(t)} G(V,r)$ be the Hilbert scheme of subschemes in
 $G(V,r)$ with Hilbert polynomial equal to $P(t)$.
For any element $U_i$ of the trivial\-izing cover there is a map
$\widetilde \mu_i: \widetilde \Sigma_i \to \Univ^{P(t)}G(V,r)$
into the universal subscheme $$\Univ^{P(t)}G(V,r) \subset
 \Hilb^{P(t)}G(V,r)\times G(V,r).$$ Also there is a mapping
$\mu_i: U_i \to \Hilb^{P(t)}G(V,r)$ of the base $U_i$ into the
Hilbert scheme. These morphisms include into the fibred square
\begin{equation*}\xymatrix{\widetilde \Sigma_i
\ar[d]_{\pi_i} \ar[r]^{\widetilde
\mu_i\;\;\;\;\;\;\;\;\;\;\;\;\;\;}
& \Univ^{P(t)} G(V,r) \ar[d]^{\pi_H}\\
U_i \ar[r]^{\mu_i\;\;\;\;\;\;\;\;\;\;\;\;\;\;}& \Hilb^{P(t)}
G(V,r)}
\end{equation*}

\begin{remark} Since  $U_i$ are quasiprojective schemes and
$\Hilb^{P(t)}G(V,r)$ is a projective scheme, then $\mu_i$ are
projective morphisms.
\end{remark}

Let $T$ be a regular scheme of dimension 1. We introduce the
notation $\Sigma=T\times S$. Let also $\E$ be a flat family of
coherent semistable torsion-free sheaves on the surface $S$ with
fixed Hilbert polynomial $rp_E(t)$. We suppose that $\E$ is
parameterized by $T$. Let $T_0 \subset T$ be nonempty open subset
such that the sheaf $\E_0=\E|_{\Sigma_0}$ is locally free on the
preimage $\pi^{-1}(T_0)=\Sigma_0\subset \Sigma$. Consider a
blowing up $\sigma\!\!\! \sigma : \widehat \Sigma \to \Sigma$ of
the sheaf of ideals $\I=\FFitt^0 \EExt^1 (\E, \OO_{\Sigma})$. The
composite map $\widehat \Sigma \stackrel{\sigma \!\!\!
\sigma}{\longrightarrow} \Sigma \stackrel{\pi}{\longrightarrow} T$
is a flat projective morphism because its fibres are projective
schemes, $\widehat \Sigma$ is an irreducible scheme and $T$ is
regular scheme of dimension 1. It is clear that $\widehat
\Sigma_0= \sigma\!\!\! \sigma^{-1} \Sigma_0 \cong \Sigma_0$. For
any $x\in T_0$ the sheaf $\E_0|_{x\times S}= E_x$ induces an
immersion $j_x: S\hookrightarrow G(V,r)$. The composite of this
immersion with a Pl\"{u}cker immersion  $Pl: G(V,
r)\hookrightarrow P(\bigwedge^r V)$ distinguishes a very ample
invertible sheaf $L_E=j_x^{\ast}\OO_{G(V,r)}(1)=
j_x^{\ast}Pl^{\ast}\OO_{P(\bigwedge^r V)}(1)$. The composite map
$j_x\circ Pl: S \hookrightarrow P(\bigwedge^r V)$ is given by the
composite morphism of sheaves
\begin{eqnarray}\bigwedge^r V \otimes \OO_S
&\stackrel{\sim}{\longrightarrow}& \bigwedge^r H^0(S, E\otimes
L^{m})\otimes
\OO_S \nonumber \\
&\twoheadrightarrow& H^0(S, \bigwedge^r (E\otimes L^{ m}))\otimes
\OO_S\twoheadrightarrow \bigwedge^r (E\otimes
L^{m}).\nonumber\end{eqnarray} In this case, the epimorphicity of
two recent maps is provided by the choice of  $m\gg 0.$ The same
reason guarantees the fact that the sheaf $\bigwedge^r (E\otimes
L^{m})$ is very ample. It follows by the construction of the
morphism $ j_x\circ Pl$ that $L_E=\bigwedge^r (E\otimes L^{m})=
L^{mr} \otimes \det E$. In the further text we will not
distinguish in the notation the first Chern class $c_1=c_1(E)$ and
its image in the Picard group of the surface $S$. For example we
write $L_E^{mr}=L^{mr}\otimes c_1.$

The subscheme $\widehat \Sigma_0=T_0 \times S$ has an immersion
into the relative projective space $P(\bigwedge^r \pi_{\ast}
(\E_0\otimes \L_0^m))\cong P(\bigwedge^r V)\times T_0$. This
immersion is given by the  sheaf $\L_E=L_E \boxtimes \OO_{T_0}$
very ample relatively $T_0$, and includes in commutative diagram
\begin{equation}\label{triangle}\xymatrix{\Sigma_0 \ar@{^(->}[r] \ar[rd]_{\pi}&
P(\bigwedge^r \pi_{\ast} (\E_0\otimes \L_0^{m}))\ar[d]_p\\
&T_0}
\end{equation}
Here and further $\L_0:=L\boxtimes \OO_{T_0}$ is an invertible
$\OO_{\Sigma_0}$-sheaf very ample relatively $T_0.$

Note that the scheme $\widehat \Sigma$ is supplied with the
locally free sheaf $\widehat \E$. It is flat over the base $T$.

Also note that for $m\gg 0$ there exists an epimorphism $V\otimes
\L_0^{(-m)} \twoheadrightarrow \E_0$. Fix it. Hence there is a
morphism $\nu_0:T_0 \to \Quot^{rp_E(t)} (V\otimes L^{(-m)})$ of
the base scheme $T_0$ into the Grothendieck Quot-scheme
$\Quot^{rp_E(t)} (V\otimes L^{(-m)})$. Since the sheaf $\E$ is
fibrewise semistable we assume without loss of generality that
 $\nu_0 (T_0)\subset Q.$ Let $\widetilde T=T\times _Q
\widetilde Q$ and  $\nu: \widetilde T \to \widetilde Q$ be the
corresponding morphism. Then for the preimage $\widetilde \Sigma=
\widehat \Sigma \times _T \widetilde T$ of the family $\widehat
\Sigma$ by the universality of the Hilbert scheme
$\Hilb^{P(t)}\Grass(\pi_{\ast} (\widetilde \E_Q\otimes \widetilde
\L_Q^{m}),r)$ one has a fibred diagram
\begin{equation*}\xymatrix{
\Univ^{P(t)}\Grass(\pi_{\ast} (\widetilde \E_Q\otimes \widetilde
\L_Q^{m}),r) \ar[d]_{\widetilde \pi}
&\ar[l]_{\;\;\;\;\;\;\;\;\;\;\;\;\;\;\;\;\;\;\;\;\;\;\;\;\;\;\;\;\;\;\;\widetilde
\mu} \widetilde \Sigma_Q \ar[d]_{\pi}&&
\ar[ll]_{\widetilde \nu} \widetilde \Sigma \ar[d]\\
\Hilb^{P(t)}\Grass(\pi_{\ast} (\widetilde \E_Q\otimes \widetilde
\L_Q^{m}),r) &
\ar[l]_{\;\;\;\;\;\;\;\;\;\;\;\;\;\;\;\;\;\;\;\;\;\;\;\;\;\;\;\;\;\mu}
\widetilde Q && \ar[ll]_{\nu} \widetilde T }
\end{equation*}

Let $\widetilde \L= \widetilde \nu^{\ast}\widetilde \L_Q.$
One-dimensional base $T$ does not undergo a birational
trans\-formation, namely, $\widetilde T \cong T.$

Since it is clear that $\L_0\cong \widetilde \L|_{\Sigma_0}$, then
the combination of the diagram (\ref{triangle}) with the open
immersion $P(\bigwedge^r \pi_{\ast} (\E_0\otimes
\L_0^{m}))\hookrightarrow P(\bigwedge^r \pi_{\ast}(\widetilde
\E\otimes \widetilde \L^{m}))$ and the formation of closure
$\overline \Sigma_0$ for the image of the scheme $\Sigma _0$ in
projective bundle  $P(\bigwedge^r \pi_{\ast} (\widetilde \E\otimes
\widetilde \L^{m})),$ lead to the commutative diagram
\begin{equation*}\xymatrix{\overline \Sigma_0 \ar[dr] \ar@{^(->}[r]
& P(\bigwedge^r  \pi_{\ast} (\widetilde \E\otimes
\widetilde \L^{m})) \ar[d]^p\\
& \widetilde T}
\end{equation*}
Here the scheme $\overline \Sigma_0$ coincides on the open subset
with the image of the scheme $\widetilde \Sigma.$ Since both
schemes are irreducible and flat over the same base $\widetilde
T,$ the image  $\widetilde \Sigma$ coincides with $\overline
\Sigma_0.$ This proves that one-dimensional flat family
$\widetilde \Sigma$ can be considered as the closure of the image
for its open subscheme $\Sigma_0$ under the immersion into the
projective bundle $P(\bigwedge^r \pi_{\ast} (\widetilde \E\otimes
\widetilde \L^{m}))$.

For the comparison of the structure of the families $\widetilde
\Sigma$ and $\overline \Sigma_0$ and polarizations on their fibres
we construct the immersions of both families into the same
relative projective space. Firstly consider the composite of the
immersion into the relative Grassmannian with the Pl\"{u}cker
imbedding, and a commutative triangle
\begin{equation*}\xymatrix{ \widetilde \Sigma_Q \ar[drr]_{ \pi}
\ar@{^(->}[r]^<<<<{j} &\Grass(\pi_{\ast} (\widetilde \E_Q\otimes
\widetilde \L_Q^{m}),r) \ar@{^(->}[r]^{Pl}& P(\bigwedge^r
\pi_{\ast} (\widetilde \E_Q\otimes
\widetilde \L_Q^{m}))\ar[d]_p\\
&&\widetilde Q}
\end{equation*}
Here  $\pi$ is flat morphism, $\widetilde \L_Q^{m}$  an invertible
sheaf very ample relatively $\widetilde Q$. Secondly, the morphism
 $\nu$ induces a fibred square
\begin{equation*}\xymatrix{
P(\bigwedge^r \pi_{\ast} (\widetilde \E\otimes \widetilde
\L^{m}))\ar[d]_p \ar[r]& P(\bigwedge^r \pi_{\ast} (\widetilde
\E_Q\otimes
\widetilde \L_Q^{m}))\ar[d]^p\\
\widetilde T \ar[r]^{\nu}& \widetilde Q}
\end{equation*}
The combination of two recent diagrams
\begin{equation}\label{redto1}\xymatrix{
&\widetilde \Sigma_Q \ar[drr]^{\pi} \ar@{^(->}[rr]^{Pl \circ j} &&
P(\bigwedge^r \pi_{\ast}(\widetilde \E_Q\otimes
\widetilde \L_Q^{m}))\ar[d]^p\\
\widetilde \Sigma \;\;\ar[ur] \ar@{^(->}[rr] \ar[drr]&&
P(\bigwedge^r \pi_{\ast} (\widetilde \E\otimes \widetilde \L^{m}))
\ar[ur] \ar[d]
&\widetilde Q\\
&&\widetilde T\ar[ur]_{\nu}}
\end{equation}
shows that to compute the distinguished polarizations on fibres of
the family $\widetilde \Sigma_Q$ one can use one-parameter
families of the form  $\widetilde \Sigma \to \widetilde T$.

\begin{proposition} There exist $PGL(V)$-invariant morphisms
$$\widetilde \mu:
\widetilde \Sigma_Q \to \Univ^{P(t)}G(V,r), \;\;\;\; \mu:
\widetilde Q \to \Hilb^{P(t)}G(V,r)$$ include into the fibred
diagram
\begin{equation}\label{glue}\xymatrix{&\widetilde \Sigma_Q
\ar[rd]^{\widetilde \mu} \ar[dd]^>>>>>>{\pi}\\
\widetilde \Sigma_i \ar[dd]_{\pi_i} \ar@{^{(}->}[ur]
\ar[rr]^{\widetilde \mu_i} && \Univ^{P(t)} G(V,r) \ar[dd]^\pi\\
&\widetilde Q \ar[rd]^\mu\\
U_i \ar@{^{(}->}[ur] \ar[rr]^{\mu_i}&& \Hilb^{P(t)} G(V,r)}
\end{equation}
\end{proposition}
\begin{proof} Refining the cover
$U_i$, if necessary, we achieve that every element $U_i$ has an
image $U'_i = \phi(U_i)$ which is open in $Q$ and
$U_i=\phi^{-1}(U'_i)$ where subschemes $U'_i$ form a cover for the
scheme $Q$. As previously, the covering by the schemes $U_i$
carries the trivialization of locally free sheaf $ \pi_{\ast}
(\widetilde \E_Q\otimes \widetilde \L_Q^{ m})$.

Schemes $Q\subset \Quot^{rp_E(t)} (V\otimes L^{(-m)})$ and
$\Hilb^{P(t)}G(V,r)$ are supplied with actions of the group
$PGL(V)$. These actions are induced by linear transformations of
vector space $V$. Let
\begin{eqnarray*} &&\alpha:
PGL(V)\times Q  \to Q \nonumber \\
&&\beta:PGL(V) \times \Hilb^{P(t)}G(V,r)\to
\Hilb^{P(t)}G(V,r)\end{eqnarray*} be corresponding morphisms. We
will denote restrictions of these morph\-isms on any subsets by
the same symbols. The abbreviation  $PGL(V)\star U$ denotes
 $\alpha(U \times PGL(V))$ or $\beta(U \times PGL(V))$
respectively. It is clear that sets of the form $PGL(V) \star U$
are  $PGL(V)$-invariant.

Although open subschemes $U'_i$ and $U_i$  a priori are not
$PGL(V)$-invariant, they generate the following fibred diagram
\begin{equation}\label{twoact}\xymatrix{PGL(V)\times U'_i
\ar[r]^{\alpha}& PGL(V)\star U'_i\\
PGL(V)\times U_i \ar[u]^{1\times \phi} \ar[d]_{1\times
\mu_i}\ar@{->} @< 2pt>[r]^{\alpha'}
\ar@{->} @<-2pt> [r]_{\beta'}&PGL(V)\star U_i \ar[u]_{\phi} \ar[d]^{\mu_i}\\
PGL(V)\times \mu_i(U_i) \ar[r]_{\beta}& PGL(V)\star \mu_i(U_i)}
\end{equation}
where $\alpha'=\beta'$ by the definition of actions $\alpha$ and
$\beta.$ Hence the scheme  $\widetilde Q$ can be expressed in the
form $\widetilde Q=\bigcup_i PGL(V) \star
\phi^{-1}(U'_i)=\bigcup_i PGL(V) \star U_i.$ By (\ref{twoact}) we
obtain morphisms
\begin{equation*}\xymatrix{\bigcup_i PGL(V)\star \mu_i(U_i)
&\ar[l]_{\;\;\;\;\mu} \bigcup_i PGL(V)\star U_i \ar[r]^{\phi}
& \bigcup_i PGL(V)\star U'_i\\
\mu(\widetilde Q) \ar@{=} [u]& \widetilde Q \ar@{=} [u]& Q\ar@{=}
[u]}
\end{equation*}
The morphism $\mu:=\bigcup_i PGL(V)\star \mu_i$ is defined by maps
 $\mu_i$ and their composites with group actions. By the
 central equality,
connected com\-ponents of the scheme $\widetilde Q$ are taken by
the map  $\bigcup_i PGL(V)\star \mu_i$ to connected sets.
Consequently, morphism $\mu$ is well-defined. Since morph\-isms
$\mu_i$ are by definition equivariant, then the subscheme
$\mu(\widetilde Q)$ is $PGL(V)$-invariant in $\Hilb^{P(t)}
G(V,r).$

The universal property of the Hilbert scheme guarantees the
existence of the morphism  $\widetilde \mu$ and of fibred diagram
(\ref{glue}).\end{proof}

The following simple result provides independence of the moduli
scheme $\widetilde M$ of the choice of the cover of the scheme
$\widetilde Q$ trivializing the locally free sheaf $ \pi_{\ast}
(\widetilde \E_Q \otimes \widetilde \L_Q^{m}).$

\begin{proposition} {The subscheme  $\mu(\widetilde Q)\subset
\Hilb^{P(t)}G(V,r)$ does not depend on the choice of the covering
$\bigcup_i U_i=\widetilde Q.$}
\end{proposition}
\begin{proof} Choose an another covering  $\widetilde Q=
\bigcup_j \widetilde U_j$ with trivializing isomorphisms
$\widetilde \tau_j:  \pi_{\ast} (\widetilde \E_Q \otimes
\widetilde \L_Q^{m})|_{\widetilde U_j}\stackrel{\sim}{\to} V
\otimes \OO_{\widetilde U_j}$ and repeat the construction of
subscheme $\mu(\widetilde Q).$ Notice that
\begin{eqnarray*}\bigcup_j PGL(V) \star \widetilde U'_j=\bigcup_i
PGL(V)\star U'_i=Q,\nonumber \\
\bigcup_j PGL(V) \star \widetilde U_j=\bigcup_i PGL(V)\star
U_i=\widetilde Q.\nonumber
\end{eqnarray*}

Now consider the intersection  $U_i \cap \widetilde U_j$ and
induced morphisms of trivialization
\begin{equation*}
\xymatrix{V\otimes \OO_{U_i \cap \widetilde U_j}&
\ar[l]_{\!\!\!\!\!\!\!\!\!\tau_i}^{\!\!\!\!\!\!\!\!\!\sim}
\pi_{\ast} (\widetilde \E_Q \otimes \widetilde \L_Q^{m})|_{U_i
\cap \widetilde U_j} \ar[r]^{\;\;\;\;\;\;\;\;\widetilde
\tau_j}_{\;\;\;\;\;\;\;\;\sim}& V\otimes \OO_{U_i \cap \widetilde
U_j}}
\end{equation*}
Hence, on the common part $U_i \cap \widetilde U_j$
trivializations are identified by an appropriate
$GL(V)$-transformation. Consequently, images of the induced maps
of the scheme $U_i \cap \widetilde U_j$ in the Hilbert scheme
$\Hilb^{P(t)} G(V,r)$ are identified by the corresponding
 $PGL(V)$-transformation.
Then we have $$\bigcup_i PGL(V)\star \mu_i(U_i)=\bigcup_j
PGL(V)\star \widetilde \mu_j(\widetilde U_j)=\mu(\widetilde  Q).$$
\end{proof}

\begin{proposition}\label{qpro} {$\mu(\widetilde Q)$ is the quasiprojective
subscheme in  $\Hilb^{P(t)}G(V,r)$.}
\end{proposition}
\begin{proof} Form a closure  $\overline
{\mu(\widetilde Q)}$ of the subscheme  $\mu(\widetilde Q)$ in the
projective scheme $\Hilb^{P(t)}G(V,r)$. It is enough to confirm
that the subset $\mu(\widetilde Q)$ is open in
$\overline{\mu(\widetilde Q)}.$

Let $\overline Q$ be the scheme-theoretic closure of the subscheme
$Q$ in \linebreak $\Quot^{rp_E(t)}(V \otimes L^{(-m)}),$
$\overline{ \widetilde Q}$ be the projective closure of the
quasiprojective scheme $\widetilde Q.$ We claim that the scheme
$\overline{\widetilde Q}$ can be chosen so that there is a fibred
diagram
\begin{equation} \label{close}\xymatrix{\mu(\widetilde Q)
\ar@{^(->} [d]& \ar[l]_{\mu} \widetilde Q \ar@{^(->} [d]
\ar[r]^{\phi}&
Q\ar@{^(->} [d]\\
\overline{\mu(\widetilde Q)}& \ar[l]_{\overline \mu} \overline
{\widetilde Q} \ar[r]^{\overline \phi}& \overline Q}
\end{equation}
with open immersions.

Indeed, let the scheme  $\overline {\widetilde Q}$ is not include
in the diagram (\ref{close}). Consider the locally closed
"diagonal"\, embedding  $\widetilde Q \hookrightarrow
\overline{\mu(\widetilde Q)} \times \overline{\widetilde Q}$. Let
 $Q'$ be the closure if its image. Now form a locally closed
 immersion $\widetilde Q \hookrightarrow Q' \times \overline Q$.
Let  $Q''$ be the closure of the image if this immersion.
Redenoting $\overline{\widetilde Q_1}:=Q''$ we have the required
projective closure. Morphisms $\overline \mu$ and $\overline \phi$
are given by the composite maps of closed immersions and
projections onto the direct summand
\begin{eqnarray*}
&&\overline{\mu}:\overline{\widetilde Q_1}\hookrightarrow Q'
\times \overline Q \stackrel{pr_1}{\longrightarrow}
Q'\hookrightarrow \overline{\mu(\widetilde Q)} \times
\overline{\widetilde Q} \stackrel{pr_1}{\longrightarrow}
\overline{\mu(\widetilde Q)},\nonumber \\
&&\overline{\phi}: \overline{\widetilde Q_1}\hookrightarrow Q'
\times \overline Q \stackrel{pr_2}{\longrightarrow} \overline
Q.\nonumber \end{eqnarray*} Then $\overline{\widetilde Q_1}$ is
the required projective closure. We will denote it by the symbol
$\overline{\widetilde Q}$.

Now we need a lemma to be proven later.

\begin{lemma}\label{boundary}
$\overline{\mu(\widetilde Q)}\setminus \mu(\widetilde Q)=
\overline{\mu} (\overline {\widetilde Q}\setminus \widetilde Q).$
\end{lemma}

Since the scheme  $\widetilde Q$ is quasiprojective, then the
"boundary"\, $\overline{\widetilde Q} \setminus \widetilde Q$ is
closed in  $\overline{\widetilde Q}$. Since the morphism
$\overline \mu$ is projective, it is proper and takes closed
subsets to closed subsets. Hence, the image $\overline{\mu}
(\overline {\widetilde Q}\setminus \widetilde Q)$ is closed in
 $\overline{\mu}(\overline {\widetilde Q}).$ By lemma \ref{boundary}
 the subset $\overline{\mu(\widetilde Q)}\setminus
\mu(\widetilde Q)$ is closed in $\overline{\mu}(\overline
{\widetilde Q}).$ Hence the subset $\mu (\widetilde Q)$ is open in
the projective scheme $\overline{\mu}(\overline {\widetilde Q}).$
Then  $\mu (\widetilde Q)$ is quasiprojective scheme.
\end{proof}
\begin{proof}[of lemma \ref{boundary}]
It suffices to check the set-theoretical equality. The inclusion
 $\overline{\mu(\widetilde Q)}\setminus \mu(\widetilde
Q)\supset \overline{\mu} (\overline {\widetilde Q}\setminus
\widetilde Q)$ follows immediately form the con\-struct\-ion. To
prove the opposite inclusion assume that there is "non-boundary"\,
point $x$ in the preimage
$\overline{\mu}^{\;-1}(\overline{\mu(\widetilde Q)}\setminus
\mu(\widetilde Q))$ of the "boundary". Namely, we suppose that
 a point $x$ does not belong to the subset
$\overline{\widetilde Q} \setminus \widetilde Q$. This means that
$x\in \widetilde Q$ but $\overline{\mu}(x)\in \mu(\widetilde Q).$
This contradiction proves the lemma.
\end{proof}

\section{Distinguished polarization of the scheme  $\widetilde S$
}\label{polar} In this section we obtain the explicit form of the
distinguished polarization on the scheme  $\widetilde S$. Since
this scheme can fail to be a variety we will work with very ample
invertible sheaves instead divisorial classes. Properties of a
morphism $\sigma: \widetilde S \to S$ are mostly similar to ones
of the blowup morphism since $\sigma $ is a structure morphism of
a projective spectrum of an appropriate sheaf algebra.

\begin{proposition} Distinguished polarizations $\widetilde L$
of schemes of the form $\widetilde S$ \linebreak described in the
definition \ref{admis}, provide a Hilbert polynomial which is
constant in flat families of admissible schemes.
\end{proposition}
\begin{proof} Due to (\ref{redto1}), we can work over a regular
one-dimen\-sional base $T$. Let  $\Sigma:= T\times S$ be a trivial
family of surfaces  $S$ supplied with a flat family $\E$ of
semistable coherent torsion-free sheaves with prescribed Hilbert
polynomial $rp_E(t)$. Let $\I=\FFitt^0 \EExt^1(\E, \OO_{\Sigma})$,
and $\sigma \!\!\! \sigma: \widehat \Sigma  \to \Sigma$ be the
morphism of blowing up of the scheme $\Sigma$ in the sheaf of
ideals $\I$. Since $\sigma \!\!\! \sigma$ is a projective
morphism, then there exist a projective bundle $p:\P_{\Sigma}\to
\Sigma$ and a closed immersion  $i: \widehat \Sigma
\hookrightarrow \P_{\Sigma}$ such that the triangle
\begin{equation*}\xymatrix{\widehat \Sigma \ar[r]^i
\ar[rd]_{\sigma \!\!\! \sigma}& \P_{\Sigma} \ar[d]^p\\
&\Sigma}
\end{equation*}
commutes. Indeed, since $\I$ is a finitely generated
$\OO_{\Sigma}$-module, then there exist a locally free
$\OO_{\Sigma}$-sheaf of finite rank $F$ and an epimorphism of
$\OO_{\Sigma}$-modules $F\twoheadrightarrow \I.$ Then there are
the induced epimorphism of symmetric algebras $Sym^{\cdot} F
\twoheadrightarrow \bigoplus_{s \ge 0} \I^s$ and, consequently,
the induced morphism of immersion of projective spectra $\widehat
\Sigma = \Proj \bigoplus_{s \ge 0} \I^s
\stackrel{i}{\hookrightarrow}\P_{\Sigma}=\Proj Sym^{\cdot}F.$

Now consider a composite  $\P_{\Sigma}
\stackrel{p}{\longrightarrow} \Sigma
\stackrel{\pi}{\longrightarrow} T$ of projective morphisms. As
previously, $\L$ is an invertible $\OO_{\Sigma}$-sheaf, very ample
relatively $T$. It corresponds to the sheaf  $L$ under the
restriction onto any fibre of the projection $\pi$. Then by
\cite[ch. II, exercise 7.14 (b)]{Hart}, for  $m\gg 0$ the sheaf
 $\L_T:=\OO_{\P_{\Sigma}}(1) \otimes p^{\ast} \L ^{m}$
is very ample relatively  $T$. The restriction of this sheaf onto
the image of the immersion $i$ gives a very ample relatively $T$
sheaf $\widehat \L =i^{\ast}\L_T= \OO_{\widehat \Sigma}(1) \otimes
\sigma \!\!\! \sigma^{\ast} \L ^{m}$ on the scheme $\widehat
\Sigma.$ Then there exist a projective bundle $p_T:\P_T\to T$ and
a closed immersion $j: \widehat \Sigma \hookrightarrow \P_T$,
include into the commutative triangle
\begin{equation*}\xymatrix{\widehat \Sigma \ar[r]^{j}
\ar[rd]_{\pi \circ p}& \P_T \ar[d]^{p_T}\\
& T}\end{equation*} Here $\widehat \Sigma$ is the closure of the
image of the open subset $\Sigma_0$ under the immersion $j$. It
follows  \cite[ch. III, theorem 9.9]{Hart} that the Hilbert
polynomi\-al of the image  $j(\widehat \Sigma)$ is fibrewise
constant over  $T$.

Prove that the sheaf  $\widehat \L$ provides a polarization given
by $L$, on a general enough fibre which is isomorphic to the
surface $S$. The restriction onto the fibre $\pi^{-1}(t) \cong S$
yields $\widehat \L |_{\pi^{-1}(t)}= i^{\ast}
(\OO_{\P_{\Sigma}}(1) \otimes \sigma \!\!\!
\sigma^{\ast}\L^{m})|_{\pi^{-1}(t)} =L^{m}.$

Restriction onto the special fibre $\pi^{-1} (t_0)\cong \widetilde
S \not \cong S$ results in the equalities $\widetilde L^m:=
\widehat \L |_{\pi^{-1}(t_0)}=
 (\OO_{\widehat \Sigma}(1)\otimes
\sigma \!\!\! \sigma^{\ast}\L^{ m})|_{\pi^{-1}(t_0)}=
 (\sigma \!\!\! \sigma^{-1} \I \cdot \OO_{\widehat \Sigma})\otimes
\sigma \!\!\! \sigma^{\ast}\L^{m}|_{\pi^{-1}(t_0)}= (\sigma^{-1}I
\cdot \OO_{\widetilde S})\otimes \sigma^{\ast} L^{m}$, where
$I\subset \OO_S$ is the sheaf of ideals obtained by the
restriction of the sheaf $\I$ on the fibre $\pi^{-1}(t_0).$
\end{proof}

\section{Isomorphism $H^0(\widetilde S, \widetilde E \otimes \widetilde L^m)
\stackrel{\sim}{\to} H^0(S, E\otimes L^m)$.} Obviously, such an
isomorphism exists because both the spaces of global sections have
equal dimensions. In this section we compute the isomorph\-ism
$H^0(\widetilde S, \widetilde E \otimes \widetilde L^m)
\stackrel{\sim}{\to} H^0(S, E\otimes L^m)$ distinguished by the
construction of standard resolution. This construction puts into
the correspondence a pair $(\widetilde S,\widetilde E)$ to the
coherent torsion-\linebreak free sheaf $E$. First we construct the
required homomorphism of vector spaces and then prove that this is
an isomorphism.

In the sequel we replace schemes $Q$ and $\widetilde Q$ by their
nonsingular resolutions $\xi: Q' \to Q$ and $\widetilde \xi:
\widetilde Q' \to \widetilde Q$ such that there is a birational
morphism $\phi': \widetilde Q' \to Q'.$ The family $\widetilde
\Sigma_Q$ is replaced by its preimage $\widetilde
\Sigma'=\widetilde \Sigma_Q \times _{\widetilde Q} \widetilde Q'$.
Form inverse images $\widetilde \E'_Q=\widetilde \xi^{\ast}
\widetilde \E_Q$ and $\E'_Q=\xi^{\ast}\E_Q.$ Since $\mu(\widetilde
Q')=\mu(\widetilde Q)$ we preserve notations $\widetilde Q$ and
$Q$ for nonsingular resolutions, $\phi: \widetilde Q \to Q$ for
the corresponding birational morphism, $\widetilde \Sigma_Q$ for
the family of schemes, $\widetilde \E_Q$ and $\E_Q$ for families
of sheaves.

\begin{proposition} Apply the standard resolution  to the family
$p: Q\times S \to Q$ and to the family of sheaves $\E_Q$. This
induces for any pair of points $(\widetilde q, q),$ $\widetilde q
\in \widetilde Q,$
 $q=\phi(\widetilde q) \in Q$, the fixed isomorphism $\upsilon:H^0(\widetilde S, \widetilde E \otimes \widetilde L^m)
\stackrel{\sim}{\to} H^0(S, E\otimes L^m)$.
\end{proposition}

\begin{proof} Remind that there is a map $\mu: \widetilde Q
\to \Hilb^{P(t)}G(V,r)$, and $\widetilde \Sigma=\pi^{-1} \mu
(\widetilde Q)$. The other notation is fixed in the following
diagram:
\begin{equation}\label{big}\xymatrix{\widetilde \Sigma \ar[d]_{\pi}
\ar[rdd]^<<<<<<<<M && \ar[ll]_{\widetilde \mu}\widetilde \Sigma_Q
\ar[d]_{\pi} \ar[rdd]^<<<<<<<<{\Phi} \ar[rr]^{\widetilde \phi}&&
Q\times S\ar[d]_p \ar[rdd]^=\\
\mu(\widetilde Q)&& \ar[ll]^{\mu}\widetilde Q
\ar[rr]_{\phi}&&Q\\
&\ar[ul]^p \mu(\widetilde Q) \times S&& \ar[ll]^{\mu \times 1}
\ar[ul]^p \widetilde Q \times S \ar[rr]_{\phi \times 1}&&\ar[ul]^p
Q\times S}
\end{equation}
All the parallelograms of the left hand part and the bottom
parallelogram of the right hand part are fibred. All the morphisms
marked with the symbol $\pi$ are inherited from the structure
morphism of the universal scheme $\Univ^{P(m)}G(V,r).$ All the
morphisms marked with the symbol $p$ are projections of the
corresponding direct products onto the first summands.

The scheme  $Q\times S$ carries the coherent reflexive sheaf
$\E_Q$ with homolo\-gical dimension equal to 1. It is flat over
$Q$. Schemes  $\widetilde \Sigma$ and $\widetilde \Sigma_Q$ are
supplied with locally free sheaves $\widetilde \E$ and $\widetilde
\E_Q$ respectively and  $\widetilde \E_Q= \widetilde \mu^{\ast}
\widetilde \E$. On quasiprojective schemes $\widetilde \Sigma$ and
$Q\times S$ there are fixed invertible sheaves $\widetilde \L$ and
$\L$. They are very ample relatively $\mu(\widetilde Q)$ and $Q$
respectively. In this case  $ (\widetilde \phi_{\ast}\widetilde
\mu^{\ast}\widetilde \L^{m})^{\vee \vee}= \L^{ m}$. The following
notation is introduced: $\widetilde \L_Q:=\widetilde \mu^{\ast}
\widetilde \L$. Also, according to the standard resolution,
$(\widetilde \phi_{\ast} \widetilde \E_Q)^{\vee \vee}=\E_Q$. This
equality implies the following
\begin{lemma} \label{monosh} For appropriate very ample invertible
$\OO_{\widetilde Q}$-sheaf $\widetilde \LL$ and $\OO_Q$-sheaf
$\LL$ such that  $(\phi_{\ast}\widetilde \LL)^{\vee \vee}=\LL$,
and for an appropriate integer  $n\gg 0$ there is an inclusion
\begin{equation*}
(\Phi_{\ast}(\widetilde \E_Q\otimes \widetilde \L_Q^{m}\otimes
\widetilde \LL^{mn}))^{\vee \vee}\hookrightarrow (\phi \times
1)^{\ast}(\E_Q \otimes \L^{m} \otimes \LL^{ mn}).\end{equation*}
\end{lemma}
\begin{corollary}\label{cordesc} There is an inclusion \begin{equation*}
[(\mu \times 1)^{\ast} M_{\ast}(\widetilde \E\otimes \widetilde
\L^{m})]^{\vee \vee}\hookrightarrow (\phi \times 1)^{\ast}(\E_Q
\otimes \L^{m} \otimes \LL^{mn}).
\end{equation*}
\end{corollary}
Proofs of lemma \ref{monosh} and corollary \ref{cordesc} will be
done later.

The consequent base changes accordingly to (\ref{big}) lead to the
chain  \begin{eqnarray}\label{chaim}\pi_{\ast}(\widetilde
\E_Q\times \widetilde \L^{m} )= \pi_{\ast} \widetilde \mu^{\ast}
(\widetilde \E \otimes \widetilde
\L^{m})\stackrel{\sim}{\leftarrow} \mu^{\ast} \pi_{\ast}
(\widetilde \E \otimes \widetilde \L^{m}) = \mu^{\ast}
p_{\ast}M_{\ast}(\widetilde \E \otimes \widetilde \L^{m})\nonumber
\\
\hookrightarrow p_{\ast}(\mu \times 1)^{\ast}M_{\ast} (\widetilde
\E \otimes \widetilde \L^{m})/tors \hookrightarrow p_{\ast}[(\mu
\times 1)^{\ast}M_{\ast} (\widetilde \E \otimes \widetilde \L^{
m})]^{\vee \vee}\nonumber\\
\hookrightarrow \linebreak p_{\ast}(\phi \times
1)^{\ast}(\E_Q\otimes \L^{m}\otimes \LL^{nm}).\end{eqnarray} The
first arrow in (\ref{chaim}) is an isomorphism since $m\gg 0$
\cite[lecture 7, $3^{\circ}$]{Mum}. The second arrow is a morphism
of locally free sheaf and both sheaves coincide on the open
subset. Hence, the second arrow is a monomorphism. The sheaf
morphism induced by the inclusion into the reflexive hull, is an
inclusion because the sheaf from the left is torsion-free. The
recent morphism is defined by the corollary \ref{cordesc}.

For $m\gg 0$ we have the inclusion map of  $\OO_{\widetilde Q}$-
sheaves
\begin{equation}\label{inc}\Upsilon:\pi_{\ast}
(\widetilde \E_Q\otimes \widetilde \L_Q^{m} )\hookrightarrow
p_{\ast} (\phi \times 1)^{\ast} (\E_Q \otimes \L^{m}\otimes
\LL^{nm}).
\end{equation}
Since the sheaf $\widetilde \E_Q\otimes \widetilde \L_Q^{m}$ is
flat over $\widetilde Q$ and the sheaf $\E_Q\otimes \L^{m}$ is
flat over $Q$, both sheaves in (\ref{inc}) are locally free of
rank $rp_E(m)$ .

Since $\pi$ is a projective morphism of Noetherian schemes and the
sheaf $\widetilde \E_Q \otimes \widetilde \L_Q^m$ is flat over
$\widetilde Q$, and for $m\gg 0$ functions $\dim
H^i(\pi^{-1}(\widetilde q), \widetilde E \otimes \widetilde L^m)$
are constant as functions of a point $\widetilde q \in \widetilde
Q$, then by  \cite[ch. III, corollary 12.9]{Hart} there is an
isomorphism $\pi_{\ast} (\widetilde \E_Q\otimes \widetilde \L_Q^m
)\otimes k_{\widetilde q}\cong H^0(\widetilde S, \widetilde E
\otimes \widetilde L^m)$. Here we use following notations:
$\widetilde S= \pi^{-1}(\widetilde q),$ $\widetilde q \in
\widetilde Q,$ $\widetilde E= \widetilde \E|_{\pi^{-1}(\widetilde
q)}$. By the similar reason, $p_{\ast}(\E_Q\otimes \L^m ) \otimes
k_q\cong H^0(S, E\otimes L^m)$ for $S=p^{-1}(q),$ $q\in Q,$
$E=\E|_{p^{-1}(q)}$. Suppose that points $\widetilde q$ and $q$
satisfy $\phi(\widetilde q)=q.$ Then we have the homomorphism of
vector spaces of global sections $\upsilon:H^0(\widetilde S,
\widetilde E\otimes \widetilde L^m) \to H^0(S, E\otimes L^m)$
induced by the inclusion (\ref{inc}).

Note that by the construction of the morphism $\sigma: \widetilde
S \to S$ by means of blowing up a family of surfaces \cite{Tim3}
there is an isomorphism $\sigma^{\ast} \OO_S= \OO_{\widetilde S}$.
Then, keeping in mind the equality $\widetilde E=
\sigma^{\ast}E/tor\!s$, consider the mapping of spaces of global
sections
\begin{equation*}
H^0(\sigma^{\ast}): H^0(S, E\otimes L^m) \to H^0(\widetilde S,
\sigma^{\ast}(E\otimes L^m)).
\end{equation*}
It is induced by formation of  the inverse image. Also consider
the map
\begin{equation*} \zeta:H^0(\widetilde S, \sigma^{\ast}(E\otimes L^m))
\to H^0(\widetilde S, \widetilde E \otimes \sigma^{\ast}L^m),
\end{equation*}
induced by the morphism  $\sigma^{\ast}E \twoheadrightarrow
\widetilde E$, and the inclusion
\begin{equation*}
\xi: H^0(\widetilde S, \widetilde E \otimes \widetilde L^m)
\hookrightarrow H^0(\widetilde S, \widetilde E \otimes
\sigma^{\ast}L^m),
\end{equation*}
induced by the morphism of the invertible sheaves  $\widetilde L^m
\hookrightarrow \sigma^{\ast} L^m$. Then there is a commutative
diagram of homomorphisms of vector spaces
\begin{equation*}\xymatrix{
H^0(S, E\otimes L^m) \ar[rr]^{H^0(\sigma^{\ast})}&&
H^0(\widetilde S, \sigma^{\ast}(E\otimes L^m))\ar[d]^{\zeta}\\
H^0(\widetilde S, \widetilde E \otimes \widetilde L^m)
\ar[u]^{\upsilon} \ar@{^{(}->}[rr]^{\xi}&& H^0(\widetilde S,
\widetilde E \otimes \sigma^{\ast}L^m)}
\end{equation*}
Since  $\xi$ is a monomorphism, then $\upsilon$ is also
monomorphism. Vector spaces  $H^0(\widetilde S, \widetilde E
\otimes \widetilde L^m)$ and $H^0(S, E\otimes L^m)$ have equal
dimensions hence  $\upsilon$ is an isomorphism.
\end{proof}
\begin{proof}[of lemma  \ref{monosh}] Note that there is
a twisted analogue of the formula $[\widetilde \phi_{\ast}
\widetilde \E_Q]^{\vee \vee}=\E_Q$. In can be proven by the
reasoning completely similar to those done in  \cite{Tim2}. For
invertible sheaves  $\widetilde \L_Q$ and  $\L$ such that
$[\widetilde \phi_{\ast} \widetilde \L_Q]^{\vee \vee}=\L$, the
following equality holds $[\widetilde \phi_{\ast} (\widetilde \E_Q
\otimes \widetilde \L_Q^{m})]^{\vee \vee}=\E_Q \otimes \L^{m}.$
Let  $\widetilde \LL$ and $\LL$ be very ample invertible sheaves
as described in the formulation of lemma.

Consider a sheaf $\Phi_{\ast}(\widetilde \E_{Q} \otimes \widetilde
\L_Q \otimes \pi^{\ast}\widetilde \LL^{nm})=\Phi_{\ast}(\widetilde
\E_{Q} \otimes \widetilde \L_Q \otimes
\Phi^{\ast}p^{\ast}\widetilde \LL^{nm})=\Phi_{\ast}(\widetilde
\E_{Q} \otimes \widetilde \L_Q) \otimes p^{\ast}\widetilde
\LL^{nm}.$ The first isomorphism holds by the diagram (\ref{big}),
the second is true by projection formula. Formation of a reflexive
hull yields $[\Phi_{\ast}(\widetilde \E_{Q} \otimes \widetilde
\L_Q \otimes \pi^{\ast}\widetilde \LL^{nm})]^{\vee
\vee}=[\Phi_{\ast}(\widetilde \E_{Q} \otimes \widetilde
\L_Q)]^{\vee \vee} \otimes p^{\ast}\widetilde \LL^{nm}.$ Now
consider another $\OO_{\widetilde Q\times S}$-sheaf $(\phi \times
1)^{\ast} (\E_Q \otimes \L^m).$ It is also reflexive on
$\widetilde Q \times S$ \cite[lemmata 1.2, 1.3]{Tim1}. Sheaves
$[\Phi_{\ast}(\widetilde \E_{Q} \otimes \widetilde \L_Q)]^{\vee
\vee}$ and $(\phi \times 1)^{\ast} (\E_Q \otimes \L^m)$ coincide
on those open subsets of the scheme  $\widetilde Q \times S$ where
they are locally free. These subsets are obtained by excluding of
closed subschemes of codimension $\ge 3$ from $\widetilde Q \times
S$. By assumption, the scheme  $\widetilde Q \times S$ is integral
and normal. Then  \cite[corollary 1.10]{Tim2}
$[\Phi_{\ast}(\widetilde \E_{Q} \otimes \widetilde \L_Q)]^{\vee
\vee}=(\phi \times 1)^{\ast} (\E_Q \otimes \L^m)$. Hence after
tensoring by the invertible sheaf  $p^{\ast}\widetilde \LL^{nm}$
we have the isomorphism
\begin{equation}\label{isosh}[\Phi_{\ast}(\widetilde \E_{Q}
\otimes \widetilde \L_Q \otimes \pi^{\ast}\widetilde
\LL^{nm})]^{\vee \vee}=(\phi \times 1)^{\ast} (\E_Q \otimes
\L^m)\otimes p^{\ast}\widetilde \LL^{nm}\end{equation}

We claim that  $p^{\ast} \widetilde \LL^{nm}\hookrightarrow
p^{\ast} \LL^{nm}.$ Indeed, $\phi_{\ast} \widetilde
\LL^{nm}\hookrightarrow (\phi_{\ast} \widetilde \LL^{nm})^{\vee
\vee}=\LL^{nm}.$ Applicat\-ion of the inverse image results in
$\phi^{\ast} \phi_{\ast}\widetilde \LL^{nm}\twoheadrightarrow
\widetilde \LL^{nm}$. The epimorphicity is provided by the
condition  $nm \gg 0$. Then
 $\phi^{\ast} \phi_{\ast}\widetilde \LL^{nm}/tors =
\widetilde \LL^{nm}$ yields  $\widetilde \LL^{nm}=\phi^{\ast}
\phi_{\ast}\widetilde \LL^{nm}/tors$ $\hookrightarrow \phi^{\ast}
\LL^{nm}$. Tensoring this inclusion by the right hand side of
(\ref{isosh}) we get $(\phi \times 1)^{\ast}(\E_Q \otimes
\L^m)\otimes p^{\ast}\widetilde \LL^{nm} \hookrightarrow (\phi
\times 1)^{\ast}(\E_Q \otimes \L^m)\otimes
p^{\ast}\phi^{\ast}\LL^{nm}=(\phi \times 1)^{\ast}(\E_Q \otimes
\L^m\otimes p^{\ast}\LL^{nm})$. This completes the proof of the
lemma.
\end{proof}

\begin{proof}[Proof of corollary \ref{cordesc}] Note that
 $\widetilde \E_Q \otimes \widetilde \L_Q^m=\widetilde
\mu^{\ast}(\widetilde \E \otimes \widetilde \L^m).$ Since
$\widetilde \LL$ is very ample invertible $\OO_{\widetilde
Q}$-sheaf, there is an inclusion $[\Phi_{\ast} \widetilde
\mu^{\ast}(\widetilde \E \otimes \widetilde \L^m)]^{\vee
\vee}\hookrightarrow [\Phi_{\ast} (\widetilde
\mu^{\ast}(\widetilde \E \otimes \widetilde \L^m)\otimes
\widetilde \LL^{nm}]^{\vee \vee}.$ The base change applied to the
first sheaf, and lemma \ref{monosh} yield $[\Phi_{\ast} \widetilde
\mu^{\ast}(\widetilde \E \otimes \widetilde \L^m)]^{\vee \vee}=[
(\mu\times 1)^{\ast}M_{\ast}(\widetilde \E \otimes \widetilde
\L^m)]^{\vee \vee}\hookrightarrow (\phi \times
1)^{\ast}(\E_Q\otimes \L^m \otimes \LL^{nm})$.
\end{proof}

\section{(Semi)stability}
The notion of (semi)stability for pairs $(\widetilde S, \widetilde
E)$ is defined in this section.

\begin{definition}\label{semistable} $S$-{\it (semi)stable pair}
$((\widetilde S,\widetilde L), \widetilde E)$ is the following
data:
\begin{itemize}
\item{$\widetilde S=\bigcup_{i\ge 0} \widetilde S_i$ --
admissible scheme, $\sigma: \widetilde S \to S$ -- canonical
morphism, $\sigma_i: \widetilde S_i \to S$ -- its restrictions on
components  $\widetilde S_i$, $i\ge 0;$}
\item{$\widetilde E$ -- vector bundle on the scheme
$\widetilde S$;}
\item{$\widetilde L \in Pic\, \widetilde S$ -- distinguished polarization
;}
\end{itemize}
such that
\begin{itemize}
\item{$\chi (\widetilde E \otimes \widetilde
L^{m})=rp_E(t);$}
\item{the sheaf $\widetilde E$ is {\it Gieseker-(semi)stable} on
the scheme $\widetilde S$. Namely, for any proper subsheaf
 $\widetilde F \subset \widetilde E$
for $m\gg 0$ one has
\begin{eqnarray*}
\frac{h^0(\widetilde F\otimes \widetilde L^{m})}{\rank F}&<&
\frac{h^0(\widetilde E\otimes \widetilde L^{m})}{\rank E},
\\ (\mbox{\rm respectively,} \;\;
\frac{h^0(\widetilde F\otimes \widetilde L^{m})}{\rank F}&\leq&
\frac{h^0(\widetilde E\otimes \widetilde L^{m})}{\rank E}\;);
\end{eqnarray*}}
\item{on each of additional components  $\widetilde S_i, i>0,$
the sheaf  $\widetilde E_i:=\widetilde E|_{\widetilde S_i}$ is
{\it quasi-ideal sheaf,} namely has a description of the form
(\ref{ei}) for some \linebreak $q_0\in \bigsqcup_{l\le c_2}
\Quot^l \bigoplus^r \OO_S$. }\end{itemize}
\end{definition}

\begin{remark} If  $\widetilde S \cong S,$ then (semi)stability
of a pair $(\widetilde S, \widetilde E)$ is equivalent to
Gieseker-(semi)stability of vector bundle $\widetilde E$ on the
surface $\widetilde S$ with respect to the polarization
$\widetilde L \in Pic\, \widetilde S.$\end{remark}

To investigate the relation of $S$-(semi)stability of the pair
 $(\widetilde S, \widetilde E)$ to Gieseker-(se\-mi)stability of
 the corresponding sheaf $E$ on the surface $S$ note that for $m\gg 0$
 $rp_E(m)=h^0(E\otimes
L^{m})$. For the Gieseker-stability the behavior of the Hilbert
polynomial under $m\gg 0$ is important. Therefore we assume that
$m$ is big enough.

\begin{definition} The locally free sheaf $\widetilde E$ on the
admissible scheme $\widetilde S$ is said to be {\it obtained from
the sheaf $E$ by its standard resolution} if there exists a flat
family $\E$ of coherent $\OO_S$-sheaves with base
$T=\Spec k[t]$, such that \\
(i) for $t\ne 0$ sheaves  $E_t=\E|_{t\ne 0}$ are locally free;\\
(ii) for $t=0$ the sheaf $E_0=\E|_{t=0}$ is isomorphic to the
sheaf $E$;\\
(iii) standard resolution yields in the blowing up $\sigma
\!\!\!\sigma: \widetilde{T\times S}\to T\times S$ supplied with
locally free sheaf $\widetilde \E$. The fibre of the composite map
$\widetilde{T\times S} \stackrel{\sigma \!\!\! \sigma}{\to} T
\times S \stackrel{p}{\to} T$ at the point $t=0$ is isomorphic to
 $\widetilde S$ and carries the locally free sheaf $\widetilde \E|_{t=0}\cong \widetilde E.$
\end{definition}
\begin{remark} In particular, by the proposition \ref{resdes} this definition means that
for the locally free $\OO_{\widetilde S}$-sheaf $\widetilde E$
there is a coherent $\OO_S$-sheaf  $E$ such that $\widetilde E=
\sigma^{\ast}E /tor\!s.$
\end{remark}

\begin{proposition}\label{ssc} Let the locally free  $\OO_{\widetilde
S}$-sheaf $\widetilde E$ is obtained from a coherent $\OO_S$-sheaf
$E$ by its standard resolution. The sheaf $\widetilde E$ is
(semi)stable on the scheme
 $\widetilde S$ if and only if the sheaf $E$ is (semi)stable.
\end{proposition}
\begin{proof} Let  $E$ be Gieseker-semistable on
$(S,L)$ and $\widetilde E$ be the locally free sheaf on the scheme
$\widetilde S$. Let $\widetilde E$ be obtained from $E$ by
standard resolution. Obviously, $\widetilde E$ is quasi-ideal
sheaf on additional components of $\widetilde S$ provided it is
obtained from a coherent sheaf by standard resolution. Fix any
point  $q \in \Quot^{rp_E(t)} (V \otimes L^{(-m)})$ corresponding
to the quotient sheaf $E$. Consider a proper subsheaf  $\widetilde
F \subset \widetilde E.$ Since  $m\gg 0$ we assume that both the
sheaves  $\widetilde E \otimes \widetilde L^{m}$ and  $\widetilde
F\otimes \widetilde L^{m}$ are globally generated. Fix an
epimorphism $H^0(\widetilde S, \widetilde E \otimes \widetilde
L^{m}) \otimes \widetilde L^{(-m)} \twoheadrightarrow \widetilde
E$. The subsheaf $\widetilde F$ is generated by a subspace of
global sections $V_{\widetilde F}=H^0(\widetilde S, \widetilde F
\otimes \widetilde L^{m})\subset H^0(\widetilde S, \widetilde E
\otimes \widetilde L^{m}).$ Then a subspace $V_F\subset H^0(S,
E\otimes L^{m})$ which is isomorphic to $V_{\widetilde F}$ and
generates some subsheaf $F\in E$, is given by the distinguished
isomorphism $\upsilon:H^0(\widetilde S, \widetilde E \otimes
\widetilde L^{m}) \stackrel{\sim}{\to} H^0(S,E\otimes L^{m})$ by
the equality $V_F=\upsilon (V_{\widetilde F})$. Since sheaves
 $\widetilde F$ and $F$ are canonically isomorphic on the
 corresponding open subsets of schemes $\widetilde S$ and $S$,
 then their ranks are equal. Clearly,  $V_F= H^0(S, F\otimes L^{ m})$
 and
\begin{eqnarray*}\frac{h^0(\widetilde S,
\widetilde E \otimes \widetilde L^{m})}{r}&-& \frac{h^0(\widetilde
S,
\widetilde F \otimes \widetilde L^{m})}{r'}\nonumber \\
=\frac{h^0(S, E \otimes L^{m})}{r}&-& \frac{h^0( S, F \otimes
L^{m})}{r'}>(\ge) 0.\nonumber
\end{eqnarray*}
This implies the semistability of $\widetilde E.$ The opposite
implication is proven similarly.
\end{proof}
\begin{remark} \label{sscr} This shows that there is a bijection among
subsheaves of $\OO_S$-sheaf $E$ and subsheaves of the
corresponding $\OO_{\widetilde S}$-sheaf $\widetilde E.$ This
bijection preserves Hilbert polynomials.
\end{remark}

\section{M-equivalence of semistable pairs}

In this section we investigate the behavior of Jordan --
H\"{o}lder filtration for semistable coherent sheaf under the
standard resolution. Also the notion of M-equi\-val\-ence for
semistable pairs is introduced and relation of M-equivalence to
S-equival\-ence for semistable coherent sheaves is examined. In
particular it is proven that S-equivalent coherent sheaves on the
surface $S$ are resolved in M-equivalent pairs of the form
$(\widetilde S, \widetilde E)$.

Remind some notions from the theory of semistable coherent
sheaves.
\begin{definition}\cite[definition 1.5.1]{HL}
The {\it Jordan -- H\"{o}lder filtration} for semistable sheaf
 $E$ with reduced Hilbert polynomial $p_E(t)$ on the
 polarized projective scheme  $X$ is a sequence of subsheaves
 \begin{equation*} 0=F_{0}\subset F_1
\subset \dots \subset F_{\ell}=E,
\end{equation*}
such that quotient sheaves  $gr_i(E)=F_i / F_{i-1}$ are stable
with reduced Hilbert polynomials equal $p_E(t)$.
\end{definition}
Denote by the symbol  $gr(E)$ a polystable sheaf
$\bigoplus_{i=1}^{\ell} gr_i(E)$. Well-known theorem \cite[Prop.
1.5.2]{HL} claims that the isomorphism class of the sheaf
 $gr(E)$ has no dependence on a choice of Jordan -- H\"{o}lder
 filtration of $E$.
\begin{definition}\cite[definition 1.5.3]{HL} Semistable
sheaves $E$ and $E'$ are called  {\it S-equivalent} if
$gr(E)=gr(E').$
\end{definition}
\begin{remark} Obviously, S-equivalent stable sheaves are
isomorphic.\end{remark}

Define Jordan-H\"{o}lder filtration for $S$-semistable sheaf on
reducible admissible polarized scheme $(\widetilde S, \widetilde
L)$. This definition will be completely analogous to the classical
definition for Gieseker-semistable sheaf.

\begin{definition} {\it Jordan -- H\"{o}lder filtration}
for a sheaf $\widetilde E$ on the polarized project\-ive reducible
scheme $(\widetilde S, \widetilde L)$ such that a pair
$((\widetilde S, \widetilde L), \widetilde E)$ is semistable in
the sense of definition \ref{semistable}, and with reduced Hilbert
polynomial  $p_{E}(t),$ is a sequence of subsheaves
\begin{equation*} 0=\widetilde F_{0}\subset \widetilde F_1
\subset \dots \subset \widetilde F_{\ell}=\widetilde E,
\end{equation*}
such that quotients  $gr_i(\widetilde E)=\widetilde F_i
/\widetilde F_{i-1}$ are Gieseker-stable with reduced Hilbert
polynomials equal to $p_{E}(t)$.
\end{definition}

The following example shows that S-equivalent coherent sheaves can
have different associated sheaves of Fitting ideals leading to
non-iso\-morphic schemes $\widetilde S$.
\begin{example}
Consider scheme of moduli for semistable coherent sheaves of rank
2 with Chern classes $c_1=0,$ $c_2=2$ on the complex projective
plane $S=\P^2$. As it is proven in \cite{LeP79}, for even values
of $c_1$ and $c_2-c_1^2/4$ the moduli scheme for semistable
sheaves has no universal family. This means that there is strictly
semistable coherent sheaf $E$ with Jordan -- H\"{o}lder filtration
which leads to the exact triple $0\to I \to E \to I' \to 0$. Here
 $I, I'$ are sheaves of maximal ideals of a reduced point $x\in \P^2.$
Note that a polystable sheaf which is S-equivalent to the sheaf
$E$ equals  $I\oplus I'.$ In this case $\Ext^1(I', I)\ne 0.$ To
prove this consider an exact $\OO_S$-triple $0\to I\to \OO_S \to
k_x\to 0$ and apply the functor $\Ext^{\cdot}(I',-)$. Since all
extensions of the form  $0\to k_x \to A \to I' \to 0$ are trivial,
then  $\Ext^1(I', k_x)=0.$ We have an isomorphism of groups of
extensions $\Ext^1(I',I) \cong \Ext^1(I', \OO_S).$ The last group
is non-trivial. Indeed, it contains the class $\varepsilon$ of
non-trivial extension which corresponds to the locally free
resolution for the sheaf of ideals $I'$: $0\to \OO_S \to F_0 \to
I' \to 0$. The non-trivial extension $E$ corresponding to
$\varepsilon$ in  $\Ext^1(I',I)$ includes into the exact diagram
\begin{equation*} \xymatrix{&0&0\\
&k_x \ar[u] \ar[r]^= & k_x \ar[u]\\
0\ar[r]& \OO_S \ar[u] \ar[r]& F_0 \ar[u] \ar[r]& I' \ar[r]& 0\\
0\ar[r]& I \ar[u] \ar[r]& E \ar[u] \ar[r]& I'\ar[u]_{=}
\ar[r]&0\\
&0\ar[u]& 0\ar[u]}
\end{equation*}

For a coherent torsion-free $\OO_S$-sheaf $F$ we use the notation
$\varkappa (F): =F^{\vee \vee}/F.$

From the middle vertical triple we have $\varkappa (E)= k_x.$ Also
for the polystable sheaf $I \oplus I'$ holds $\varkappa(I \oplus
I')=k_x^{\oplus 2}$. Then $\FFitt^0\EExt^2(\varkappa (E),
\OO_S)={\mathfrak m}_x$ and \linebreak
$\FFitt^0\EExt^2(\varkappa(I \oplus I'),\OO_S)=
\FFitt^0\EExt^2(\varkappa(I), \OO_x)\cdot
\FFitt^0\EExt^2(\varkappa(I'), \OO_x)=I I'={\mathfrak m}_x^2$ is a
sheaf of ideals of the first infinitesimal neighborhood of the
point $x$.
\end{example}

The following example shows that the fibred product cannot be used
to construct the notion of equivalence for semistable pairs.

\begin{example} Consider sheaves of maximal ideals
$I_1=I_2={\mathfrak m}_x$ of a reduced point $x\in S$. Then
corresponding schemes $\widetilde S_1$ and $\widetilde S_2$ have
the form $\widetilde S_1=\Proj \bigoplus_{s\ge
0}(I_1[t]+(t))^s/(t^{s+1})$ and $\widetilde S_2=\Proj
\bigoplus_{s\ge 0}(I_2[t]+(t))^s/(t^{s+1})$. As usually $\sigma_i:
\widetilde S_i \to S$ are canonical morphisms, $\widetilde
S_i=\widetilde S_{i0} \bigsqcup_{\sigma_{i0}^{-1}(x)} \P^2$ is the
decomposition into irreducible components where
$\sigma_i|_{\widetilde S_0}=\sigma_{i0}: \widetilde S_{i0} \to S$
is a blowing up of reduced point $x$, $\sigma_{i0}^{-1}(x)\cong
\P^1$ is exceptional divisor of this blowing up. Schemes
$\widetilde S_1$ and $\widetilde S_2$ are isomorphic. Let $i:
\widetilde S_1 \to \widetilde S_2$ be the identifying isomorphism.
Let schemes $\widetilde S_1$ and $\widetilde S_2$ carry stable
vector bundles $\widetilde E_1$ and $\widetilde E_2$ which are
images of nonlocally free coherent sheaf $E$ on the surface $S$.
Obviously, $i_{\ast}\widetilde E_1=\widetilde E_2$. Obviously, in
this case vector bundles $\widetilde E_i,$ $i=1,2$ are nontrivial
under restriction on the exceptional divisor $\P^1$. Form the
fibred product $\widetilde S_1 \times _S \widetilde S_2$, and let
$\sigma'_i: \widetilde S_1 \times _S \widetilde S_2 \to \widetilde
S_i$ be its projections on factors. The product  $\widetilde S_1
\times _S \widetilde S_2$ contains four-dimensional component.
This component is isomorphic to the product $\P^2 \times \P^2$. It
contains  the product of exceptional divisors of blowing ups
$\sigma_{i0}: \widetilde S_{i0} \to S$ as a closed subscheme
isomorphic to a quadric $\P^1 \times \P^1$. Then inverse images
$\sigma'^{\ast}_i \widetilde E_i$ turn to be non-isomorphic on the
fibred product $\widetilde S_1 \times_S \widetilde S_2$. Indeed,
the restriction $\sigma'^{\ast}_1 \widetilde E_1|_{\P^1 \times
\P^1}$ is non-trivial along first factor of the product $\P^1
\times \P^1$ and trivial along the second one. The restriction
$\sigma'^{\ast}_2 \widetilde E_2|_{\P^1 \times \P^1}$ is trivial
along the first factor and non-trivial along the second one.
\end{example}

Now consider the schemes $\widetilde S_1= \Proj \bigoplus _{s\ge
0}(I_1[t]+(t))^s/(t)^{s+1}$ and $\widetilde S_2= \linebreak \Proj
\bigoplus _{s\ge 0}(I_2[t]+(t))^s/(t)^{s+1}$ with their canonical
morphisms $\sigma _1:\widetilde S_1 \to S$ and $\sigma_2:
\widetilde S_2 \to S$ to the surface  $S$. Form inverse images of
sheaves of ideals $I'_2=\sigma_1^{-1} I_2 \cdot \OO_{\widetilde
S_1}\subset \OO_{\widetilde S_1}$ and $I'_1=\sigma_2^{-1} I_1
\cdot \OO_{\widetilde S_2}\subset \OO_{\widetilde S_2}$, and
projective spectra $\widetilde S_{12}=\Proj (\bigoplus _{s\ge
0}(I'_2[t]+(t))^s/(t)^{s+1})$ and $\widetilde S_{21}=\Proj
(\bigoplus _{s\ge 0}(I'_1[t]+(t))^s/(t)^{s+1})$. There are
canonical morphisms  $\sigma'_2: \widetilde S_{12}\to \widetilde
S_1$ and $\sigma'_1: \widetilde S_{21}\to \widetilde S_2$.

\begin{proposition} $\widetilde S_{12}$ and $\widetilde
S_{21}$ are equidimensional schemes. Moreover, $\widetilde S_{12}
\cong \widetilde S_{21}.$
\end{proposition}
\begin{proof}
First we prove that $\widetilde S_{12}\cong \widetilde S_{21}$, и
and that these schemes can be include into flat families with
general fibre isomorphic to $S$, or to $\widetilde S_1$, or to
 $\widetilde S_2$. This implies that all components of the scheme
$\widetilde S_{12}$ have dimension not bigger then 2. Then we will
give the scheme-theoretic characterization of schemes $\widetilde
S_{12}$. It proves that  $\widetilde S_{12}$ is equidimensional
scheme, namely, all reduced schemes corresponding to its
components have dimension 2.

Let $T= \Spec k[t].$ Turn to the trivial 2-parameter family of
surfaces  $T\times T \times S$ with projections $T \times S
\stackrel{p_{13}}{\longleftarrow}T\times T\times S
\stackrel{p_{23}}{\longrightarrow} T\times S$. Introduce the
notations $\I_1:=\OO_T \boxtimes I_1 \subset \OO_{T \times S},$
$\I_2:= \OO_T \boxtimes I_2\subset \OO_{T\times S}$. Form inverse
images  $p_{13}^{\ast }\I_1$ and $p_{23}^{\ast} \I_2$. These are
sheaves of ideals on the scheme $T\times T\times S.$ Consider the
morphism $\sigma \!\!\! \sigma_1 \times \id_T: \widehat \Sigma_1
\times T \to T\times T \times S$ with identity map on the second
factor. Also consider a preimage $(\sigma \!\!\! \sigma_1 \times
\id_T)^{-1} p_{23}^{\ast}\I_2 \cdot \OO_{\widehat \Sigma_1 \times
T}$ on the scheme  $\widehat \Sigma_1 \times T$, and the
corresponding morphism of blowing up $\sigma\!\!\! \sigma_{12}:
\Sigma \!\!\! \Sigma_{12} \to \widehat \Sigma_1 \times T$. Now
restrict the sheaf $(\sigma \!\!\! \sigma_1 \times \id_T)^{-1}
p_{23}^{\ast}\I_2 \cdot \OO_{\widehat \Sigma_1 \times T}$ on the
fibre of the composite map $\widehat \Sigma_1 \times T
\stackrel{\sigma\!\!\! \sigma_1 \!\times
\id_T}{-\!\!-\!\!\!\longrightarrow} T\times T\times S
\stackrel{p_{12}}{\longrightarrow} T\times T$ в точке $(t_1,
t_2).$ Let $\widetilde i: \widetilde S_1 \hookrightarrow \widehat
\Sigma_1 \times T$ be the morphism of the embedding of this fibre.
The commutativity of the diagram
\begin{equation*}\xymatrix{\widehat \Sigma_1 \times T
\ar[rr]^{\sigma \!\!\!\sigma_1 \times \id_T} &&T\times T\times S
\ar[rr]^{p_{12}}&& T\times T\\
\widetilde S_1 \ar@{^(->}[u]^{\widetilde i} \ar[rr]^{\sigma_1}&& S
\ar@{^(->}[u]_i \ar[rr]&& (t_1, t_2)\ar@{^(->}[u]}
\end{equation*}
leads to $\widetilde i^{-1}((\sigma \!\!\! \sigma_1 \times
\id_T)^{-1} p_{23}^{\ast} \I_2 \cdot \OO_{\widehat \Sigma_1 \times
T})\cdot \OO_{\widetilde S_1}=\sigma_1^{-1}
i^{-1}(p_{23}^{\ast}\I_2)\cdot \OO_{\widetilde S_1} =\sigma_1^{-1}
I_2 \cdot \OO_{\widetilde S_1}$.

Now consider the embedding of the line $j_T: T \hookrightarrow T
\times T$ fixed by the equation $at_1+bt_2+c=0$, $a,b,c \in k$.
The corresponding fibred diagram
\begin{equation}\label{famil}\xymatrix{\Sigma \!\!\! \Sigma_{12}
\ar[r]^{\!\!\sigma \!\!\! \sigma_{12}\;\;}&\widehat \Sigma_1\times
T \ar[rr]^{\sigma \!\!\! \sigma_1 \times \id_T} &&T\times T\times
S
\ar[r]^{\;\;\;\;p_{12}}&T\times T\\
\Sigma \!\!\! \Sigma_{12j} \ar[u]^{j_{12}} \ar[r]&\Sigma_{1j}
\ar[u]_{j_1} \ar[rr]^{\sigma \!\!\! \sigma'_1}&& T\times S
\ar[u]^{j_T \times  \id_S} \ar[r]&T\ar[u]_{j_T}}
\end{equation}
fixes notations. If the embedding $j_T$ does not correspond to the
case $b=0$ then $\Sigma_{1j}\simeq \widehat \Sigma _1$ and
$j_1^{-1}((\sigma \!\!\! \sigma_1 \times \id_T)^{-1} p_{23}^{\ast}
\I_2 \cdot \OO_{\widehat \Sigma \times T})\cdot \OO_{\Sigma_{1j}}=
\sigma \!\!\! \sigma_1^{-1}\I_2 \cdot \OO_{\Sigma_1}.$ Otherwise
 (for $b=0$) we have $\Sigma_{1j}\cong T \times S.$

The morphism $\sigma_{1j}: \widehat \Sigma_{1j} \to \Sigma_{1j}$
of the blowing up of the sheaf of ideals $\sigma \!\!\!
\sigma_1^{-1} \I_2 \cdot \OO_{\Sigma_1}$ is include into the
commutative diagram
\begin{equation*}\xymatrix{\Sigma \!\!\! \Sigma_{12}
\ar[rr]^{\sigma\!\!\! \sigma_{12}}&& \widehat \Sigma_1 \times T\\
\widehat \Sigma_{1j} \ar[u] \ar[rr]^{\sigma_{1j}} && \Sigma_{1j}
\ar[u]_{j_1}}
\end{equation*}
By the universal property of the left fibred product in
(\ref{famil}), there is a morphism  $u: \widehat \Sigma_{1j}\to
\Sigma \!\!\! \Sigma_{12j}$.

The morphism of blowing up  $\sigma \!\!\! \sigma'_2: \widehat
\Sigma_{12} \to \widehat \Sigma_1$ of the sheaf of ideals $\sigma
\!\!\! \sigma_1^{-1}\I_2 \cdot \OO_{\widehat \Sigma_1}$ is include
into the commutative diagram
\begin{equation}\label{gldiamond}\xymatrix{\widehat \Sigma_{12}
\ar[r]^{\sigma \!\!\! \sigma'_1}\ar[d]_{\sigma \!\!\! \sigma'_2}
& \widehat \Sigma_2 \ar[d]^{\sigma \!\!\! \sigma_2}\\
\widehat \Sigma_1 \ar[r]^{\sigma \!\!\! \sigma_1}& T\times S}
\end{equation}
Note that in this diagram $\sigma\!\!\!\sigma'_1$ is a morphism of
blowing up of the sheaf of ideals $\sigma \!\!\! \sigma_2^{-1}\I_1
\cdot \OO_{\widehat \Sigma_2}$ and it follows that  $\widehat
\Sigma_{21} =\widehat \Sigma_{12}$. Also $\widehat \Sigma_1,$
$\widehat \Sigma_2$, $\widehat \Sigma_{12}$ are reduced
irreducible schemes. Each of them is fibred over the regular
one-dimensional base $T$ with fibres isomorphic to the projective
schemes. Hence schemes $\widehat \Sigma_1,$ $\widehat \Sigma_2$,
$\widehat \Sigma_{12}$ are flat families of projective schemes
over $T$. Each of these families has fibre isomorphic to the
surface $S$, at general enough point of $T$. This implies that
each fibre of the family $\widehat \Sigma _{12}$ has a form of
projective spectrum $\Proj \bigoplus_{s\ge
0}(I[t]+(t))^s/(t)^{s+1}$ for an appropriate sheaf of ideals
$I\subset \OO_{S}$. Fibres of flat family of projective schemes
carry polarizations with following property. Hilbert polynomials
of fibres compute with respect to these polarizations, remain
constant over the base. By the construction, such polarizations on
fibres of schemes $\widehat \Sigma_1,$ $\widehat \Sigma_2$,
$\widehat \Sigma_{12}$ are exactly the same as polarizations
compute in \ref{polar}.

Now we prove that  $\Sigma \!\!\!\Sigma_{12j}$ if family of
schemes flat over $T$. Consider the exact $\OO_{\widehat \Sigma_1
\times T}$-triple induced by the sheaf of ideals $(\sigma \!\!\!
\sigma_1 \times \id_T)^{-1} p_{23}^{\ast}\I_2 \cdot \OO_{\widehat
\Sigma_1 \times T}$:
$$
0\to (\sigma \!\!\! \sigma_1 \times \id_T)^{-1} p_{23}^{\ast}\I_2
\cdot \OO_{\widehat \Sigma_1 \times T} \to \OO_{\widehat \Sigma_1
\times T} \to \OO_Z \to 0
$$
for an appropriate closed subscheme  $Z$. Apply the functor
$j_1^{\ast}$ and note that the sheaf of ideals $j_1^{-1}((\sigma
\!\!\! \sigma_1 \times \id_T)^{-1} p_{23}^{\ast}\I_2 \cdot
\OO_{\widehat \Sigma_1 \times T}) \cdot \OO_{\Sigma_{1j}}$ is
isomorphic to the quotient sheaf $j_1^{\ast}(\sigma \!\!\!
\sigma_1 \times \id_T)^{-1} p_{23}^{\ast}\I_2 \cdot \OO_{\widehat
\Sigma_1 \times T}/tor\!s,$ for the torsion subsheaf given by the
equality $tor\!s= \TTor_1^{j_1^{-1}\OO_{\widehat \Sigma_1 \times
T}} (j_1^{-1} \OO_Z, \OO_{\Sigma_{1j}})$. Note that
$\Sigma_{1j}\cong \widehat \Sigma_1,$ and $j_1^{\ast}\OO_{\widehat
\Sigma_1 \times T}\cong \OO_{\Sigma_{1j}}.$ With the last two
isomorphisms taken into account we have \linebreak
$\TTor_1^{j_1^{-1} \OO_{\widehat \Sigma_1 \times T}}(j_1^{-1}
\OO_Z, \OO_{\Sigma_{1j}})=
\TTor_1^{\OO_{\Sigma_{1j}}}(j_1^{-1}\OO_Z, \OO_{\Sigma_{1j}})=0.$
Then  $$j_1^{\ast}(\sigma \!\!\! \sigma_1 \times \id_T)^{-1}
p_{23}^{\ast}\I_2 \cdot \OO_{\widehat \Sigma_1 \times T}=
j_1^{-1}((\sigma \!\!\! \sigma_1 \times \id_T)^{-1}
p_{23}^{\ast}\I_2 \cdot \OO_{\widehat \Sigma_1 \times T}) \cdot
\OO_{\Sigma_{1j}}=\sigma \!\!\! \sigma_1^{-1} \I_2 \cdot
\OO_{\Sigma_1}.$$ Also for blowups one has $\Sigma \!\!\!
\Sigma_{12j}=\Proj \bigoplus_{s\ge 0} (j_1^{\ast}(\sigma \!\!\!
\sigma_1 \times \id_T)^{-1} p_{23}^{\ast}\I_2 \cdot \OO_{\widehat
\Sigma_1 \times T})^s= \Proj \bigoplus_{s\ge 0} (\sigma \!\!\!
\sigma_1^{-1} \I_2 \cdot \OO_{\Sigma_1})^s=\widehat \Sigma_{12}$.
Since $\widehat \Sigma_{12}$ is a flat family over $T$ then the
scheme $\Sigma \!\!\! \Sigma_{12j}$ is also flat over $T$.

Any two points on  $T\times T$ can be connected by a chain of two
lines satisfying the condition $b\ne 0.$ Then Hilbert polynomials
of fibres of the scheme $\Sigma \!\!\! \Sigma_{12} \to T\times T$
are constant over the base  $T\times T$. Hence the scheme
$\Sigma\!\!\! \Sigma_{12}$ is flat over the base $T\times T.$

To characterize the scheme structure of the special fibre of the
scheme $\Sigma_{12}$ (and consequently the corresponding fibre of
the scheme  $\Sigma \!\!\! \Sigma_{12}$) it is enough to consider
the embedding  $j_T$ defined by the equation $t_2=0,$ and a
subscheme $\widetilde \Sigma _1=j_1(\Sigma_{1j})$. It is a flat
family of subschemes with fibre $\widetilde S_1 =\Proj
\bigoplus_{s\ge 0} (I_1[t]+(t))^s/(t)^{s+1}$. As proven before,
the preimage  $\widetilde \Sigma_{12}= \sigma \!\!\!
\sigma_{12}^{-1}(\widetilde \Sigma_1)$ is also flat over $j_T(T)
\cong T$ with generic fibre isomorphic to $\widetilde S_1=\Proj
\bigoplus_{s\ge 0} (I_1[t]+(t))^s/(t)^{s+1}$. Applying in this
situation the reasoning of the article \cite{Tim3} we obtain that
the special fibre $\widetilde S_{12}$ of the scheme  $\widetilde
\Sigma_{12}$ has the following scheme-theoretic characterization:
$\widetilde S_{12}= \Proj \bigoplus_{s\ge
0}(I'_2[t]+(t))^s/(t)^{s+1}$ for the sheaf of ideals $I'_2 \subset
\OO_{\widetilde S_1}$ defined as $I'_2=\sigma_1^{-1}I_1 \cdot
\OO_{\widetilde S_1}.$
\end{proof}

Hence, for any two schemes $\widetilde S_1=\Proj \bigoplus_{s\ge
0}(I_1[t]+(t))^{s}/(t)^{s+1}$ and $\widetilde S_2=\Proj
\bigoplus_{s\ge 0}(I_2[t]+(t))^{s}/(t)^{s+1}$ the scheme
$\widetilde S_{12}=\Proj \bigoplus_{s\ge
0}(I'_1[t]+(t))^{s}/(t)^{s+1}=\Proj \bigoplus_{s\ge
0}(I'_2[t]+(t))^{s}/(t)^{s+1}$ is defined together with morphisms
$\widetilde S_1\stackrel{\sigma'_1}{\longleftarrow} \widetilde
S_{12} \stackrel{\sigma'_2}{\longrightarrow} \widetilde S_2,$ such
that the diagram
\begin{equation*}\xymatrix{\widetilde S_{12}
\ar[r]^{\sigma'_2}\ar[d]_{\sigma'_1}& \widetilde S_2
\ar[d]^{\sigma_2}\\
\widetilde S_1 \ar[r]_{\sigma_1}&S}
\end{equation*}
commutes. The operation $(\widetilde S_1, \widetilde S_2) \mapsto
\widetilde S_1 \diamond \widetilde S_2=\widetilde S_{12}$ defined
by this way, is obviously associative. Moreover, since for any
admissible morphism $\sigma :\widetilde S\to S$ there are
equalities $\widetilde S \diamond S =S \diamond \widetilde S
=\widetilde S$, then admissible morphisms of each class $[E]$ of
S-equivalent semistable coherent sheaves generate a commutative
monoid  $\diamondsuit[E]$ with binary operation $\diamond$ and
neutral element $\id_S: S \to S.$

Note that by proposition  \ref{ssc} and remark \ref{sscr} there is
a bijective correspondence among subsheaves of coherent
$\OO_S$-sheaf $E$ and subsheaves of the corresponding locally free
$\OO_{\widetilde S}$-sheaf $\widetilde E.$ This correspondence
pre\-serves Hilbert polynomials. Let there is a fixed
Jordan-H\"{o}lder filtration in $E$ formed by subsheaves  $F_i$.
Then there is a sequence of semistable subsheaves $\widetilde F_i$
with the same reduced Hilbert polynomial and $\rank \widetilde
F_i=\rank F_i$ distinguished in $\widetilde E$ by the described
correspondence.

Let $X$ be a projective scheme, $L$ be ample invertible
$\OO_X$-sheaf, $E$ be a coherent $\OO_X$-sheaf. Let the sheaf
$E\otimes L^{m}$ is globally generated, namely, there is an
epimorph\-ism $q: H^0(X, E\otimes L^{m}) \otimes L^{ (-m)}
\twoheadrightarrow E.$ Fix a subspace $H \subset H^0(X, E\otimes
L^{m}).$ The subsheaf $F\subset E$ is said to be {\it generated by
the subspace $H$} if it is an image of the composite map  $H
\otimes L^{(-m)}\subset H^0(X, E\otimes L^{ m}) \otimes L^{
(-m)}\stackrel{q} {\twoheadrightarrow }E.$

\begin{proposition} \label{corr} The transformation $E \mapsto
\sigma^{\ast}E /tors$ is compatible on all sub\-sheaves $F \subset
E$ with the isomorphism  $\upsilon$ for all $m\gg 0.$
\end{proposition}
\begin{proof} Take an arbitrary subsheaf $F\subset E$
of rank $r'$. It necessary to check that the subsheaf $\widetilde
F\subset \widetilde E=\sigma^{\ast}E /tor\!s$ generated in
 $\widetilde E$ by the subspace $\upsilon^{-1}
H^0(S, F\otimes L^{m})$, coincides with the subsheaf
$\sigma^{\ast}F /tor\!s.$

It is clear that sheaves  $\sigma^{\ast}F /tor\!s $ and
$\widetilde F$ coincide on the open subset $W$ of the scheme
$\widetilde S$ where the scheme morphism $\sigma: \widetilde S \to
S$ is an isomorphism. Then it rests to check their coincidence on
additional components of the scheme $\widetilde S.$ The structure
of the sheaf $\widetilde E$ on additional components is described
by the data (\ref{addker}, \ref{ei}). As earlier, $U$ is the open
neighborhood of $\Supp \varkappa$ in $S$. Coincidence of
subsheaves
 $\widetilde F$ and $\sigma^{\ast}F
 /tor\!s$ on the open subset  $\sigma^{-1}(U)\cap W$ provides
(possibly after diminishing of the open subset $W$) isomorphisms
$\widetilde F|_{\sigma^{-1}(U) \cap W}= (\sigma ^{\ast}
F/tors)|_{\sigma^{-1}(U) \cap W}=\sigma^{\ast} \bigoplus^{r'}
\OO_{U \cap \sigma(W)}$ and  the inclusion of the inverse images
of trivial sheaves $\sigma^{\ast} \bigoplus^{r'} \OO_{U \cap
\sigma(W)}\hookrightarrow$  $\sigma^{\ast} \bigoplus^r \OO_{U \cap
\sigma(W)}$.

There is a commutative triangle
\begin{equation*}\xymatrix{\bigoplus^r \OO_U \ar@{->>}[r]^{q_0}& \varkappa \\
\bigoplus^{r'}\OO_U \ar@{^(->}[u] \ar[ur]^{q'_0} }
\end{equation*}
where the morphism $q'_0$ is defined as composite map. Application
of the functor of the inverse image $\sigma^{\ast}$ and
restrictions on each of additional components lead to the
expressions
\begin{eqnarray*}\widetilde E_i&=&\sigma_i^{\ast}\ker q_0/tor\!s_i,\\
\widetilde F_i=\sigma^{\ast}F/tor\!s|_{\widetilde
S_i}&=&\sigma_i^{\ast}\ker q'_0/tor\!s_i,
\end{eqnarray*}
what completes the proof.

\end{proof}
\begin{corollary} \label{sssh} Sheaves $\widetilde F_i=\sigma^{\ast}
F_i /tor\!s$ are semistable of rank $r_i=\rank F_i$ with reduced
Hilbert polynomial equal to  $p_E(t)$.
\end{corollary}
\begin{proof} The equality of ranks follows from the equality
$\widetilde F_i=\sigma^{\ast} F_i /tor\!s$ and from the fact that
the morphism  $\sigma$ is an isomorphism on open subscheme in
$\widetilde S$. The Hilbert polynomial for all $t\gg 0$ is fixed
by the equalities $\chi (\widetilde F_i \otimes \widetilde L^{t})=
h^0(\widetilde S, \widetilde F_i \otimes \widetilde L^{ t})=
h^0(S, F_i \otimes L^{ t} )= \chi (F_i \otimes L^{ t})=rp_E(t).$
\end{proof}

For $\widetilde E= \sigma^{\ast} E/tors$ consider epimorphisms
$\widetilde q: H^0(\widetilde S, \widetilde E \otimes \widetilde
L^m) \otimes \widetilde L^{(-m)}\twoheadrightarrow \widetilde E$
and  $q: H^0( S,
 E \otimes  L^m) \otimes
L^{(-m)}\twoheadrightarrow E$.
\begin{definition} Subsheaves  $\widetilde F \subset \widetilde E$ и $F \subset
E$ are called  {\it $\upsilon$-correspond\-ing } if there exist
subspaces  $\widetilde V \subset H^0(\widetilde S, \widetilde E
\otimes \widetilde L^m)$ and $V = \upsilon (\widetilde V)\subset
H^0( S, E \otimes L^m)$ such that  $\widetilde q(\widetilde V
\otimes \widetilde L^{-m})=\widetilde F$, $q(V \otimes L^{-m})=F$.
Notation: $F= \upsilon (\widetilde F).$ The corresponding quotient
sheaves  $\widetilde E /\widetilde F$ и $ E/F$ will be also called
$\upsilon$-{\it corresponding } and denoted
 $E/F=\upsilon(\widetilde E /\widetilde F).$\end{definition}

\begin{proposition} \label{satur} The transformation $E \mapsto
\sigma^{\ast}E /tors$ takes saturated subsheaves to saturated
subsheaves.
\end{proposition}
\begin{proof} Let $F_{i-1}\subset F_i$ be a saturated subsheaf.
Assume that the quotient sheaf $\widetilde F_i/ \widetilde
F_{i-1}$ has a subsheaf of torsion $\tau.$ This subsheaf is
generated by vector subspace  $\widetilde T \subset H^0
(\widetilde S, \widetilde F_i \otimes \widetilde
L^{m})/H^0(\widetilde S, \widetilde F_{i-1} \otimes \widetilde
L^{m}).$ Let $\widetilde T'$ be its preimage in $H^0(\widetilde S,
\widetilde F_i \otimes \widetilde L^{m})$ and $\TT' \subset
\widetilde F_i$ be a subsheaf generated by subspace $\widetilde
T'.$ Then there is a sheaf  epimorphism $\TT' \twoheadrightarrow
\tau$ with  kernel $\TT' \cap \widetilde F_{i-1}.$ Let $\widetilde
K= H^0(\widetilde S, (\TT' \cap \widetilde F_{i-1})\otimes
\widetilde L^{ m})\subset H^0(\widetilde S, \widetilde
F_{i-1}\otimes \widetilde L^{ m})$ be its generating subspace.
Then the isomorphism $\upsilon$ leads to the exact diagram of
vector spaces
\begin{equation*}\xymatrix{0\ar[r]& \upsilon(\widetilde K)
\ar[d]_{\wr}^{\upsilon} \ar[r]& \upsilon(\widetilde T')
\ar[d]_{\wr}^{\upsilon} \ar[r]& \upsilon(\widetilde
T')/\upsilon(\widetilde K)
\ar[d]_{\wr}^{\overline \upsilon} \ar[r]&0\\
0\ar[r]& \widetilde K \ar[r]& \widetilde T' \ar[r]& \widetilde T'
/\widetilde K \ar[r]&0}
\end{equation*}
with morphism $\overline \upsilon$ induced by the morphism
$\upsilon.$ Also there are exact sequences of
$\upsilon$-corresponding coherent sheaves
\begin{equation*}\xymatrix{
0\ar[r]&\upsilon(\TT' \cap \widetilde F_{i-1}) \ar[r]&
\upsilon(\TT') \ar[r]& \upsilon(\tau) \ar[r]&0,\\
0\ar[r]& \TT' \cap \widetilde F_{i-1} \ar[r]& \TT' \ar[r]& \tau
\ar[r]& 0.}
\end{equation*}
Sheaves $\TT' \cap \widetilde F_{i-1},$ $\upsilon(\TT' \cap
\widetilde F_{i-1}),$ $\TT'$ and $\upsilon(\TT')$ coincide under
restriction on open subsets  $W$ and $\sigma (W)$ respectively.
Then if $\tau$ is a torsion sheaf, then $\upsilon (\tau)$ is also
torsion sheaf. This contradicts saturatedness of the
subsheaf$F_{i-1}.$\end{proof}

\begin{corollary} There are isomorphisms $\sigma^{\ast}
(F_i/F_{i-1})/tor\!s \cong \widetilde F_i/ \widetilde F_{i-1}.$
\end{corollary}
\begin{proof} Take an exact triple
$$
0\to F_{i-1} \to F_i \to F_i/F_{i-1} \to 0
$$
and apply the functor $\sigma^{\ast}.$ This yields
$$
0 \to \sigma^{\ast} F_{i-1}/\tau \to \sigma^{\ast} F_i \to
\sigma^{\ast} F_{i}/F_{i-1}\to 0$$ where the symbol $\tau$ denotes
the subsheaf of torsion violating exactness. Factoring  first two
sheaves by torsion and applying the proposition \ref{corr} one has
an exact diagram
\begin{equation*}\xymatrix{&0\ar[d]&0\ar[d]&0\ar[d]&\\
0\ar[r]& N \ar[d]\ar[r]& tor\!s(\sigma^{\ast}F_i)\ar[r] \ar[d] &\tau'\ar[r] \ar[d]& 0\\
0\ar[r]& \sigma^{\ast} F_{i-1}/\tau \ar[r]\ar[d]& \sigma^{\ast}
F_i\ar[r] \ar[d]&
\sigma^{\ast}(F_i /F_{i-1}) \ar[r] \ar[d]&0\\
0\ar[r]& \widetilde F_{i-1} \ar[r] \ar[d]& \widetilde F_i \ar[r]
\ar[d]& \widetilde F_i/
\widetilde F_{i-1} \ar[r] \ar[d]&0\\
&0&0&0&}
\end{equation*}
where the sheaf $N$ is defined as $\ker (\sigma^{\ast}F_{i-1}/\tau
\to \widetilde F_{i-1})$. It rests to note that the sheaf $\tau'$
is torsion sheaf. Also since the subsheaf $F_{i-1}\subset F_i$ is
saturated then due to the proposition \ref{satur}, the sheaf
$\widetilde F_i/ \widetilde F_{i-1}$ has no torsion. Then
$\widetilde F_i/\widetilde F_{i-1} \cong \sigma^{\ast}(F_i
/F_{i-1}) /tor\!s$.
\end{proof}
\begin{corollary} \label{stabquot} Quotient sheaves $\widetilde F_i/\widetilde
F_{i-1}$ are stable and their reduced Hilbert polynomial is equal
to $p_E(t).$
\end{corollary}
\begin{proof} Consider a subsheaf $\widetilde R \subset
\widetilde F_i/\widetilde F_{i-1}$ and the space of global
sections $$H^0(\widetilde S, \widetilde R \otimes \widetilde
L^{m})\subset H^0(\widetilde S, (\widetilde F_i/\widetilde
F_{i-1})\otimes \widetilde L^{m})= H^0(\widetilde S, \widetilde
F_i \otimes \widetilde L^{m})/ H^0(\widetilde S, \widetilde
F_{i-1} \otimes \widetilde L^{m}).$$ We assume as usually $m$ to
be as big as higher cohomology groups vanish. Let $\widetilde H$
be the preimage of subspace $H^0(\widetilde S, \widetilde R
\otimes \widetilde L^{m})$ in $H^0(\widetilde S, \widetilde F_i
\otimes \widetilde L^{m})$, and $H:= \upsilon(\widetilde H)$.
Denote by $\HH$ a subsheaf in $F_i$ if this subsheaf is generated
by the subspace $H$. It is clear that $\HH /F_{i-1} \subset F_i
/F_{i-1}.$ From the chain of obvious equalities
\begin{eqnarray} h^0(\widetilde S, \widetilde R \otimes \widetilde L
^{m})=\dim \widetilde H - h^0(\widetilde S, \widetilde
F_{i-1}\otimes \widetilde L^{ m})=\dim H-h^0(S, F_{i-1} \otimes
L^{m})\nonumber\\
=h^0(S, \HH \otimes L^{m})- h^0(S, F_{i-1}\otimes L^{m})=h^0(S,
(\HH/F_{i-1}) \otimes L^{m})\nonumber \end{eqnarray} it follows
that
\begin{eqnarray*}\frac{h^0(\widetilde S, (\widetilde F_i
/\widetilde F_{i-1})\otimes \widetilde L^{m})} {\rank (\widetilde
F_i/\widetilde F_{i-1})}&-& \frac{h^0(\widetilde S,\widetilde R
\otimes \widetilde
L^{m})}{\rank \widetilde R}=\nonumber \\
\frac{h^0(S, (F_i/F_{i-1})\otimes L^{m})} {\rank (F_i/F_{i-1})}&-&
\frac{h^0(S,(\HH/F_{i-1}) \otimes L^{m})}{\rank (\HH/F_{i-1})} >0.
\end{eqnarray*}
This proves stability of the quotient sheaf $\widetilde F_i /
\widetilde F_{i-1}.$
\end{proof}

Now consider the exact triple $0\to E_1 \to E \to gr_1(E) \to 0$
and the corresponding triple of spaces of global sections
\begin{equation*} 0\to H^0(S, E_1\otimes L^{m})\to H^0(S, E\otimes
L^{m})\to H^0(S, gr_1(E)\otimes L^{m})\to 0.
\end{equation*}
It is exact for $m\gg 0.$ The transition to the corresponding
 $\OO_{\widetilde S}$-sheaf $\widetilde
E$, to its subsheaf $\widetilde E_1$, to global sections, and
application of the isomorphism $\upsilon$, lead to the commutative
diagram of vector spaces
\begin{equation*} \xymatrix{
0\ar[r]& H^0(S, E_1\otimes L^{m})\ar[r]& H^0(S, E\otimes
L^{m})\ar[r]&
H^0(S, gr_1(E)\otimes L^{m})\ar[r]& 0\\
0\ar[r]& H^0(\widetilde S, \widetilde E_1\otimes \widetilde
L^{m})\ar[r] \ar[u]_{\wr}^{\upsilon}& H^0(\widetilde S, \widetilde
E\otimes \widetilde L^{m})\ar[r]\ar[u]_{\wr}^{\upsilon}&
H^0(\widetilde S, gr_1(\widetilde E)\otimes \widetilde L^{
m})\ar[r]\ar[u]_{\wr}^{\overline \upsilon}& 0}
\end{equation*}
where the isomorphism $\overline \upsilon$ is induced by the
isomorphism $\upsilon.$ Continuing the reasoning inductively for
rest subsheaves of Jordan-H\"{o}lder filtration of the sheaf $E$,
we get that bijective correspondence of subsheaves is continued
onto quotients of filtrations. Then the transition from quotient
sheaves $E_i/E_{i-1}$ to  $\widetilde E_i/\widetilde E_{i-1}$
preserves Hilbert polynomials and stability.

\begin{definition} {\it Jordan--H\"{o}lder filtration}
of semistable vector bundle $\widetilde E$ with Hilbert polynomial
equal to $rp(t)$, is a sequence of semistable subsheaves $0
\subset \widetilde F_1 \subset \dots \subset \widetilde F_{\ell}
\subset \widetilde E$ with reduced Hilbert polynomials equal to
$p(t)$, such that quotient sheaves $gr_i(\widetilde E)=\widetilde
F_i/\widetilde F_{i-1}$ are stable.

The sheaf $\bigoplus_i gr_i(\widetilde E)$ will be called as {\it
associated polystable sheaf} for the bundle $\widetilde E.$
\end{definition}

Then it follows from the results of propositions \ref{corr},
\ref{satur} and corollaries \ref{sssh} -- \ref{stabquot} that the
transformation $E \mapsto \sigma^{\ast} E /tor\!s$ takes the
Jordan -- H\"{o}lder filtration of the sheaf $E$ to Jordan --
H\"{o}lder filtration of bundle $\sigma^{\ast} E /tor\!s.$

Let $(\widetilde S, \widetilde E)$ and $(\widetilde S', \widetilde
E')$ be semistable pairs.

\begin{definition} Semistable pairs  $(\widetilde S,
\widetilde E)$ and $(\widetilde S', \widetilde E')$ are called
{\it $M$-equivalent (monoidally equivalent)} if for morphisms of
$\diamond$-product $\widetilde S \diamond \widetilde S'$ to
factors $\overline \sigma': \widetilde S \diamond \widetilde S'
\to \widetilde S$ and $\overline \sigma: \widetilde S \diamond
\widetilde S' \to \widetilde S'$ and for associated polystable
sheaves  $\bigoplus_i gr_i (\widetilde E)$ and $\bigoplus_i gr_i
(\widetilde E')$ there are isomorphisms \begin{equation*}\overline
\sigma'^{\ast} \bigoplus_i gr_i (\widetilde E)/tor\!s \cong
\overline \sigma^{\ast} \bigoplus_i gr_i (\widetilde
E')/tor\!s.\end{equation*}
\end{definition}

\begin{proposition} S-equivalent semistable coherent sheaves $E$
and $E'$ correspond to M-equivalent semistable pairs $(\widetilde
S, \widetilde E)$ and $(\widetilde S', \widetilde E')$.
\end{proposition}
\begin{proof} Standard resolution takes semistable coherent sheaf
$E$ to semistable pair  $(\widetilde S, \widetilde E)$. Jordan --
H\"{o}lder filtration of the sheaf $E$ is taken to Jordan --
H\"{o}lder filtration of the bundle $\widetilde E.$ Then the
polystable sheaf $\bigoplus_i gr_i(E)$ is taken to the associated
polystable sheaf  $\bigoplus _i gr_i(\widetilde E)$. Hence we have
for the sheaf $E$
\begin{equation}\label{grE}\overline \sigma '^{\ast}\bigoplus_i
gr_i(\widetilde E)/tor\!s= \overline \sigma
'^{\ast}[\sigma^{\ast}\bigoplus_i gr_i(E)/tor\!s]/tor\!s =
\overline \sigma '^{\ast}\sigma^{\ast}\bigoplus_i
gr_i(E)/tor\!s.\end{equation} Analogously for a sheaf $E'$ which
is S-equivalent to the sheaf  $E$ one has
\begin{equation}\label{grE'}\overline \sigma ^{\ast}\bigoplus_i
gr_i(\widetilde E')/tor\!s=\overline \sigma
^{\ast}[\sigma'^{\ast}\bigoplus_i gr_i(E')/tor\!s]/tor\!s
=\overline \sigma ^{\ast} \sigma'^{\ast}\bigoplus_i
gr_i(E')/tor\!s.\end{equation} Right hand sides of (\ref{grE}) and
(\ref{grE'}) are isomorphic by the isomorphism of polystable
$\OO_S$-sheaves $\bigoplus_i gr_i(E)\cong \bigoplus_i gr_i(E')$
and by commutativity of diagram \begin{equation*}
\xymatrix{\widetilde S \diamond \widetilde S' \ar[d]_{\overline
\sigma'} \ar[r]^{\overline \sigma }&\widetilde S'
\ar[d]^{\sigma'}\\
\widetilde S \ar[r]_{\sigma}& S}\end{equation*} for
$\diamond$-product. The proposition is proven. \end{proof}

\section{ Minimal resolution in the monoid $\diamondsuit[E]$}

This section plays auxiliary role. We prove results concerning
with local freeness of subsheaves and quotient sheaves in  Jordan
-- H\"{o}lder filtration for sheaves of the form $\widetilde
E=\sigma^{\ast} E/tors$.
\begin{definition} $\overline \sigma:\overline S_{\ast}
\to S$ is the {\it minimal resolution} in the monoid $\diamondsuit
[E]$ generated by canonical morphisms for the class $[E]$ of
S-equivalent semistable coherent
$\OO_S$-sheaves, if the following hold:\\
i) (resolution) for any canonical morphisms $\sigma_i: \widetilde
S_i \to S,$ $\sigma_j: \widetilde S_j \to S$ diagrams
\begin{equation*}\xymatrix{\overline S_{\ast} \ar[d]_{\overline \sigma_i}
\ar[rd]_{\overline \sigma} \ar[r]^{\overline \sigma_j}& \widetilde S_i \ar[d]^{\sigma_i}\\
\widetilde S_j \ar[r]_{\sigma_j}& S}
\end{equation*}
commute. All sheaves $\overline \sigma ^{\ast} E_i/ tor\!s$ for
$\overline \sigma= \sigma_j \circ \overline \sigma_i= \sigma _i
\circ \overline \sigma_j$ are
locally free; \\
ii) (minimality) for any morphism  $\sigma'\in \diamondsuit [E]$,
$\sigma': \widetilde S' \to S $ such that  i) holds, there exists
a morphism $f: \widetilde S' \to \widetilde S_{\ast}$ which
includes into commutative diagrams
\begin{equation*}\xymatrix{\widetilde S' \ar[ddr] \ar[dr]_{f\!\!\!} \ar[drr]\\
& \overline S_{\ast}\ar[d]^{\overline \sigma_i}
\ar[rd]_>>>>>>{\overline \sigma}\ar[r]_{\overline \sigma_j}&
\widetilde
S_i\ar[d]^{\sigma_i}\\
& \widetilde S_j \ar[r]_{\sigma_j}& S}\end{equation*} for all
$i,j$, and $\sigma'=\overline \sigma \circ f$.
\end{definition}

\begin{remark} We naturally assume that for one-element
class defined by a stable sheaf $E$ the corresponding canonical
morphism $\sigma$ is minimal.
\end{remark}

\begin{remark}\label{diamond} For a pair of semistable
coherent sheaves  $E_1$ and $E_2$, and for cor\-responding
resolutions  $\sigma_1: \widetilde S_1 \to S$ and $\sigma_2:
\widetilde S_2 \to S$ the $\diamond$-product $\widetilde S_1
\diamond \widetilde S_2$ satisfies conditions of resolution and of
minimality. This implies that it is the minimal resolution for the
pair of sheaves $E_1$, $E_2$.
\end{remark}

\begin{proposition} Every class of S-equivalent coherent
$\OO_S$-sheaves has the \linebreak minimal resolution $\overline
S_{\ast}$. This resolution corresponds to the morphism $\overline
\sigma: \overline S_{\ast} \to S$.
\end{proposition}
\begin{proof} It is enough to confirm that there is a finite
collection of canonical morphisms $\sigma: \widetilde S \to S$
corresponding to sheaves of each S-equivalence class. This means
that the minimal resolution $\overline S_{\ast}$ can be
constructed using the operation $\diamond$ by finite number of
steps.

Indeed, every S-equivalence class consists of sheaves with
singularities supported at the same points of the surface $S$.
Namely, $\Supp \EExt^1(E, \OO_S)$ is constant in the S-equivalence
class. Further, colengths of zeroth Fitting ideal sheaves
\linebreak
 $\FFitt^0 \EExt^1(E,\OO_S)= \FFitt^0 \EExt^2(\varkappa,
\OO_S)$ are bounded from above globally over $\overline M.$ Let
$l_0$ be the value of ${\rm colength\;} \FFitt^0
\EExt^2(\varkappa, \OO_S)$ maximal over $\overline M$.

Now turn to the Grothendieck's Quot scheme $\Quot^l \bigoplus^r
\OO_S$ para\-meterizing \linebreak zero-dimensional quotient
sheaves of length $l$ on the surface  $S$. It is well-known that
this is a Noetherian algebraic scheme of finite type. It has
natural stratification defined by isomorphy classes of
zero-dimensional quotient sheaves.
\begin{claim} The number of strata is finite.
\end{claim}
Indeed, consider an epimorphism  $q: \OO_S^{\oplus r}
\twoheadrightarrow \varkappa$ and an inclusion of (any) direct
summand $\OO_S \hookrightarrow \OO_S^{\oplus r}.$ Denote
$\varkappa_r:=\varkappa.$ Then $\OO_{Z_1}$ is the image of this
direct summand under the composite map $\OO_S \hookrightarrow
\OO^{\oplus r} \twoheadrightarrow \varkappa_r.$ There is an exact
diagram
\begin{equation*}\xymatrix{&0\ar[d]& 0\ar[d]& 0\ar[d]\\
0\ar[r]& I_{Z_1} \ar[d]\ar[r]& K_r\ar[d] \ar[r]& K_{r-1}\ar[d] \ar[r]&0\\
0\ar[r]& \OO_S\ar[d] \ar[r]& \OO_S^{\oplus r}
\ar[d] \ar[r]& \OO_S^{\oplus (r-1)} \ar[d] \ar[r]& 0\\
0\ar[r]& \OO_{Z_1} \ar[d]\ar[r]& \varkappa_r \ar[d] \ar[r]&
\varkappa_{r-1} \ar[d] \ar[r]& 0\\
&0&0&0}
\end{equation*}
Now, performing the procedure for the right vertical triple and
iterating the process we come to the cofiltration
$$
\varkappa_r \twoheadrightarrow \varkappa_{r-1}\twoheadrightarrow
\dots \twoheadrightarrow \varkappa_2 \twoheadrightarrow
\varkappa_1=\OO_{Z_r}
$$
with kernels $\OO_{Z_1}, \dots, \OO_{Z_{r-1}}.$ Namely, we have a
series of short exact sequences
\begin{eqnarray*}
0\to \OO_{Z_1} \to \varkappa_r \to \varkappa _{r-1} \to 0, \nonumber\\
0\to \OO_{Z_2} \to \varkappa _{r-1} \to \varkappa_{r-2}\to 0, \nonumber\\
.\;.\;.\;.\;.\;.\;.\;.\;.\;.\;.\;.\;.\;.\;.\;.\;.\;.\;.\;.\;.\;.\;.\nonumber \\
0\to \OO_{Z_{r-1}} \to \varkappa_2 \to \OO_{Z_r} \to 0. \nonumber
\end{eqnarray*}
Then every Artinian sheaf of $\OO_S$-modules, if it is a quotient
sheaf of the sheaf $\OO_S^{\oplus r}$, has a (possibly non-unique)
representation by a sequence of zero-dimensional subschemes $Z_1,
\dots, Z_r$ and a sequence of vectors $\varkappa_2 \in
\Ext^1_{\OO_S}(\OO_{Z_r}, \OO_{Z_{r-1}}),$ $\varkappa_3 \in
\Ext^1_{\OO_S}(\varkappa_2,\OO_{Z_{d-2}}), \dots, \varkappa_r\in
\Ext^1_{\OO_S}(\varkappa_{r-1}, \OO_{Z_1}).$ Here the same symbol
denotes for brevity the element of Ext-group as well as the
extension sheaf itself. It rests to note that the following sets
are finite:

the set of ordered partitions $l_1, l_2, \dots, \l_r$,
$l_1+l_2+\dots +l_r=l,$ of length $r$ for the integer $l$
(appearance of zeroes $l_i=0$ for some $i$ admitted);

the set of isomorphy classes of zero-dimensional subschemes of
length $l_i$ on $k$-scheme $S$ of finite type;

the set of isomorphy classes of extensions which are described by
the elements of the group $\Ext^1_{\OO_S}(\varkappa _i,
\OO_{Z_{r-i}})$.

The finiteness of these sets implies the finiteness of the number
of isomorphy classes of quotient sheaves in $\Quot^l \OO_S^{\oplus
r}$. This proves the claim.

The number of isomorphy classes of schemes $\widetilde S$
corresponding to the sheaves in the class $[E]$, does not exceed
the number of isomorphy classes of Artinian sheaves in the union
$\bigsqcup _{l\le l_0}\Quot^l \bigoplus^r\OO_S$. The last is
finite, as follows from the previous reasoning. Then there is a
product  $\widetilde S_{\ast}=\diamond \widetilde S$ over all
isomorphy classes in the given S-equivalence class. By Remark
\ref{diamond}, this product satisfies minimality condition and
hence $\widetilde S_{\ast}=\overline S_{\ast}$. This completes the
proof.
\end{proof}

\begin{proposition} If $\overline \sigma: \overline
S_{\ast}\to S$ is the minimal resolution in the monoid
$\diamondsuit [E]$ then for any  $E\in [E]$ images $\overline
F_i:= \overline \sigma^{\ast}F_i/tors$ of the sheaves $F_i$ in
Jordan -- H\"{o}lder filtration are locally free and $\overline
\sigma^{\ast} gr_i (E)= \overline F_i /\overline F_{i-1}.$
\end{proposition}

\begin{proof} Consider polystable sheaf $\bigoplus_i gr_i(E)$
which belongs to the given S-equi\-valence class. Its resolution
on the scheme $\overline S_{\ast}$ also has a form of direct sum
$\bigoplus _i \overline \sigma^{\ast} gr_i(E)/tor\!s$. Hence the
summands  $\overline \sigma^{\ast} gr_i(E)/tor\!s$ are also
locally free. The kernel of an epi\-morphism of locally free
sheaves is locally free sheaf. Then for any semistable sheaf $E$
in the given S-equivalence class, with Jordan -- H\"{o}lder
filtration $0\subset F_1 \subset \dots \subset F_{l-1}\subset
F_l=E$ the sheaves if the image of the filtration $\overline
F_i=\overline \sigma^{\ast} F_i /tor\!s$ are locally free. Then we
have for quotients of the filtration: $\overline \sigma^{\ast}
gr_i(E)/tor\!s=\overline F_i/\overline F_{i-1}.$
\end{proof}

We turn again to the polystable sheaf $gr(E)= \bigoplus_i gr_i(E)$
in the given S-equival\-ence class. Let $\sigma_{gr}: \widetilde
S_{gr} \to S$ be the corresponding canonical morphism defined by
the sheaf of ideals $I_{gr}=\FFitt^0 \EExt^1(gr(E), \OO_S)$.
\begin{proposition} For all $E \in [E]$ sheaves
$\sigma_{gr}^{\ast}E/tor\!s$ are locally free.
\end{proposition}
\begin{proof}
The sheaf $\sigma_{gr}^{\ast} gr(E)/tor\!s$ is locally free. This
implies that all direct summands  $\sigma_{gr}^{\ast}
gr_i(E)/tor\!s$ are also locally free. Consider inductively
following exact sequences
\begin{eqnarray}0\to \sigma_{gr}^{\ast} F_1/tor\!s
=\!\!\!= \sigma_{gr}^{\ast} gr_1(E)/tor\!s\to 0,\nonumber\\
0\to \sigma_{gr}^{\ast} F_1/tor\!s\to \sigma_{gr}^{\ast}
F_{2}/tor\!s \to \sigma_{gr}^{\ast} gr_{2}(E)/tor\!s \to 0,\nonumber\\
.\;.\;.\;.\;.\;.\;.\;.\;.\;.\;.\;.\;.\;.\;.\;.\;.\;.\;.\;.\;.\;.\;.\;.\;.\;.\;.\;.\;.\;.\;.\;.\;.\;.\;.\;.\;.\;.\;.\;.\;.\;.\;.\;.\; \nonumber\\
0\to \sigma_{gr}^{\ast} F_{l-1}/tor\!s \to \sigma_{gr}^{\ast}
E/tor\!s \to \sigma_{gr}^{\ast} gr_l(E)/tor\!s\to 0,\nonumber
\end{eqnarray}
induced by Jordan -- H\"{o}lder filtration of any semistable sheaf
 $E$ of the given S-equi\-valence class. It follows that the sheaf
$\sigma_{gr}^{\ast} E/tor\!s$ is locally free. \end{proof}

\section{ Boundedness of families of semistable pairs}
In this section we show that pairs of the form $((\widetilde S,
\widetilde L),\widetilde E)$, which are deformation equivalent to
pairs with $(\widetilde S, \widetilde L)\cong (S,L),$ constitute
in  $\Hilb^{P(t)}G(V,r)$ a subset equal to $\mu(\widetilde Q).$

\begin{definition} A pair $((\widetilde S, \widetilde L),
\widetilde E)$ is called {\it S-pair} if  $(\widetilde S,
\widetilde L)\cong (S,L)$.
\end{definition}

\begin{definition} A pair $((\widetilde S, \widetilde L),\widetilde
E)$ is {\it deformation equivalent to S-pairs} if there exist
\begin{itemize}
\item{ connected algebraic scheme  $T$ of dimension 1,}
\item{ flat family of schemes $\pi: \Sigma \to T$,}
\item{invertible sheaf $\widetilde
\L$ very ample relative $\pi$ ,}
\item{ locally free  $\OO_{\Sigma}$-sheaf $\widetilde \E$}
\end{itemize}
such that
\begin{itemize}
\item{for any closed point  $t\in T$ the fibre $\pi^{-1}(t)$ is admissible scheme,
$\widetilde \L|_{\pi^{-1}(t)}$ is its distinguished polarization,
$\widetilde \E|_{\pi^{-1}(t)}$ is locally free sheaf with Hilbert
polynomial equal to $\chi(\widetilde \E|_{\pi^{-1}(t)} \otimes
\widetilde \L^t|_{\pi^{-1}(t)})=rp_E(t),$}
\item{there is a point  $t_0 \in T$ such that  $((\pi^{-1}(t_0), \widetilde \L|_{\pi^{-1}(t_0)}),
\widetilde \E|_{\pi^{-1}(t_0)})\cong ((\widetilde S, \widetilde
L), \widetilde E),$}
\item{at every general point  $t \ne t_0$ $((\pi^{-1}(t),\widetilde \L|_{\pi^{-1}(t)}),
\widetilde \E|_{\pi^{-1}(t)})$ is S-pair.}
\end{itemize}
\end{definition}
Pairs deformation equivalent to S-pairs will be called {\it
dS-pairs}. A subset of points in $\Hilb^{P(t)}G(V,r)$
corresponding to $dS$-pairs  is denoted by the symbol $K^{dS}$. A
subset of points corresponding to $S$-pairs is denoted by the
symbol $K^S.$

Now we need  the following well-known
\begin{theorem}\label{bounded}\cite[3.3.7]{HL}
Let $f: X \to Y$ be projective morphism of schemes of finite type
over $k$ and let $\OO_X(1) $ be an $f$-ample line bundle. Let $P$
be a polynomial of degree $d$, and let $\mu_0$ be a rational
number. Then the family of purely $d$-dimensional sheaves on the
fibres of $f$ with Hilbert polynomial $P$ and maximal slope
$\mu_{max}\ge \mu_0$ is bounded.
\end{theorem}
\begin{remark} The theorem is formulated for more general
case of purely $d$-dimen\-sional sheaves and the symbol
$\mu_{max}$ denotes maximal slope in the Harder --\linebreak
Narasimhan filtration of the sheaf $E$. Since we are interested in
semistable sheaves, the Harder -- Narasimhan filtration for
semistable sheaf consists of this sheaf itself. Then
$\mu_{max}=\mu$ is constant.
\end{remark}

Note that by the reasonings of section 6, semistable $dS$-pairs
carry locally free sheaves.  Their slope is constant in
Jordan-H\"{o}lder filtration and therefore is bounded. Then
applying the theorem \ref{bounded} in this situation we have the
following
\begin{proposition} {Family of semistable  $dS$-pairs is
bounded.} \end{proposition}

Then there is an integer $\widetilde m_0$ such that for
$m>\widetilde m_0$ sheaves $\widetilde E \otimes \widetilde L^m$
define closed immersions of the corresponding schemes  $\widetilde
S$ into Grassmann variety $G(V,r)$ for $V$ being $k$-vector space
of dimension $rp_E(m)$. The integer $\widetilde m_0$ is uniform
for all semistable $dS$-pairs. Therefore in the whole construction
of the moduli scheme one must take  $m> \max (m_0, \widetilde
m_0).$ The number  $m_0$ is characterized before, in section 2.

\begin{proposition} $S$-pairs constitute an open subscheme in $K^{dS}$.
\end{proposition}
\begin{proof} Note that the set of points of Hilbert scheme where fibres of the morphism  $$\pi : \Univ^{P(t)}G(V,r)\to
\Hilb^{P(t)}G(V,r)$$ correspond to $S$-pairs, is precisely an
intersection of $K^{dS}$ with the set of points of the scheme
$\Hilb^{P(t)}G(V,r)$ where the fibres of the morphism $\pi$ are
geometrically integral. The second set is open due to
\cite[Th\'{e}or\-\`{e}me 12.2.4(viii)]{EGAIV3}. Restriction on the
subscheme $K^{dS}$ proves the proposition.
\end{proof}

It is clear from reasonings in the previous sections that
$\mu(\widetilde Q)\subset K^{dS}.$ Both subsets are
$PGL(V)$-invariant.
\begin{proposition} {There is a coincidence of subsets  $\mu(\widetilde Q) = K^{dS}$.}
\end{proposition}
\begin{proof} Assume that  $K^{dS}\setminus \mu(\widetilde Q)
\ne \emptyset.$ Note that subsets formed by semistable $S$-pairs
in $\mu(\widetilde Q)$ and in $K^{dS}$ coincide and equal $K^S$.
Let $T=\Spec k[t]$ be an affine curve in $K^{dS}$ such that its
open part  $T_0=\{t\ne t_0\}\subset T$ belongs to the subset $K^S$
but special point  $t_0=T\setminus T_0\in K^{dS}\setminus
\mu(\widetilde Q).$ Let $\pi_1: \widetilde \Sigma_1\to T$ be the
flat family of subschemes defined by the fibred product
$\widetilde \Sigma_1=T\times_ {\Hilb^{P(t)}G(V,r)}
\Univ^{P(t)}G(V,r)$, carrying the locally free sheaf  $\widetilde
\E_1$, and corresponding to the immersion $T\hookrightarrow
K^{dS}\subset \Hilb^{P(t)}G(v,r)$. Then we have flat family of
schemes  $\pi_1: \widetilde \Sigma_1 \to T$, supplied with locally
free sheaf
 $\widetilde \E_1$, with birational
morphisms $\xi_1: \widehat\Sigma_1 \to \widetilde \Sigma_1$ and
$\sigma \!\!\! \sigma_1: \widehat \Sigma_1 \to T \times S$,
forming a "Hironaka's house". Denote by the symbol $\E_2$ the
reflexive sheaf  $(\sigma \!\!\! \sigma_{1\ast } \xi_1^{\ast}
\widetilde \E_1)^{\vee \vee}=\E_2$ on the product $T \times S$. It
is locally free on the open subset off the codimension 3. Due to
 \cite[corollary 6.3]{Eisen}, the module over the principal ideal domain
is flat if and only if it is torsion-free. Then $\E_2$ is the
sheaf of flat  $\OO_T$-modules. It is a subsheaf of locally free
$\OO_T$ -module  $\FF$ of the same rank. Namely, there is
inclusion $\E_2 \hookrightarrow \FF$. Let  $\alpha \!\!\! \alpha:=
\FF /\E_2$. It is clear that by choice of the determinant of the
sheaf $\FF$ we can achieve  $\alpha \!\!\! \alpha$ to be Artinian
sheaf with support on the fibre  $t_0\times S$: $\alpha \!\!\!
\alpha|_{t_0\times S}\cong \alpha \!\!\!\alpha.$ Applying to the
exact triple of  $\OO_{T \times S}$-sheaves
$$ 0\to \E_2 \to \FF \to \alpha \!\!\! \alpha \to 0$$
the restriction onto the fibre  $t_0\times S$, we get an exact
triple
$$0\to \TTor_1^{T\times S}(\alpha \!\!\! \alpha, k_{t_0}\boxtimes
\OO_S) \to \E_2|_{t_0\times S} \to \FF|_{t_0\times S}\to \alpha
\!\!\! \alpha|_{t_0\times S} \to 0.
$$
Since  $\alpha \!\!\! \alpha $ is a sheaf supported on the fibre
$t_0 \times S,$ then $\TTor_1^{T\times S}(\alpha \!\!\! \alpha,
k_{t_0}\boxtimes \OO_S)=\TTor_1^{S}(\alpha \!\!\! \alpha,
\OO_S)=0$ and $\E_2$ is a flat family of torsion-free sheaves on
the smooth surface. Due to \cite[proof of the proposition
 4.3]{O'Gr}, $\E_2$ is $\OO_{S\times
T}$-sheaf of homological dimension equal to  1. Then it is subject
to standard resolution.

\begin{remark} By the construction of the special fibre
$\pi^{-1}(t_0)=\widetilde S_1$ one can assume that
$\widehat\Sigma_1=\widetilde \Sigma_1$ and $\xi_1$ is an identity
morphism.
\end{remark}

Application of standard resolution to the sheaf  $\E_2$ leads to a
family of schemes $\pi_2:\widehat \Sigma_2 \to T$ which is flat
over  $T$. It is supplied with the locally free sheaf $\widetilde
\E_2$. Also the procedure or resolution gives the birational
morphism  $\sigma \!\!\! \sigma_2: \widehat \Sigma _2 \to T \times
S$ such that  $(\sigma \!\!\! \sigma_{2\ast} \widetilde
\E_2)^{\vee \vee}=\E_2.$ Applying the diagram (\ref{gldiamond})
one gets locally free sheaves $\sigma \!\!\!
\sigma'^{\ast}_1\widetilde \E_2$ and $\sigma \!\!\!
\sigma'^{\ast}_2\widetilde \E_1$. Denoting  $\sigma \!\!\!
\sigma:= \sigma \!\!\! \sigma'_1 \circ \sigma\!\!\! \sigma_2=
\sigma \!\!\! \sigma'_2 \circ \sigma\!\!\! \sigma_1$ one has
obvious equalities  $(\sigma \!\!\! \sigma_{\ast} \sigma \!\!\!
\sigma'^{\ast}_1\widetilde \E_2)^{\vee \vee}= (\sigma \!\!\!
\sigma_{\ast} \sigma \!\!\! \sigma'^{\ast}_2\widetilde \E_1)^{\vee
\vee}=\E_2$.

Note that on open subsets $\pi_1^{-1}T_0\cong \pi_2^{-1}T_0$
sheaves $\sigma \!\!\! \sigma'^{\ast}_1\widetilde \E_2$ and
$\sigma \!\!\! \sigma'^{\ast}_2\widetilde \E_1$ are iso\-morphic.
Denote $\pi_i^{-1} (t_0)=:\widetilde S_i$ for $i=1,2$ and,
respectively, $\widetilde E_i:= \widetilde \E_i |_{\widetilde
S_i}.$ The proof of the following lemma will be given later.

\begin{lemma}\label{coincide} Pairs  $(\widetilde S_1, \widetilde E_1)$
and $(\widetilde S_2, \widetilde E_2)$ are isomorphic.
\end{lemma}
The result of the lemma contradicts the assumption
$K^{dS}\setminus \mu(\widetilde Q)\ne \emptyset$ and completes the
proof.
\end{proof}
\begin{proof}[of the lemma \ref{coincide}] Note that locally
free sheaves  $\sigma \!\!\! \sigma'^{\ast}_1\widetilde \E_2
\otimes \det (\sigma \!\!\! \sigma'^{\ast}_1\widetilde
\E_2)^{\vee}$ and $\sigma \!\!\! \sigma'^{\ast}_2\widetilde \E_1
\otimes \det (\sigma \!\!\! \sigma'^{\ast}_2\widetilde
\E_1)^{\vee}$ coincide on the open subset off the codimension not
less then 2. Hence they coincide on the whole of the scheme
$\widehat \Sigma_{12}$, namely $\sigma \!\!\!
\sigma'^{\ast}_1\widetilde \E_2 \otimes \det (\sigma \!\!\!
\sigma'^{\ast}_1\widetilde \E_2)^{\vee}=\sigma \!\!\!
\sigma'^{\ast}_2\widetilde \E_1 \otimes \det (\sigma \!\!\!
\sigma'^{\ast}_2\widetilde \E_1)^{\vee}.$ Also note that the left
hand part of this equality is a locally free sheaf and it is
trivial along the exceptional divisor of the morphism of blowing
up $\sigma \!\!\! \sigma'_1$. Analogously, the sheaf on the right
is trivial along the exceptional divisor of the morphism of
blowing up  $\sigma \!\!\! \sigma'_2$. This implies that there
exist a scheme  $\widehat \Sigma_0$, a pair of birational
morphisms $\widehat \Sigma_2 \stackrel{\eta_2}{\longrightarrow}
\widehat \Sigma_0 \stackrel{\eta_1}{\longleftarrow} \widehat
\Sigma_1$ and a locally free sheaf $\widetilde \E_0$ on the scheme
$\widehat \Sigma_0$ such that  $\widetilde \E_i \otimes \det
\widetilde \E_i^{\vee}=\eta_i^{\ast} \widetilde \E_0$ for $i=1,2.$

We reduce our consideration to the case with trivial determinant.
Introduce an auxiliary notation $\widetilde E'_i:=\widetilde \E_i
\otimes \det \widetilde \E_i^{\vee}$ for $i=1,2$. Then $\widetilde
E'_i=\eta_i^{\ast} \widetilde \E_0.$ Note that sheaves $E$ and
$E'=E \otimes \det E^{\vee}$ have equal sheaves of singularities
$\varkappa=\varkappa'$. Then along additional components of
reducible fibres $\widetilde \Sigma_i$ we have  $$\widetilde
E'_i|_{\rm add}=\sigma^{\ast}\ker (\oplus^r
\OO_S\twoheadrightarrow\varkappa)|_{\rm add}/tor\!s.
$$
 This means that sheaves $\widetilde E'_i$ are nontrivial along
additional components of schemes $\widetilde S_i$. Since
$\widetilde \E_0$ is locally free sheaf then $\eta_i$ are identity
morphisms.

Hence morphisms $\sigma \!\!\! \sigma_i$ coincide and we will use
 notations
$\widehat \Sigma:=\widehat \Sigma_1=\widehat \Sigma_2$ and $\sigma
\!\!\! \sigma: \widehat \Sigma \to T\times S$. Also there are two
locally free sheaves $\widetilde \E_i$, $i=1,2,$ satisfying the
condition  $\widetilde \E_1 \otimes \det \widetilde
\E_1^{\vee}=\widetilde \E_2 \otimes \det \widetilde \E_2^{\vee}$.
Determinants $\det \widetilde \E_i$ coincide on the supplement of
the exceptional divisor of the morphism $\sigma \!\!\! \sigma$,
and $(\sigma \!\!\! \sigma_{\ast} \det \widetilde \E_i)^{\vee
\vee}= \det \E_2$. Then sheaves $\widetilde E_i=\widetilde
\E_i|_{\widetilde S}$ can differ only on additional components of
the scheme $\widetilde S$. Now note that the special fibre of the
projection $\widetilde S=\pi^{-1}(t_0)$ is a projective spectrum
of  $\OO_S$-algebra. Let $\OO(1)$ be its twisting Serre's sheaf.
Then  $\widetilde E_1=\widetilde E_2 \otimes \OO(l_1-l_2)$ for
some integers  $l_1, l_2$.

Let $l_1-l_2\ge 0$ (the opposite case can be considered
similarly). By the definition of the distinguished polarization
one has $\widetilde L^m=\sigma ^{\ast} L^m \otimes \OO(1).$ There
is an inclusion of invertible sheaves $\sigma ^{\ast} L^m \otimes
\OO(1) \hookrightarrow \sigma ^{\ast} L^m.$ Then there is an
inclusion \linebreak $\widetilde L^m \otimes
\OO(l_1-l_2)\hookrightarrow \widetilde L^m.$ By local freeness of
the sheaf  $\widetilde E_2$ there is an exact triple
\begin{equation}\label{compare} 0\to \widetilde E_2\otimes
\widetilde L^m \otimes \OO(l_1-l_2) \to \widetilde E_2\otimes
\widetilde L^m \to Q \otimes \widetilde E_2\otimes \widetilde L^m
\to 0,
\end{equation}
where  $Q$ is a quotient sheaf supported on additional components
of the scheme $\widetilde S$. For $m\gg 0$ the sequence of spaces
of global sections associated with (\ref{compare}) is exact
\begin{eqnarray} 0\to H^0(\widetilde S,\widetilde
E_2\otimes \widetilde L^m \otimes \OO(l_1-l_2)) \to H^0(\widetilde
S,\widetilde E_2\otimes \widetilde L^m )\nonumber \\
\to H^0(\widetilde S,Q \otimes \widetilde E_2\otimes \widetilde
L^m )\to 0\nonumber
\end{eqnarray}

Since $\widetilde E_2 \otimes \OO(l_1-l_2)=\widetilde E_1$ and
Hilbert polynomials of restrictions $\widetilde \E_1|_{\widetilde
S}$ and  $\widetilde \E_2|_{\widetilde S}$ equal (because sheaves
 $\widetilde \E_i$ are flat over the base $T$ of the family
$\widetilde \Sigma$ and coincide over $T_0$) then $h^0(\widetilde
S,\widetilde E_2\otimes \widetilde L^m \otimes \OO(l_1-l_2))=
h^0(\widetilde S,\widetilde E_2\otimes \widetilde L^m).$ This
implies  that  $\chi(Q \otimes \widetilde E_2\otimes \widetilde
L^m)=0$ for all $m\gg 0$. Hence $Q=0$, and $l_1=l_2.$ The triple
(\ref{compare}) implies that $\widetilde E_1 =\widetilde E_2.$
\end{proof}

\section{$PGL(V)$-actions, GIT-stability and GIT-quot\-ients}

In this section we analyze the numerical Hilbert -- Mumford
criterion for an appropriate  PGL(V)-linearized ample invertible
vector bundle on $\mu (\widetilde Q) $. This shows that good
PGL(V)-quotient $\mu (\widetilde Q)/PGL(V)$ is defined in the
category of algebraic schemes over the field $k$. It turns out
that every $PGL(V)$-orbit in  $\mu (\widetilde Q)$ contains at
least one $PGL(V)$-semistable point.

Let $\SS$ be the universal quotient bundle on the Grassmannian
$G(V.r)$, as usually $\OO_{G(V,r)}(1)$ is the positive generator
in its Picard group. We use following notations for projections of
the universal subscheme $\Hilb^{P(t)}G(V,r)
\stackrel{\pi}{\longleftarrow } \Univ^{P(t)}G(V,r)
\stackrel{\pi'}{\longrightarrow} G(V,r)$. Form following
 sheaves on the Hilbert scheme $\widetilde L_l^h=\det
\pi_{\ast} \pi'^{\ast} \SS (l)$. Since the projection
 $\pi: \Univ^{P(t)}G(V,r) \to \Hilb^{P(t)}G(V,r)$ is a flat
 morphism and sheaves $\SS(l)$ are locally free, then sheaves
$\widetilde L_l^h$ are invertible.

\begin{proposition} Sheaves $\widetilde L_l^h$ are very ample for $l\gg 0$.
\end{proposition}
\begin{proof} Due to \cite[Proposition 2.2.5]{HL}, for a projective scheme $X$,
ample invert\-ible sheaf $L$ on $X$, and Hilbert scheme
$\Hilb^{P(t)} X$ with universal scheme \linebreak $\Hilb^{P(t)}X
\stackrel{\pi}{\longleftarrow} \Univ^{P(t)}X
\stackrel{\pi'}{\longrightarrow}X$, sheaves $\det \pi_{\ast}
\pi'^{\ast}L^n$ are very ample whenever $n\gg 0$. Replace $L$ with
its big enough tensor power so that $\det \pi_{\ast} \pi'^{\ast}L$
is very ample. Take $X=G(V,r)$ and $L=\det \SS\otimes
\OO_{G(V,r)}(l')$ with $l'$ so big as $L$ is ample. Then the sheaf
$\det \pi_{\ast} \pi'^{\ast}L=\det \pi_{\ast} \pi'^{\ast}(\det
\SS\otimes \OO_{G(V,r)}(l'))=\det \pi_{\ast} \pi'^{\ast} \SS
(l)=\widetilde L^h_l$ for the appropriate $l\gg 0$ is very ample.
\end{proof}

Consider the action of linear algebraic group $GL(V)$ on the
vector space $V$ by its linear transformations. Then the induced
actions of the group $PGL(V)$ on Grassmann variety $G(V,r)$ and on
Hilbert scheme $\Hilb^{P(t)}G(V,r)$ are defined. The subscheme
$\mu(\widetilde Q)$ remains $GL(V)$-invari\-ant. Fix the notation
$\widetilde L_l:=\widetilde L_l^h |_{\mu(\widetilde Q)}.$

We remind the following

\begin{definition} \cite[definition 4.2.5]{HL} Let $X$ be a
$k$-scheme of finite type, $G$ an algebraic $k$-group, and
$\alpha: X\times G \to X$ -- group action. {\it $G$-linearization}
of a quasicoherent  $\OO_X$-sheaf $F$ is an isomorphism of
$\OO_{X\times G}$-sheaves $\Lambda: \alpha^{\ast}F \to pr_1^{\ast}
F$ where $pr_1: X\times G \to X$ is the projection and the
following cocycle condition holds:
$$
(\id_X\times mult)^{\ast} \Lambda =pr_{12}^{\ast} \Lambda \circ
(\alpha \times \id_G)^{\ast} \Lambda.
$$
Here $pr_{12}: X\times G \times G \to X\times G$ is a projection
onto first two factors, $mult: G\times G\to G $ is a morphism of
group multiplication in $G$.
\end{definition}

\begin{proposition}
Sheaves $\widetilde L_l$ carry $GL(V)$-linearization.
\end{proposition}
\begin{proof} Let $\gamma: G(V,r) \times GL(V) \to G(V,r)$ be the morphism
of the action of the group $GL(V)$. The universal quotient bundle
$\SS(l)$ carries $GL(V)$-linearization $\Lambda: \gamma^{\ast}
\SS(l) \stackrel{\sim}{\to} pr_1^{\ast} \SS(l)$. The linearization
is induced by the epimorphism $V \otimes
\OO_{G(V,r)}(l)\twoheadrightarrow \SS(l)$. Now consider the
morphism of  $GL(V)$-action \begin{equation*}\alpha:
\Hilb^{P(t)}G(V,r) \times GL(V) \to
\Hilb^{P(t)}G(V,r)\end{equation*} induced by the action $\gamma$.
For $l\gg 0$ the following chain of isomorphisms holds:
\begin{eqnarray}\label{chain} \alpha^{\ast}\widetilde L_l=\det
\alpha^{\ast}(\pi_{\ast}\pi'^{\ast}\SS(l))|_{\mu(\widetilde
Q)}=\nonumber\\
=\det \pi_{\ast}\pi'^{\ast}\gamma^{\ast}\SS(l)|_{\mu(\widetilde
Q)} \stackrel{\det (\pi_{\ast}\pi'^{\ast}\Lambda|\mu(\widetilde
Q))}{-\!\!\!-\!\!\!-\!\!\!-\!\!\!-\!\!\!-\!\!\!-\!\!\!-\!\!\!-\!\!\!\longrightarrow}
\det
\pi_{\ast}\pi'^{\ast}pr_1^{\ast}\SS(l)|_{\mu(\widetilde Q)}=\nonumber\\
=\det pr_1^{\ast} \pi_{\ast} \pi'^{\ast}\SS(l)|_{\mu(\widetilde
Q)}=pr_1^{\ast}\widetilde L_l
\end{eqnarray}
The central morphism in (\ref{chain}) is induced by the
linearization $\Lambda$ of the sheaf $\SS(l)$ and provides the
required linearization.
\end{proof}

Now consider \cite[ch. 4, sect. 4.2]{HL} an arbitrary
one-parameter subgroup $\lambda: \A^1 \setminus 0 \to GL(V).$ We
denote the image of the point $t\in \A^1 \setminus 0$ under the
morphism $\lambda$ by the symbol $\lambda (t)$. The composite of
the morphism
 $\lambda$ with the action $\alpha$ leads to the morphism
  $\alpha(\lambda): \A^1 \setminus 0 \to \mu(\widetilde
Q)$
 for any closed point $\widetilde x\in \mu(\widetilde Q)$.
 This morphism is given by the correspondence  $t\mapsto \widetilde
x_t=\alpha(\lambda(t),\widetilde x)$. By the properness of the
scheme  $\mu(\widetilde Q)$ the morphism $\alpha(\lambda)$ can be
uniquely continued to the morphism $\overline{\alpha(\lambda)}:
\A^1 \to \mu(\widetilde Q)$. Then the point $\widetilde
x_0=\overline{\alpha(\lambda)}(0)$ is a fixpoint of the action of
the subgroup $\lambda$. Notation: $\widetilde x_0=\lim_{t\to 0}
\lambda(t)(\widetilde x).$ The subgroup  $\lambda$ acts on the
fibre $L_{\widetilde x_0}$ of $G$-linearized vector bundle
 $L$ with some weight $r$. Namely, if $\Lambda$ is the
 linearization on $L$ then  $\Lambda (\widetilde x_0, g)=
g^r \cdot \id_{L_{\widetilde x_0}}.$ Define the {\it weight} of
the corresponding one-dimensional representation of the group
$\lambda$ as $w^{\widetilde L_l}(\widetilde x, \lambda)=-r.$

Analogously for $GL(V)$-action $\beta: Q \times GL(V) \to Q$ upon
the sub\-scheme  $Q\subset \Quot^{P_E(t)}V\otimes \OO_S$ formed by
semistable sheaves, and for the same one-parameter subgroup
$\lambda$ we have $x_0=\overline{\beta(\lambda)}(0)=\lim_{t\to 0}
\lambda(t)(x)$.

The main tool to analyze the existence of a group quotient is
numerical Hilbert -- Mumford criterion. Recall
\begin{definition} \cite[Definition 4.2.9]{HL} The point  $x\in X$
of the scheme $X$ is {\it semistable with respect to
$G$-linearized ample vector bundle $L$} if there exist an integer
$n$ and an invariant global section $s\in H^0(X, L^{n})$ such that
$s(x)\ne 0$. The point $x$ is {\it stable} if in addition the
stabilizer $Stab(x)$ is finite and $G$-orbit of the point $x$ is
closed in the open set of all semistable points in $X$.
\end{definition}
\begin{theorem} {\bf (Hilbert -- Mumford criterion)} \rm{ \cite[Theorem
4.2.11]{HL}} {\it The point  $x\in X$ is semistable if and only if
for all nontrivial one-parameter subgroups $\lambda: \A^1\setminus
0\to G$ there is a following inequality $$ w(x,\lambda)\ge 0.
$$
The point $x$ is stable if and only if for all $\lambda$ strict
inequality holds.}
\end{theorem}

\begin{definition} Let $G$ be an algebraic group, $X,Y$
algebraic schemes, $\mu: X\to Y$ a scheme morphism, $\alpha:
Y\times G \to Y$ action of the group $G$. The morphism $\alpha':
X\times G \to X$ is called the {\it action of the group $G$ upon
the scheme  $X$  induced by the action  $\alpha$ under the
morphism} $\mu$ if the square
\begin{equation*} \xymatrix{X\times G \ar[d]_{(\mu,\id)}
\ar[r]^{\;\;\alpha'}&X
\ar[d]^{\mu}\\
Y\times G \ar[r]^{\;\;\alpha}& Y}
\end{equation*} is cartesian.
\end{definition} Let $G$ be a reductive algebraic group,
 $X,Y,Z$ proper algebraic schemes, $\mu: X\to Y$ and
 $\phi: X \to Z$ scheme morphisms, $\alpha: Y\times G \to Y$
 and $\beta : Z \times G \to Z$ actions of the group $G$. Let also
 $\alpha', \beta': X\times G \to
X$ be actions of the group $G$ on the scheme $X$ induced by
actions  $\alpha$ and $\beta $ respectively.
\begin{definition} Actions  $\alpha$ and $\beta $  are called $X$-{\it
concordant}, if the following diagram commutes
\begin{equation*}\xymatrix{X\times G \ar[d]_{=}
\ar[r]^{\;\;\alpha'}& X\ar[d]^=\\
X\times G \ar[r]^{\;\;\beta'}&X}\end{equation*}
\end{definition}

Consider two morphisms  $\mu(\widetilde Q)
\stackrel{\mu}{\longleftarrow} \widetilde Q
\stackrel{\phi}{\longrightarrow} Q$. By  definitions of actions of
the group $GL(V)$ upon schemes  $\mu(\widetilde Q)$ and $Q$ these
actions are  $\widetilde Q$-concordant.

\begin{definition} Points $y\in Y$ and $z\in Z$ are called {\it corresponding}
with respect to morphisms  $\mu$ and $\phi$ if $\mu^{-1}(y)\cap
\phi^{-1}(z)\ne \emptyset$.
\end{definition}

In the further text the symbol $\lambda (x)$ denotes the orbit of
the point  $x$ under the action of one-parameter subgroup
$\lambda$. The symbol $\lambda (t)(x)$ denotes the point  which
corresponds to the given $t$ in this orbit.

\begin{proposition} \label{fixpoints} Let $y\in Y$ and $z\in Z$ be corresponding points.
For any one-parameter subgroup $\lambda$ fixedness of the point
$y=y_0$ implies existence of the pair $(y_0,z_0)$ of corresponding
points such that the point $z_0$ is fixed, and vice versa.
\end{proposition}

\begin{proof} From the commutativity of the diagram
\begin{equation*}\xymatrix{
y_0 \times \lambda \ar[d] \ar[r]& Y \times \lambda
\ar[d]_{\alpha(\lambda)}&\ar[l] X\times \lambda
\ar[d]_{\alpha'(\lambda)}^{\!\!\!\!=\beta'(\lambda)}
\ar[r]&Z\times
\lambda \ar[d]^{\beta(\lambda)}& \ar[l] \lambda(z)\times \lambda \ar[d]\\
y_0 \ar[r] &Y& \ar[l]_{\mu} X \ar[r]^{\phi}& Z& \ar[l] \lambda(z)
}
\end{equation*}
it follows immediately that if points $y_0$ and $z$ correspond
then for any $t$ points $\lambda(t)(y_0)=y_0$ and $\lambda(t)(z)$
also correspond. Now we note that by the properness of the scheme
 $Z$ there exist a point $z_0=\lim_{t\to 0}\lambda (t)(z).$ Since the subset
$\lambda (z)$ is open in $\overline \lambda(\A^1)$, then we have
for preimages in $X$ that $\phi^{-1}\lambda(z)$ is open in
$\phi^{-1}\overline \lambda(\A^1)$. Also note that the
intersection with the closed subset $\mu^{-1}(y_0)$ yields the
openness of  $\phi^{-1}\lambda(z) \cap \mu^{-1}(y_0)$ in
$\phi^{-1}\overline \lambda(\A^1) \cap \mu^{-1}(y_0)$. Moreover,
for any $t\ne 0$ the intersection $\phi^{-1}\lambda(t)(z) \cap
\mu^{-1}(y_0)$ is nonempty and closed in $\phi^{-1}\overline
\lambda(\A^1) \cap \mu^{-1}(y_0)$. Hence, $\phi^{-1}(z_0)\cap
\mu^{-1}(y_0)$ is nonempty subset as required.
\end{proof}

Let now the schemes  $X,Y,Z$ be projective and morphisms $\mu$ and
$\phi$ be surjective. Let $L_Y$ and $L_Z$ be very ample invertible
$G$-linearized vector bundles on schemes $Y$ and $Z$ respectively.
Let $\Lambda_Y: \alpha ^{\ast} L_Y \to pr_1^{\ast} L_Y$ and
$\Lambda_Z: \beta ^{\ast} L_Z \to pr_1^{\ast}L_Z$ be isomorphisms
of their linearizations. For fibres of bundles $L_Y$ and $L_Z$ at
closed points $y\in Y$ and $z\in Z$ respectively we introduce
notations $L_{Y,y}$ and $L_{Z,z}.$ Then restrictions
$\Lambda_{Y,0}$ and $\Lambda_{Z,0}$ of isomorphisms of
linearizations to actions of one-parameter subgroup $\lambda$ and
to fibres at $\lambda$-fixpoints $y_0$ and $z_0$ has a form:
\begin{equation*} \Lambda_{Y,0}: \alpha(\lambda)^{\ast}L_{Y,y_0}\to pr_1^{\ast}L_{Y,y_0}, \quad
\Lambda_{Z,0}: \beta(\lambda)^{\ast}L_{Z,z_0}\to
pr_1^{\ast}L_{Z,z_0}.
\end{equation*}
\begin{definition} $G$-linearizations $\Lambda_Y$ and $\Lambda_Z$
are {\it fibrewise concordant} if for any two corresponding
$\lambda$-fixpoints  $y_0\in Y$ and $z_0\in Z$ there exists an
isomorphism $f_0:L_{Y,y_0}\to L_{Z,z_0}$ such that the diagram
\begin{equation}\label{char}\xymatrix{\alpha(\lambda)^{\ast}L_{Y,y_0}
\ar[d]_{\alpha(\lambda)^{\ast}f_0}^{\wr}
\ar[rr]^{\Lambda_{Y,y_0}}_{\sim}&&
pr_1^{\ast}L_{Y,y_0} \ar[d]^{pr_1^{\ast}f_0}_{\wr}\\
\beta(\lambda)^{\ast}L_{Z,z_0}\ar[rr]^{\Lambda_{Z,z_0}}_{\sim}&&
pr_1^{\ast}L_{Z,z_0}}
\end{equation}
commutes.
\end{definition}
There is an obvious
\begin{proposition} \label{weight} Fibrewise concordant
linearizations of vector bundles $L_Y$ and $L_Z$ induce for any
one-parameter subgroup $\lambda: \A^1 \setminus 0 \to G$ on fibres
at corresponding fixpoints one-dimensional representations with
equal weights.
\end{proposition}
\begin{proof}
The diagram (\ref{char}) implies the equivalence of
one-dimensional represent\-ations of multiplicative group of the
field $k$. By the multiplicative property of characters
equivalent representations of a
 group have equal characters. Hence the representations
induced by morphisms $\Lambda_{Y,0}$ and $\Lambda_{Z,0}$ have
 equal weights.
\end{proof}

In our situation actions of the group  $GL(V)$ on schemes
$\mu(\widetilde Q)$ and $Q$ are $\widetilde Q$-con\-cordant. The
reasoning in the proof of the proposition \ref{qpro} and, in
particular, the fibred diagram  (\ref{close}) allow to replace the
schemes $\widetilde Q,$ $\mu (\widetilde Q)$, and $Q$ by their
appropriate projective closures. Then setting
$X=\overline{\widetilde Q},$ $Y = \overline{\mu (\widetilde Q)}$,
and $Z= \overline Q$ we apply the proposition \ref{fixpoints}.

\begin{proposition} Bundles $\widetilde L_l$ and bundles
$L_l=p_{1\ast}(\E\otimes \L^l)$ carry concordant linearizations.
\end{proposition}
\begin{proof} Two corresponding points $\widetilde x \in \mu(\widetilde Q)$
and $x\in Q$ define epimorphisms  $V \otimes \OO_{\widetilde S}
\twoheadrightarrow \widetilde E\otimes \widetilde L^m$ and $V
\otimes \OO_S \twoheadrightarrow E\otimes L^m$ respectively. Twist
by  $l$ and formation of $rp_E(l+m)$-th exterior power lead to
epimorphisms
\begin{eqnarray*}
\bigwedge^{rp_E(l+m)}(V \otimes \widetilde L^l)\twoheadrightarrow
\bigwedge^{rp_E(l+m)} \widetilde
E\otimes \widetilde L^{(l+m)},\nonumber\\
\bigwedge^{rp_E(l+m)}(V \otimes L^l)\twoheadrightarrow
\bigwedge^{rp_E(l+m)}  E\otimes  L^{(l+m)} \nonumber
\end{eqnarray*}
respectively. Taking of global sections for $l\gg 0$ before
formation exterior powers gives
\begin{eqnarray} \bigwedge^{rp_E(l+m)}(V \otimes H^0(\widetilde
S,\widetilde L^l))\twoheadrightarrow \bigwedge^{rp_E(l+m)}
H^0(\widetilde S,\widetilde E\otimes \widetilde L^{(l+m)}),\nonumber\\
\bigwedge^{rp_E(l+m)}(V \otimes H^0(S,L^l))\twoheadrightarrow
\bigwedge^{rp_E(l+m)} H^0(S,E\otimes L^{(l+m)}).\label{epi}
\end{eqnarray}
Since projections  $\pi: \Univ^{P(t)} G(V,r) \to
\Hilb^{P(t)}G(V,r)$ and \linebreak $p_1: Q \times S \to Q$ are
flat morphisms, then direct images  \linebreak
$\pi_{\ast}\OO_{\Univ^{P(t)}G(V,r)}(l-m)$ and $p_{1\ast}\L^l$ are
locally free of rank $P(l)$. By Grauert theorem\cite[ch. III,
corollary 12.9]{Hart} one has
$\pi_{\ast}\OO_{\Univ^{P(t)}G(V,r)}(l-m) \otimes k_{\widetilde
x}\cong H^0(\widetilde S, \widetilde L^l)$ and also $
p_{1\ast}\L^l\otimes k_x \cong H^0(S, L^l).$

Choose Zariski-open neighborhoods  $\widetilde U \ni \widetilde x$
and $U \ni x$, providing local trivial\-izations
$\pi_{\ast}\OO_{\Univ^{P(t)}G(V,r)}(l-m)|_{\widetilde U}\cong
H^0(\widetilde S, \widetilde L^l)\otimes \OO_{\pi^{-1}(\widetilde
U)}$ and $p_{1\ast}\L^l|_U \cong$ \linebreak $ H^0(S, L^l) \otimes
\OO_{p_1^{-1}(U)}.$ Without loss of generality we can assume that
open subsets $\widetilde U$ and $U$ contain corresponding points
representing objects $(\widetilde S \cong S,\widetilde E \cong E)$
and $E$ respectively. In this case  $E$ is locally free sheaf.
Then on open subsets  $\widetilde U_0 \cong U_0$ formed by these
points, there is identity isomorphism $H^0(\widetilde S,
\widetilde L^l)\otimes \OO_{\pi^{-1}(\widetilde U_0)}\cong H^0(S,
L^l) \otimes \OO_{p_1^{-1}(U_0)}.$ By triviality of sheaves this
isomorphism can be continued up to the isomorphism
$$H^0(\widetilde S, \widetilde L^l)\otimes
\OO_{\pi^{-1}(\widetilde U)}\cong H^0(S, L^l) \otimes
\OO_{p_1^{-1}(U)}.$$ This fixes the isomorphism on fibres
$H^0(\widetilde S, \widetilde L^l)\otimes k_{\widetilde x} \cong
H^0(S, L^l) \otimes k_x$ at points $\widetilde x $ and $x$. Then
epimorphisms  (\ref{epi}) can be include into the commutative
diagram
\begin{equation}\label{dialin}\xymatrix{\bigwedge^{rp_E(l+m)}(V \otimes H^0(\widetilde
S,\widetilde L^l))\ar[r] \ar[d]_{\wr}& \bigwedge^{rp_E(l+m)}
H^0(\widetilde S,\widetilde E\otimes \widetilde L^{(l+m)}) \ar[d]^{\wr}\\
\bigwedge^{rp_E(l+m)}(V \otimes H^0(S,L^l))\ar[r]
&\bigwedge^{rp_E(l+m)} H^0(S,E\otimes L^{(l+m)})}
\end{equation}
where the right hand side arrow is induced by the isomorphism
$\upsilon$ obtained in the section 4.

Let $((\widetilde S_{\widetilde x_0}, \widetilde L_{\widetilde
x_0}), \widetilde E_{\widetilde x_0})$ be a semistable pair
corresponding to the point $\widetilde x_0 \in \mu(\widetilde Q).$
Also $S_{x_0}$ be the fibre of the family $p_1:Q \times S \to Q$
at the point $x_0 \in Q,$ $L_{x_0}=L$ be its polarization and
$E_{x_0}$ be a semistable coherent sheaf corresponding to the
point $x_0 \in Q$. We introduce shorthand notations: $\widetilde
W=\bigwedge^{rp_E(l+m)}(V\otimes H^0(\widetilde S_{\widetilde
x_0}, \widetilde L_{\widetilde x_0}^l))$ and
$W=\bigwedge^{rp_E(l+m)}(V\otimes H^0( S_{ x_0}, L_{x_0}^l))$.
Notifying that the fibres $(\widetilde L_l)_{\widetilde x_0}$ and
$(L_l)_{x_0}$ of vector bundles $\widetilde L_l$ and $L_l$ at
corresponding points $\widetilde x_0$ and $x_0$ are given by the
isomorphisms \begin{eqnarray}(\widetilde L_l)_{\widetilde
x_0}=\bigwedge ^{rp_E(l+m)}H^0(\widetilde S_{\widetilde x_0},
\widetilde E_{\widetilde x_0}\otimes \widetilde L_{\widetilde
x_0}^{(l+m)}),\nonumber \\
(L_l)_{x_0}=\bigwedge ^{rp_E(l+m)}H^0(S_{x_0}, E_{x_0}\otimes
L_{x_0}^{(l+m)}),\nonumber \end{eqnarray}  using restrictions of
linearizing isomorphisms onto fibres of vector bundles, and
involving the diagram (\ref{dialin}) we get the commutative
diagram
\begin{equation}\label{cube}\xymatrix{
\alpha(\lambda)^{\ast} \widetilde W
\ar[dd] \ar[rd] \ar[rr]^{\sim}  &&\beta(\lambda)^{\ast}W \ar[dd] \ar[rd]&\\
\; & \alpha(\lambda)^{\ast} (\widetilde L_l)_{\widetilde x_0}
\ar[rr]
\ar[dd]^>>>>>>>>{\widetilde \Lambda_0}&& \beta(\lambda)^{\ast}(L_l)_{x_0} \ar[dd]_{\Lambda_0}\\
pr_1^{\ast} \widetilde W \ar[rr] \ar[rd]&& pr_1^{\ast} W
\ar[rd]&\\
\; &pr_1^{\ast} (\widetilde L_l)_{\widetilde x_0} \ar[rr]^{\sim}
&& pr^{\ast}_1 (L_l)_{x_0} }
\end{equation}
All slanted arrows are epimorphisms and the rest are isomorphisms.
The front side of the diagram \ref{cube} proves the proposition.
\end{proof}
\begin{corollary} Corresponding $\lambda$-fixpoints of schemes $\mu(\widetilde Q)$
and $Q$ carry $\lambda$-actions with equal weights.
\end{corollary}
\begin{proof}
The result follows immediately from the previous proposition and
the proposition \ref{weight}.
\end{proof}
The application of numerical Hilbert -- Mumford criterion yields
in the following corollary.
\begin{corollary} \label{ew} For any pair of corresponding points $(\widetilde x, x),$
with $\widetilde x \in \mu (\widetilde Q)$ and $x\in Q$
$L_l$-(semi)stability of the point $x$ implies  $\widetilde
L_l$-(semi)stability of the point $\widetilde x$, and vice versa.
\end{corollary}

 To proceed further we need the theorem known from geometric invariant
 theory.

\begin{theorem}\cite[theorem 4.2.10]{HL} {Let $G$ be a reductive group
acting on a project\-ive scheme $X$ with a $G$-linearized ample
line bundle $L$. Then there is a projective scheme $Y$ and a
morphism $\pi: X^{ss}(L)\to Y$ such that $\pi$ is a universal good
quotient for the $G$-action. Moreover, there is an open subset
$Y^s \subset Y$ such that $X^s(L)=\pi^{-1}(Y^s)$ and such that
$\pi: X^s \to Y^s$ is a universal geometric quotient.
}\end{theorem}

We apply this theorem in the following situation: $X=\overline{\mu
(\widetilde Q)},$ $G=PGL(V),$ $L= \widetilde L_l$, $l\gg 0$. Since
we do not know if the equality $(\overline {\mu (\widetilde
Q)})^{ss}=\mu (\widetilde Q)$ holds, by the corollary \ref{ew} we
have the following proposition.
\begin{proposition} There is a quasiprojective algebraic scheme $\widetilde M$ with a morphism  $\pi: \mu(\widetilde Q)
\to \widetilde M$, and $\pi$ is a universal good
$PGL(V)$-quotient. The scheme  $\widetilde M$ contains an open
subscheme $\widetilde M^s \supset \widetilde M$ such that the
restriction $\pi|_{\mu (\widetilde Q)^s}: \mu (\widetilde Q)^s \to
\widetilde M$ is a universal geometric quotient.
\end{proposition}
\begin{remark} Let $Q_0\subset Q$ be an open subset of points
corresponding to locally free quotient sheaves, $\widetilde
Q_0\cong Q_0$ its image under standard resolution, $\mu(\widetilde
Q_0)$ cor\-responding subset in the scheme $\Hilb^{P(t)}G(V,r).$
Since the morphism  $\mu$ takes distinct classes of isomorphic
pairs $((\widetilde S, \widetilde L),\widetilde E)$ to distinct
classes of isomorphic subschemes in the Grassmannian and since
 $GL(V)$-actions on schemes $Q$ and $\mu(\widetilde Q)$ are concordant,
we have an isomorphism of good GIT-quotients $\widetilde
M_0:=\mu(\widetilde Q_0)/GL(V)\cong $ $Q_0/GL(V)=:M_0.$
\end{remark}
\begin{remark} Since the scheme  $\mu(\widetilde Q)$ is reduced
then its quotient \linebreak $\mu(\widetilde Q)/PGL(V)=\widetilde
M$ is reduced scheme (\cite[ch.0, \S 2, (2)]{MumFo}).
\end{remark}
\begin{remark} In our case  $char (k)=0$ and by
\cite[ch.1, \S  2, theorem 1.1]{MumFo}, the scheme $\widetilde M$
is Noetherian algebraic scheme because $\mu(\widetilde Q)$ is
Noetherian algebraic scheme.
\end{remark}

\section{Morphisms of compactifications and projectivity of $\widetilde M.$}

Recall that the subset  $Q$ formed in $\Quot^{rp(t)}(V\otimes
L^{(-m)})$ by semistable coherent sheaves, is a quasiprojective
algebraic scheme. The morphism of standard resolution $\phi:
\widetilde Q \to Q$ is a projective morphism of algebraic schemes.
This implies that $\widetilde Q$ is quasiprojective algebraic
scheme. By the construction, it is supplied with flat family of
$(S, L, r, rp_E(m))$-admissible schemes $\Sigma_{\widetilde Q}$,
with locally free sheaf $\widetilde \E_{\widetilde Q}.$ This sheaf
induces a mapping  $\Sigma_{\widetilde Q} \to G(r,V)$ which
becomes a closed immersion when restricted to fibres of the family
$\Sigma_{\widetilde Q}$. There is a morphism of the base
$\widetilde Q$ of the family $\Sigma_{\widetilde Q}$ into the
Hilbert scheme of subschemes in the Grassmannian $\mu: \widetilde
Q \to \Hilb^{P(t)}G(r,V).$ We denote by the symbol $\overline M$
the union of components of the Gieseker -- Maruyama scheme that
contain locally free sheaves.

Note that there is (set-theoretical) surjective mapping $\kappa:
\overline M' \to \widetilde M$ given by the formula  $E \mapsto
(\widetilde S, \sigma^{\ast} E /tors)$. The schemes $\overline M$
and $\widetilde M$ contain open subschemes  $M^s_0\subset
\overline M$ and $\widetilde M^s_0 \subset \widetilde M$. The
restriction  $\kappa_0:=\kappa |_{M^s_0}: M^s_0 \to \widetilde
M^s_0$ defines scheme isomorphism.

\begin{proposition} There is a birational morphism of Noetherian
schemes $\kappa: \overline M \to \widetilde M.$
\end{proposition}
\begin{proof}
Consider a product  $\overline M \times \widetilde M$ with
projections $\overline M \stackrel{\overline p}{\longleftarrow}
\overline M \times \widetilde M \stackrel{\widetilde
p}{\longrightarrow} \widetilde M$ and a subset $A:=\{(x, \kappa
(x))\in \overline M \times \widetilde M| x \in \overline M \}$.
Also take a subset $A_0=\{(x, \kappa (x))\in \overline M \times
\widetilde M| x \in \overline M^s_0\}=A\cap M^s_0 \times
\widetilde M^s_0$ corresponding to GIT-stable $S$-pairs. The
inclusion $A_0 \hookrightarrow \overline M \times \widetilde M$
supplies the subset $A_0$ with structure of locally closed
subscheme in the product $\overline M \times \widetilde M.$ Form a
closure $\overline A_0$ of subscheme $A_0$ in the product
$\overline M \times \widetilde M$.

Note that there is an inclusion of sets $A\subset \overline A_0.$
The image of the subset  $A$ coincides with the set $\overline
A_0$. This follows immediately from the standard resolution of a
family of semistable coherent sheaves with a base $\A^1$. The
generic fibre of the family defines the point in $A_0$, and
special fibre of the family defines the point in $\overline A_0
\backslash A_0$. In this case the special fibre corresponds to the
point of the subset $A \backslash A_0.$ Considering different
immersions  $\A^1 \hookrightarrow \overline M'$ we get a bijection
$\overline A_0\simeq A$. Then the subset $A$ is supplied with a
structure of a closed subscheme in the product $\overline M \times
\widetilde M$. By the construction of the subset $A$ we have
$\overline p (A)=\overline M$. By the construction of scheme
$\widetilde M$ also $\widetilde p (A)=\widetilde M.$

Morphisms $\overline \kappa$ and $\widetilde \kappa$ are defined
as composite maps due to commutative diagram
\begin{equation*}\xymatrix{& A
\ar[ld]_{\overline \kappa} \ar@{^(->}[d] \ar[rd]^{\widetilde \kappa}\\
\overline M&\ar[l]_{\overline p} \overline M \times \widetilde M
\ar[r]^{\;\;\widetilde p}& \widetilde M }
\end{equation*}
By the construction of the morphism $\kappa$, morphisms $\overline
\kappa$ and $\widetilde \kappa$ are surjective and birational.
Besides, for any closed point $x\in \overline M$ the
correspondence $x \mapsto (x, \kappa(x))$ defines set-theoretical
map $\overline M \to A$. This map is an inverse for the morphism
$\overline \kappa$ if this morphism is considered as a map of
sets. Then the morphism  $\overline \kappa$ is birational and
bijective on every component of the scheme  $A$. Hence $\overline
\kappa: A \to \overline M$ is an isomorphism. Redenoting the
composite as $\kappa: \overline M \stackrel{\overline \kappa
^{-1}}{\longrightarrow} A \stackrel{\widetilde
\kappa}{\longrightarrow} \widetilde M$ we get the required
morphism of schemes.
\end{proof}
\begin{proposition}\label{pro} $\widetilde M$ is a projective scheme.
\end{proposition}
The proof of this proposition is based on the simple lemma.
\begin{lemma} Quasiprojective complete scheme is projective.
\end{lemma}
\begin{proof} Let $X$ be a quasiprojective complete scheme.
Since $X$ is quasiprojective, then there is an appropriate
projective space $\P$ and an immersion $X \hookrightarrow \P$.
Since  $X$ is complete, then for any scheme  $Y$ the projection
onto the first factor $pr_2: X \times Y \to Y$ takes closed
subschemes to closed subschemes. Set  $Y=\P$ and consider the
diagonal embedding $\Delta:\P \hookrightarrow \P \times \P$. By
the separatedness of the scheme $\P$ this diagonal embedding is
closed. In the commutative diagram with fibred square
\begin{equation*}\xymatrix{ X \ar[r]^{\!\!\!\delta} \ar@{^(->}[d]& X \times \P \ar@{^(->}[d] \ar[r]^{\;\;\;pr_2} & \P\\
\P \ar@{^(->} [r]^{\Delta}& \P \times \P \ar[ur]_{pr_2}}
\end{equation*}
the diagonal immersion $\Delta$ is closed. Hence the morphism
$\delta$ is closed immersion. The image  $pr_2 \circ \delta (X)
\cong X$ is closed in $\P$ by the completeness of the scheme $X$.
Then $X$ is a projective scheme.
\end{proof}
\begin{proof}[Proof of the proposition  \ref{pro}] By the lemma
it is enough to confirm that $\widetilde M$ is a complete scheme.
As proven before, there is a morphism  $\alpha: \overline M \to
\widetilde M,$ where $\overline M$ is a projective scheme. We
prove that for any scheme  $Y$ and for the  projection $pr_2:
\widetilde M \times Y \to Y$ the image $pr_2(Z)$ of any closed
subscheme  $Z \subset \widetilde M \times Y$ is closed in $Y$. Let
$Z'$ be a preimage of the subscheme $Z$ in  $\overline M \times
Y$. Then there is a commutative diagram where the square is fibred
\begin{equation*}\xymatrix{Z' \ar@{^(->}[r] \ar[d]&\overline M \times Y
\ar[d] \ar[r]^{\;\;\;pr_2}& Y\\
Z \ar@{^(->}[r]&\widetilde M \times Y \ar[ur]_{pr_2}}
\end{equation*}
Since  $\overline M$ is complete, then the image
$pr_2(Z')=pr_2(Z)$ is closed in $Y$. This completes the proof.
\end{proof}
\begin{remark} In papers  \cite{Tim0, Tim1, Tim2} we constructed
the compactification of moduli of stable vector bundles which is
called  as {\it constructive compact\-ification} and denoted by
$\widetilde M^c$. It is shown that the constructive
compact\-ification has a birational projective morphism  $\phi^c:
\widetilde M^c \to \overline M$ onto the scheme of Gieseker --
Maruyama. Then the composite of this morphism with the morphism
$\kappa$ yields a birational projective morphism of schemes $\phi:
\widetilde M^c \stackrel{\phi^c}{\rightarrow} \overline M
\stackrel{\kappa}{\rightarrow} \widetilde M.$
\end{remark}

\section{Comparison of equivalences}

The purpose of this section is to examine the relation among
M-equivalence of semistable pairs and GIT-equivalence on the
scheme  $\mu (\widetilde Q).$

We consider the following procedure of процедуру {\it
passing-to-the-limit}. This computation is completely parallel to
that in \cite[lemma 4.4.3]{HL}.

Take a pair $((\widetilde S, \widetilde L), \widetilde E)$ and fix
an epimorphism $h:H^0(\widetilde S, \widetilde E \otimes
\widetilde L) \otimes \widetilde L^{\vee}\twoheadrightarrow
\widetilde E $. If also the isomorphism $H^0(\widetilde S,
\widetilde E \otimes \widetilde L) \stackrel{\sim}{\to} V$ is
fixed then the epimorphism $h$ defines the point $h\in \Hilb
^{P(t)}G(V,r)$ and the point $h\in \Quot^{rp(t)}(V\otimes
\widetilde L^{\vee})$. One-parameter subgroup  $\lambda:
\A^1\setminus 0 \hookrightarrow SL(V)$ is defined completely by
means of the decomposition  $V=\bigoplus_{n\in \Z} V_n$ of vector
space $V$ into the direct sum of weight subspaces $V_n$, $n\in
\Z$, of weight $n$. Namely, for any $v\in V_n$ the action of
elements of the subgroup $\lambda$ is defined by the expression $v
\cdot \lambda (T)=T^n v$. Of course, for almost all $n$ holds
$V_n=0$. Define ascending filtrations for  $V$ and for the sheaf
$\widetilde E$ by the expressions $V_{(n)}=\bigoplus_{\nu \le n}
V_{\nu}, $ $\widetilde E_{(n)}=h(V_{(n)}\otimes \widetilde
L^{\vee}).$ Then the following epimorphisms are defined: $h_n: V_n
\otimes \widetilde L^{\vee} \twoheadrightarrow \widetilde E_n,$
for $\widetilde E_n= \widetilde E_{(n)}/\widetilde E_{(n-1)}$.
Taking the sum over all weights yields in an epimorphism
$\overline h= \bigoplus h_n: V \otimes \widetilde
L^{\vee}\twoheadrightarrow \bigoplus \widetilde E_n =gr(\widetilde
E)=: \overline E.$
\begin{claim} In $\Quot^{rp(t)} (V\otimes
\widetilde L^{\vee})$ one has $\overline h= \lim_{T\to 0} h \cdot
\lambda (T).$
\end{claim}
We construct explicitly the family $\theta: V \otimes \widetilde
L^{\vee} \otimes k[T]\twoheadrightarrow \EE$ parametrized by
affine line  $\A^1=\Spec k[T]$, such that $\theta_0=\overline h$
and $\theta_{\rho}=h \cdot \lambda (\rho)$ for $\rho \ne 0.$ Let
$\EE:= \bigoplus_n \widetilde E_{(n)}\otimes T^n \subset
\widetilde E \otimes_k k[T,T^{-1}]$. Since the direct sum contains
finite collection of nonzero summands, then  $\EE$ is a coherent
sheaf on  $\A^1 \times \widetilde S.$ Indeed, let  $N$ be positive
integer such that $V_n=0$ and $\widetilde E_n=0$ for  $n\le -N$.
Then $\EE \subset \widetilde E \otimes T^{-N} k[T]$. Similarly,
define a module $\VV:=\bigoplus_n V_{(n)} \otimes \widetilde
L^{\vee} \otimes T^n \subset V \otimes_k \widetilde L^{\vee}
\otimes _k k[T, T^{-1}].$ It is clear that the epimorphism  $h$
induces a surjection  $h':\VV \twoheadrightarrow \EE$ of
$\A^1$-flat coherent $\OO_{\A^1 \times \widetilde S}$-sheaves.
Finally, define the isomorphism $\gamma :V \otimes _k k[T]
\stackrel{\sim}{\to} \bigoplus_n V_{(n)}\otimes T^n$ by means of
restrictions $\gamma|_{V_{\nu}}= T^{\nu} id_{V_{\nu}}$ for all
$\nu$. Also define the morphism  $\theta $ by the commutative
diagram
\begin{equation*}\xymatrix{
\bigoplus_n \widetilde E_{(n)} \otimes T^n \ar@{=}[r]& \EE
\ar@{^(->}[r]&
\widetilde E \otimes T^{-N}k[T]\\
V\otimes \widetilde L^{\vee} \otimes k[T] \ar[u]^{\theta}
\ar[r]^>>>>>>>{\gamma}_>>>>>>>{\sim}& \VV \ar[u]^{h'}
\ar@{^(->}[r]& V \otimes_k \widetilde L^{\vee} \otimes _k
T^{-N}k[T] \ar[u]_{h\otimes id}}
\end{equation*}
Restriction to the fibre corresponding to  $T=0$ leads to
$$\EE/T\EE=\bigoplus_n \widetilde E_n $$
and hence  $\theta_0=\bigoplus_n h_n.$

Restriction to the open complement $\A^1 \setminus 0$ corresponds
to the invert\-ibility of the element $T$. Hence tensoring by
$\otimes_{k[T]} k[T,T^{-1}]$ we get a commutative diagram
\begin{equation*}\xymatrix{\EE \otimes_{k[T]}k[T, T^{-1}]
\ar[r]^{\sim}& \widetilde E
\otimes_k k[T, T^{-1}]\\
V\otimes_k \widetilde L^{\vee} \otimes _k k[T,T^{-1}]
\ar[u]^{\theta} \ar[r]^{\gamma}& V\otimes _k \widetilde L^{\vee}
\otimes_k k[T,T^{-1}] \ar[u]_{h\otimes id}}
\end{equation*}
The map  $\gamma$ defines action of one-parameter subgroup
$\lambda.$ Then, $\theta$ is the required mapping.

\begin{definition} Let $G$ be an algebraic group and $f: Y \to X$
a GIT-quotient. Closed points  $y_1$ and $y_2$ of the  scheme  $Y$
are {\it GIT-equivalent} if $f(y_1)=f(y_2).$
\end{definition}

\begin{proposition} \label{eqeq} M-equivalence implies
GIT-equivalence, and vice versa.
\end{proposition}
\begin{proof} Consider two  M-equivalent semistable pairs
$((\widetilde S, \widetilde L), \widetilde E)$ and
\linebreak$((\widetilde S_{gr}, \widetilde L_{gr}), \widetilde
E_{gr}).$ Each fibre of the morphism of formation of GIT-quotient
of $\mu(\widetilde Q)$ contains one closed orbit. Indeed, it is
known \cite[Theorem 4.3.3]{HL} that points representing polystable
coherent sheaves and only these points have closed orbits in
$\Quot$. By the concordance of $PGL(V)$-actions, the pairs of the
form $((\widetilde S_{gr}, \widetilde L_{gr}), \widetilde E_{gr})$
and only these pairs have closed orbits in Hilbert scheme.

It is enough to prove the proposition for such points that one of
them has the form  $((\widetilde S_{gr}, \widetilde L_{gr}),
\widetilde E_{gr})$.

Consider an epimorphism
\begin{equation}\label{ep1} V \otimes \widetilde
L^{\vee}\twoheadrightarrow \widetilde E. \end{equation} Since
$K^{dS}=\mu(\widetilde Q),$ then the pair $((\widetilde S,
\widetilde L), \widetilde E)$ has a preimage in $Q$. Let this is
an epimorphism $V \otimes L^{\vee}\twoheadrightarrow E.$ The only
nontrivial situation is that when $E$ is strictly semistable.
Consider its S-equivalence class $[E]$. Let $\overline \sigma:
\overline S_{\ast} \to S$ be the minimal resolution in the monoid
$\diamondsuit [E]$, $\overline L$ be a very ample invertible sheaf
on the scheme  $\overline S_{\ast}$. The morphism of the minimal
resolution includes into the commutative diagram
\begin{equation*}\xymatrix{\overline S_{\ast} \ar[d]_{\sigma'_{gr}} \ar[rd]_{\overline \sigma }
\ar[r]^{\sigma'}& \widetilde S_{gr} \ar[d]^{\sigma_{gr}}\\
\widetilde S \ar[r]_{\sigma}&S}
\end{equation*}
Then the epimorphism  (\ref{ep1}) induces the epimorphism
\begin{equation}\label{ep2} V \otimes \sigma'^{\ast}_{gr}\widetilde
L^{\vee}\twoheadrightarrow \sigma'^{\ast}_{gr}\widetilde E.
\end{equation}
Let $r\overline p(t)= \chi(\sigma'^{\ast}_{gr}\widetilde E \otimes
\overline L^t)).$ By the results of section 7, if $F_i$ are
subsheaves in Jordan -- H\"{o}lder filtration for the sheaf $E$
then quotient sheaves  $\overline F_i / \overline
F_{i-1}=\overline \sigma^{\ast}
(F_i/F_{i-1})/tors=\sigma'^{\ast}_{gr}(\widetilde F_i/\widetilde
F_{i-1})/tors$ are locally free.

Consider the scheme of quotients $\Quot ^{r\overline
p(t)}(V\otimes \overline L^{\vee}).$  In this scheme the
passing-to-\linebreak the-limit process in the family of locally
free sheaves with general sheaf of the form (\ref{ep2}) is also
considered. The the limit object has the form $\bigoplus_i
\overline F_i /\overline F_{i-1}=\bigoplus_i \overline \sigma
(F_i/F_{i-1})/tors= \sigma'^{\ast}_{gr} gr(\widetilde E)/tors
=\sigma'^{\ast}\widetilde E_{gr}.$ The family of $\OO_{\overline
S_{\ast}}$-sheaves $\EE$ obtained in the passing-to-the-limit
process, induces the morphism of its base into Hilbert scheme
$\Hilb^{P(t)}G(V,r).$ Indeed, formation of direct images for
sheaves on fibres at general points $T \ne 0$ under the morphism
$\sigma'_{gr}$ leads to $\OO_{\widetilde S}$-sheaves. These
sheaves are isomorphic to $\widetilde E$. Analogously, the direct
image of the sheaf on the fibre at the point $T=0$ under the
morphism $\sigma'$ leads to $\OO_{\widetilde S_{gr}}$-sheaf
$\widetilde E_{gr}.$

The reasoning done shows that the orbit of the point representing
the object  $((\widetilde S_{gr}, \widetilde L_{gr}), \widetilde
E_{gr})$, belongs to the closure of the orbit of M-equivalent
point $((\widetilde S, \widetilde L), \widetilde E)$. This implies
that GIT-equivalence is equi\-valent to M-equivalence.
\end{proof}

\section{$\widetilde M$ as moduli space} In this section we
prove that the constructed scheme $\widetilde M$ is a coarse
moduli space for the functor
$$ {\mathfrak f}:(RSchemes_k) \to (Sets)
$$ in the theorem \ref{th}.

It is enough to confirm that the functor $\mathfrak f$ is
corepresented by the scheme $\widetilde M$. Choose an object
$((\widetilde S, \widetilde L), \widetilde E)$. Note that the
sheaf $\widetilde E \otimes \widetilde L$ defines the immersion
$j: \widetilde S \hookrightarrow G(V,r)$. This immersion is
defined not in unique way but up to the class of the isomorphism
$H^0(\widetilde S, \widetilde E \otimes \widetilde L)
\stackrel{\sim} {\longrightarrow} V$ modulo multiplication by
nonzero scalars $\vartheta \in k^{\ast}.$ Hence the point
corresponding to the subscheme $j(\widetilde S) \subset G(V,r),$
is defined in the Hilbert scheme $\Hilb^{P(t)}G(V,r)$ up to the
action of the group $PGL(V).$ Then the object $((\widetilde S,
\widetilde L), \widetilde E)$ defines the morphism  $h\in \Hom
(\Spec k, \widetilde M).$

Inversely, by the proposition \ref{eqeq}, the morphism  $h\in \Hom
(\Spec k, \widetilde M)$ distinguishes a point representing the
M-equivalence class of object  $((\widetilde S, \widetilde L),
\widetilde E)$.

We construct for any scheme
 $B$ and for natural transformation  $\psi': \mathfrak f
\to \underline F'$ a unique natural transformation $\omega:
\underline {\widetilde M}\to \underline F'$ such that
$\psi'=\omega \circ \psi.$

Let the transformation  $\alpha$ correspond to the  flat family
$\pi: \widetilde \Sigma \to B$ with fibrewise polarization
$\widetilde \L$, supplied with the family of locally free sheaves
$\widetilde \E$.

In this case the restriction onto any fibre $\pi^{-1}(y)$ of the
morphism $\pi$ provides an object $((\pi^{-1}(b), \widetilde
\L|_{\pi^{-1}(b)}), \widetilde \E|_{\pi^{-1}(b)})$. This object
belongs to the class  $\mathfrak F$. Then there is a morphism
$\widetilde \Sigma \to G(\pi_{\ast} (\widetilde \E \otimes
\widetilde \L) , r)$ such that the triangle
\begin{equation*}\xymatrix{G(\pi_{\ast} (\widetilde \E \otimes
\widetilde \L) , r)\ar[d]& \ar[l] \widetilde
\Sigma \ar[dl]_{\pi}\\
B}
\end{equation*}
commutes. The sheaf $\pi_{\ast} (\widetilde \E \otimes \widetilde
\L)$ is locally free, hence the Grassmannian bundle $G(\pi_{\ast}
(\widetilde \E \otimes \widetilde \L) , r)$ is locally trivial
over $B$. Let  $\bigcup_i B_i = B$ be the trivializing open cover.
Subfamilies $\widetilde \Sigma_i$ are defined as fibred products
\begin{equation*} \xymatrix{\widetilde \Sigma \ar[d]_{\pi} & \ar[l] \widetilde \Sigma_i \ar[d]^{\pi_i}\\
B & \ar[l] B_i}
\end{equation*}
The horizontal arrows are open immersions. Fix isomorphisms of
trivial\-i\-zations $\tau_i:G(\pi_{\ast} (\widetilde \E \otimes
\widetilde \L) , r)|_{B_i}\to G(V, r) \times B_i.$ The composite
map  $\widetilde \Sigma_i \stackrel{\!j_i}{\to} G(\pi_{\ast}
(\widetilde \E \otimes \widetilde \L) ,
r)|_{B_i}\stackrel{\!\!\tau_i}{\to} G(V, r) \times B_i
\stackrel{pr_1}{\to} G(V,r)$ provides a morphism of the base to
the Hilbert scheme $\mu_i: B_i \to \Hilb^{P(t)}G(V,r).$ This
morphism is defined up to $PGL(V)$-action. Elements of $PGL(V)$
define gluing the elements of the trivializing cover. Then the
formation of GIT-quotient leads to the morphism $B \to \widetilde
M$. Its dual in the opposite category $(Scheme\!s_k)^o$ defines
the natural transformation $\omega:\underline{\widetilde M} \to
\underline F'.$

{\it Acknowledgments.} The author expresses deep and sincere
gratitude to D.~Orlov in Steklov Mathematical Institute (Moscow,
Russia), M.~Reid in the Institute of Mathematical Research of
Warwick university (Great Britain), N.~Shefferd-Barron in
Cambridge university (Great Britain),  A.~Langer in Warsaw
university (Poland), for fruitful discussions and interest to this
work.

\end{document}